\documentclass{amsart}

\usepackage[a4paper,outer=3.4cm,inner=3.4cm,total={14cm,21.4cm}]{geometry}
\usepackage[english]{babel}
\usepackage[utf8]{inputenc}
\usepackage{amsmath,amssymb,amsthm}
\usepackage{enumerate}
\usepackage{mathtools} %provides dcases
\usepackage{stackengine}
\usepackage{adjustbox}
\usepackage{array}
\usepackage{verbatim}
\usepackage[shortlabels]{enumitem}
\setlist{leftmargin=10mm}
\usepackage{stmaryrd} % provides \sslash
\usepackage{calligra,mathrsfs} % for typeseeting inner Hom, from https://tex.stackexchange.com/questions/141434/how-to-type-sheaf-hom/156701

\usepackage[colorlinks=true]{hyperref}
\hypersetup{colorlinks, citecolor=blue, filecolor=black, linkcolor=black, urlcolor=black}

\usepackage{tikz}
\usetikzlibrary{decorations.pathmorphing}
\usepackage{tikz-cd}

\usepackage[capitalize]{cleveref}
\crefformat{equation}{\normalfont(#2#1#3)}

\newtheorem{thm}{Theorem}[section]
\newtheorem{Theorem}[thm]{Theorem}
\newtheorem{Lemma}[thm]{Lemma}
\newtheorem{Proposition}[thm]{Proposition}
\newtheorem{Corollary}[thm]{Corollary}
\newtheorem{Conjecture}[thm]{Conjecture}
\newtheorem{Choices}[thm]{Choices}
\newtheorem{Claim}[thm]{Claim}

\theoremstyle{definition}
\newtheorem{Definition}[thm]{Definition}
\newtheorem{Question}[thm]{Question}
\newtheorem{Example}[thm]{Example}
\newtheorem{Remark}[thm]{Remark}

\newlist{propenum}{enumerate}{1}
\setlist[propenum]{label=(\roman*), ref=\theProposition.(\roman*)}
\crefalias{propenumi}{Proposition}

\newlist{defenum}{enumerate}{1}
\setlist[defenum]{label=(\alph*), ref=\theDefinition.(\alph*)}
\crefalias{defenumi}{Definition}

\newcommand*\isomarrow{% arrow with tilde as used for isomorphisms
	\xrightarrow{\raisebox{-0.35em}{\smash{\ensuremath{\sim}}}}
}  

\newcommand{\wh}{\widehat}
\newcommand{\A}{\mathbb{A}}
\newcommand{\B}{\mathbb{B}}
\newcommand{\C}{\mathbb{C}}
\newcommand{\G}{\mathbb{G}}
\newcommand{\N}{\mathbb{N}}
\renewcommand{\O}{\mathcal{O}}
\newcommand{\Q}{\mathbb{Q}}
\newcommand{\W}{\mathcal{W}}
\newcommand{\Z}{\mathbb{Z}}
\newcommand{\Gal}{\operatorname{Gal}}
\renewcommand{\H}{\mathcal{H}}
\newcommand{\m}{\mathfrak{m}}
\newcommand{\mg}{\mathfrak{g}}

\newcommand{\im}{\operatorname{im}}
\newcommand{\Hom}{\operatorname{Hom}}
\newcommand{\Map}{\operatorname{Map}}
\newcommand{\Mor}{\operatorname{Mor}}
\newcommand{\End}{\mathrm{End}}
\newcommand{\Aut}{\mathrm{Aut}}
\newcommand{\GL}{\mathrm{GL}}
\newcommand{\Sym}{\operatorname{Sym}}
\newcommand{\Perf}{\operatorname{Perf}}
\newcommand{\Spec}{\operatorname{Spec}}
\newcommand{\Spa}{\operatorname{Spa}}
\newcommand{\Spd}{\operatorname{Spd}}
\newcommand{\id}{{\operatorname{id}}}
\newcommand{\cts}{{\operatorname{cts}}}
\newcommand{\an}{{\mathrm{an}}}

\newcommand{\cyc}{{\operatorname{cyc}}}
\newcommand{\hotimes}{\hat{\otimes}}
\newcommand{\et}{{\operatorname{\acute{e}t}}}
\newcommand{\proet}{{\operatorname{pro\acute{e}t}}}
\newcommand{\qproet}{{\operatorname{qpro\acute{e}t}}}
\newcommand{\Smd}{\mathrm{Smd}}
\newcommand{\HT}{\operatorname{HT}}
\newcommand{\HTlog}{\operatorname{HTlog}}
\newcommand{\res}{{\mathrm{res}}}
\newcommand{\wt}{\widetilde}
\newcommand{\wtOm}{\wt\Omega}
\newcommand{\rk}{\operatorname{rk}}
\newcommand{\ab}{\mathrm{ab}}
\renewcommand{\lim}{\varprojlim}
\newcommand{\cH}{{\ifmmode \check{H}\else{\v{C}ech}\fi}}
\newcommand{\tf}{[\tfrac{1}{p}]}
\renewcommand{\tt}{\mathrm{tt}}

\newcommand{\aeq}{\stackrel{a}{=}}
\newcommand{\ad}{\mathrm{ad}}
\newcommand{\tr}{\operatorname{tr}}
\newcommand{\cB}{\mathcal A}
\newcommand{\cBun}{\mathscr{B}\hspace{-0.1em}\mathrm{\textit{un}}}

\newcommand{\cHiggs}{\mathscr H\hspace{-0.1em}\mathrm{\textit{iggs}}}
\newcommand{\Higgs}{\mathrm{Higgs}}
\newcommand{\wtcH}{\widetilde{\mathcal H}}

\DeclareMathOperator{\HOM}{\mathscr{H}\hspace{-0.3em}\mathrm{\textit{om}}}

\begin{document}
	\title{Moduli spaces in $p$-adic non-abelian Hodge theory}
	\author{Ben Heuer}
	\date{}
	\maketitle
	\begin{abstract}
		We propose a new moduli-theoretic approach to the $p$-adic Simpson correspondence for a smooth proper rigid space $X$ over $\C_p$ with coefficients in any rigid analytic group $G$, in terms of a comparison of moduli stacks. For its formulation, we introduce the class of ``smoothoid spaces'' which are perfectoid families of smooth rigid spaces, well-suited for studying relative $p$-adic Hodge theory. For any smoothoid space $Y$, we then construct a ``sheafified non-abelian Hodge correspondence'', namely a canonical isomorphism
		\[R^1\nu_{\ast}G\isomarrow \Higgs_G\]
		where $\nu:Y_{v}\to Y_{\et}$ is the natural morphism of sites, and where $\Higgs_G$ is the sheaf of isomorphism classes of $G$-Higgs bundles on $Y_{\et}$. We also prove a generalisation of Faltings' local $p$-adic Simpson correspondence to $G$-bundles and to perfectoid families.
		
		We apply these results to deduce $v$-descent criteria for \'etale $G$-bundles which show that \mbox{$G$-Higgs} bundles on $X$ form a small \mbox{$v$-stack} $\cHiggs_G$. As a second application, we construct an analogue of the Hitchin morphism on the Betti side: a morphism 
		$\cBun_{G,v}\to \mathcal A_G$
		from the small $v$-stack of $v$-topological \mbox{$G$-bundles} on $X$ to the Hitchin base. This allows us to give a conjectural reformulation of the \mbox{$p$-adic} Simpson correspondence for $X$ in a more geometric and more canonical way, namely in terms of a comparison of Hitchin morphisms.
	\end{abstract}
	\setcounter{tocdepth}{2}
	%\tableofcontents
	
	%------------------------------------------------------------------------------

	\section{Introduction}
	Let $K$ be a complete algebraically closed extension of $\Q_p$. The conjectural {$p$-adic} Corlette--Simpson correspondence for a connected smooth proper rigid space $X$ over $K$ aims to relate $K$-linear continuous representations of the \'etale fundamental group $\pi_1^{\et}(X)$ of $X$ to a full subcategory of the Higgs bundles on $X$. By a reinterpretation of Faltings' pioneering work on generalised representations \cite{Faltings_SimpsonI} in terms of Scholze's diamonds \cite{etale-cohomology-of-diamonds}, this should go through a $p$-adic non-abelian Hodge correspondence, namely an equivalence of categories
\begin{equation}\label{eq:intro-naHT-correspondence}
	\{\text{vector bundles on $X_{v}$}\}\isomarrow \{\text{Higgs bundles on $X_{\et}$}\},
\end{equation}
via a natural fully faithful embedding of representations of $\pi_1^{\et}(X)$ into vector bundles on $X_v$.

The aim of this article is to lay the foundations for a new conceptual approach to $p$-adic non-abelian Hodge theory, in particular to \Cref{eq:intro-naHT-correspondence}, which puts moduli spaces at its centre: Indeed, in \cite{heuer-proper-correspondence},
	we will construct the equivalence  \Cref{eq:intro-naHT-correspondence} based on ideas of this article. Second, we suggest a new conjectural geometric formulation of the equivalence \Cref{eq:intro-naHT-correspondence} in terms of a comparison of moduli stacks, which will be proved in \cite{HX} in the case when $X$ is a curve. Third, this makes it possible to explain the relation to representations of $\pi^{\et}_1(X)$ in terms of a $p$-adic character variety, which opens up new geometric ways to study the essential image of representations of $\pi^{\et}_1(X)$  under \Cref{eq:intro-naHT-correspondence}, generalising the approach for line bundles from \cite{heuer-geometric-Simpson-Pic}.
	Fourth, we develop these foundations not just for $\GL_n$, but for general rigid groups $G$: This opens up a new line of investigation into generalisations of \Cref{eq:intro-naHT-correspondence} to $G$-torsors, which we continue in \cite{HMZ}.
	
	Besides, the technical foundations of $p$-adic Hodge theory of perfectoid families of rigid spaces that we provide in this article are of broader interest to relative $p$-adic Hodge theory: For example, in \cite{heuer-relative-HT}, we use them to construct the relative Hodge--Tate spectral sequence.
	
	\begin{comment}
	This article is the first in a series on a new conceptual approach to  $p$-adic non-abelian Hodge theory in terms of $p$-adic analytic moduli spaces: We begin by suggesting a new $v$-cohomological approach to the local aspects of the $p$-adic Simpson correspondence that is a non-abelian generalisation of a key result in Scholze's strategy for the $p$-adic Hodge--Tate comparison. Our results also provide a generalisation of Faltings' ``local correspondence'' allowing very general non-abelian coefficients, namely $G$-torsors under any rigid group $G$.
	
	 Second, we introduce a new class of ``smoothoid spaces'' that allows us to formulate relative $p$-adic Hodge theory of rigid spaces in perfectoid families, and show that the local correspondence extends to this class. Third, we then use smoothoid spaces to construct analytic moduli space on either side of the correspondence in terms of small $v$-stacks. Fourth, we use these to construct a new period morphism in non-abelian Hodge theory: an analogue of the Hitchin fibration on the Betti side. This associates to any representation of $\pi_1^{\et}(X)$ a geometric analogue of Hodge--Tate--Sen weights. We explain how this gives a new geometric formulation of the conjectural $p$-adic Simpson correspondence: Namely, we provide evidence that the moduli stack of $v$-vector bundles is a twist of the moduli stack of Higgs bundles.
	\end{comment}
	\subsection{The sheafified non-abelian Hodge correspondence}
	
	In the spirit of Simpson's non-abelian Hodge theory \cite{SimpsonCorrespondence}, the starting point of this article is the idea to realise the $p$-adic non-abelian Hodge correspondence \Cref{eq:intro-naHT-correspondence} by finding a non-abelian generalisation of Scholze's approach to the Hodge--Tate spectral sequence:
	
	To explain this, let us begin by considering the $p$-adic Hodge--Tate short exact sequence
	\begin{equation}\label{eq:HT-SES}
	0\to H^1_{\an}(X,\O)\to \Hom_{\cts}(\pi_1^{\et}(X),K)\to H^0(X,\Omega^1_X(-1))\to 0.
	\end{equation}
	The fundamental idea of Scholze's construction of \cref{eq:HT-SES} is to realise it as a Leray sequence for the morphism of sites $\nu:X_v\to X_{\et}$, where $X_v$ is the $v$-site of the diamond associated to $X$. 
	The key result that leads to \cref{eq:HT-SES} is then the natural isomorphism \cite[Proposition 3.23]{Scholze2012Survey}
	\begin{equation}\label{eq:HT}
	\HT:R^1\nu_{\ast}\O\isomarrow\Omega^1_X(-1).
	\end{equation}
	In this article, we give an analogue of \cref{eq:HT-SES}  in which the middle term is replaced by continuous representations $\pi_1^{\et}(X)\to \GL_n(K)$,  following Scholze's strategy with $\O$ replaced by $\GL_n(\O)$.
	
	\medskip
	
	In a generalisation that so far has not been studied in the $p$-adic setting, we more generally pass from $\GL_n$ to  any rigid analytic group variety $G$ over $K$. Such $G$ are the $p$-adic analogues of complex Lie groups, and we simply refer to them as ``rigid groups''. For example, $G$ could be the analytification of an algebraic group, or any  open subgroup thereof. In order to describe the correct replacement for the right hand side of \cref{eq:HT}, we need the following:
	
	\begin{Definition}\label{d:intro-G-Higgs}
		A $G$-Higgs bundle on $X$ is a pair $(E,\theta)$ of a $G$-torsor $E$ on $X_{\et}$ and a section $\theta\in H^0(X,\ad(E)\otimes \Omega^1_X(-1))$ such that $\theta\wedge \theta=0$, where $\ad(E)$ is the adjoint bundle.
	\end{Definition}

	Towards a non-abelian analogue of Scholze's construction, our first main result is now the following {``sheafified $p$-adic non-abelian Hodge correspondence''} for rigid analytic spaces:
	
	\begin{Theorem}\label{intro-first-version-of-main-Thm}
		Let $X$ be any smooth rigid space over $K$ and let $\nu:X_{v}\to X_{\et}$ be the natural morphism of sites. Let $G$ be a rigid group over $K$, considered as a sheaf on $X_{v}$. For example, $G$ could be the analytification of any algebraic group. Then there is a canonical isomorphism 
		\[ \HTlog: R^1\nu_{\ast}G\isomarrow \Higgs_G\]
		of sheaves of pointed sets on  $X_{\et}$ that is functorial in $X$, $G$ and $K$. Here
		\[ \Higgs_G:=(\mathrm{Lie}(G)\otimes_K \Omega_X^1(-1))^{\wedge =0}/ G\]
		is the sheaf of isomorphism classes of $G$-Higgs bundles on $X_{\et}$:
		Explicitly, $\mathrm{Lie}(G)$ is the Lie algebra of $G$ as a $K$-vector space, $(-)^{\wedge=0}$ denotes the subspace of elements $\theta$ satisfying $\theta\wedge \theta=0$, and the sheaf quotient is formed with respect to the adjoint action of $G$ on $\mathrm{Lie}(G)$.
	\end{Theorem}
 This is a vast generalisation of Scholze's isomorphism \cref{eq:HT}, which we recover as the special case of $G=\G_a$.
In earlier work, we had studied $G=\G_m$, but already for $G=\GL_n$, the theorem is new.
In other words, $\HTlog$ gives a \textit{canonical} correspondence between isomorphism classes of $v$-topological $G$-bundles on $X$ and $G$-Higgs bundles on $X$ after \'etale sheafification.

As in Scholze's strategy for $\G_a$, the Leray sequence of $\nu$ now yields a short exact sequence 
\begin{equation}\label{eq:HT-seq-G}
0\to H^1_{\et}(X,G)\to H^1_{v}(X,G) \xrightarrow{\HTlog} \Higgs_G(X)
\end{equation}
of pointed sets. Of course, for non-abelian $G$, such a sequence  gives less structure than a short exact sequence of abelian groups as in \eqref{eq:HT-SES}. The second main idea of this article is therefore to pass from sheaves of isomorphism classes to moduli spaces, and to turn the last morphism in \eqref{eq:HT-seq-G} into a morphism of $v$-stacks. The exactness of the Hodge--Tate sequence can then find its generalisation in geometric properties of this morphism, such as being a fibration. As we will illustrate below at the hand of examples,  this  strategy indeed turns out to be fruitful.

\subsection{Moduli spaces in $p$-adic non-abelian Hodge theory}
	In order to define $p$-adic analytic moduli spaces of $v$-topological $G$-torsors and $G$-Higgs bundles on $X$, we need new technical foundations  to formulate relative $p$-adic Hodge theory. For this, our first step in this article is to introduce and study a new class of perfectoid families of smooth rigid spaces:
	\begin{Definition}
		A \textbf{smoothoid space} is an analytic adic space over a perfectoid field $K$ that locally admits a smooth morphism of adic spaces to a perfectoid space over $K$.
	\end{Definition}
	An example would be the product $X\times T$ of the smooth rigid space $X$ with a perfectoid space $T$, or any object of its \'etale site. There is a reasonable notion of differentials on a smoothoid space:
	For any smoothoid space $Y$ we consider the map $\nu:Y_{v}\to Y_{\et}$ and define
	\[ \wtOm^n_{Y}:=R^n\nu_{\ast}\O_Y.\]
	We show that this is a vector bundle on $Y$. The cup product induces a natural wedge product on this, yielding a good functorial definition of $G$-Higgs bundles, exactly as in \cref{d:intro-G-Higgs}.
	
	\medskip

	Towards a moduli-theoretic non-abelian Hodge correspondence, we suggest to study the functors fibred in groupoids on the $v$-site $\Perf_K$ of affinoid perfectoid spaces $T$ over $K$
	\begin{alignat*}{3}
	\cBun_{G,v}\colon T&\mapsto &&\{\text{$G$-torsors on }(X\times T)_{v}\},\\
	 \cHiggs_{G}\colon T&\mapsto&& \{\text{$G$-Higgs bundles on }(X\times T)_{\et}\}.
	\end{alignat*}
	Our next result says that these can reasonably be regarded as geometric objects:
	
	\begin{Theorem}\label{t:intro-Bun-Higgs-v-stacks}
		
		$\cBun_{G,v}$ and $\cHiggs_{G}$ are small $v$-stacks in the sense of Scholze \cite[\S12]{etale-cohomology-of-diamonds}.
	\end{Theorem}
	While for  $\cBun_{G,v}$ only the smallness requires work, the result for $\cHiggs_{G}$ is much deeper. We deduce it from $v$-descent criteria, namely our main result is that \cref{intro-first-version-of-main-Thm} generalises:
	
	\begin{Theorem}\label{t:main-thm-2-intro}
		Let $X$ be a smoothoid space over $K$, then there is a canonical isomorphism 
		\[ \HTlog: R^1\nu_{\ast}G\isomarrow (\mathrm{Lie}(G)\otimes_K \wtOm_X^1)^{\wedge =0}/ G\]
		of sheaves of pointed sets on  $X_{\et}$, where $\nu:X_{v}\to X_{\et}$. It is functorial in $X$, $G$ and $K$. 
		For commutative $G$, there is for any $n\in \N$ a canonical isomorphism $R^n\nu_{\ast}G\isomarrow \mathrm{Lie}(G)\otimes_K \wtOm_X^n$.
	\end{Theorem}

	\subsection{The Hitchin morphism on the Betti side}
	The moduli stack $\cHiggs_{G}$ gives rise to a $p$-adic incarnation of the \textbf{Hitchin morphism} which plays an important role in complex non-abelian Hodge theory. In the $p$-adic setting, this is a morphism of $v$-stacks
	\begin{equation}\label{eq:Hitchin-Higgs}
		\mathcal H:\cHiggs_{G}\to  \mathcal A_G
	\end{equation}
	to the Hitchin base $\mathcal A_G$ over $X$. Here $\mathcal A_G$ is in general a $v$-sheaf, but  if $X$ is proper and $G$ is split reductive or commutative, then $\mathcal A_G$ is represented by an affine rigid space, for example
	\[ \mathcal A_{\GL_n}=\textstyle\bigoplus\limits_{k=1}^nH^0(X,\Sym^k\wtOm_{X})\otimes_K \G_a.\]
	The construction of $\mathcal H$ in this case is essentially the same as the classical one due to Hitchin \cite{HitchinFibration}, given by sending a Higgs bundle $(E,\theta)$ to the characteristic polynomial of $\theta$. 
	\medskip
	
	On the ``Betti'' side,	\cref{t:intro-Bun-Higgs-v-stacks} now allows us to construct an analogous  morphism:
	
	\begin{Definition}
	 For any rigid group $G$, the \textbf{Hitchin morphism on the Betti side}
	\begin{equation}\label{eq:Hitchin-Betti}
		\wt{\mathcal H}:\cBun_{G,v}\to \mathcal A_G
	\end{equation}
	is defined by sending a $v$-topological $G$-torsor on $Y=X\times T$ to the associated class in $R^1\nu_{\ast}G(Y)$ and  using $\HTlog$ to pass to the Higgs side, where we use the Hitchin map \cref{eq:Hitchin-Higgs}.
\end{Definition}

Assume now that $X$ is a smooth proper rigid space over $K$ and that $G$ is reductive. In this setting, we envision the $p$-adic Simpson correspondence to have a geometric incarnation in terms of a comparison of the Hitchin morphisms \cref{eq:Hitchin-Higgs} and \cref{eq:Hitchin-Betti}.
While $\cBun_{G,v}$ and $\cHiggs_{G}$ are not in general isomorphic, we conjecture that one is a \textit{twist} of the other in a canonical way, via the Hitchin fibrations.
 Indeed, building on this article, the following is  proved in \cite{HX}:

\begin{Theorem}[\cite{HX}]\label{t:part-II-twist-for-curves}
	Let $X$ be a smooth projective curve over $K$ and let $G$ be a  reductive group. Then there is a Picard groupoid $\mathscr P$ on $\mathcal A_{G,v}$ that acts on both $\mathcal H:\cHiggs_{G}\to \mathcal A_{G}$ and  $	\wt{\mathcal H}:\cBun_{G,v}\to \mathcal A_G$, and a $\mathscr P$-torsor $\mathscr H$ for which there is a canonical equivalence of $v$-stacks
	\[
	\cBun_{G,v}\isomarrow \mathscr H\times^{\mathscr P}\cHiggs_{G}.\]
\end{Theorem}

Here $\mathscr P$ is the stack of line bundles on the spectral curve. In particular, this shows that for $G=\GL_n$, up to connected components, the morphism  $\wt {\mathcal H}:	\cBun_{G,v}\to \mathcal A_{G}$ is in this case a fibration in abelian varieties.
To illustrate this phenomenon further, let us also mention the easier case of $G=\G_m$, for which \cite[Theorem 1.3]{heuer-diamantine-Picard} immediately implies the following:
\begin{Theorem}[\cite{heuer-diamantine-Picard}]\label{t:intro-Pic-twist}
	For any smooth proper rigid space $X$ over $K$, the sequence
	\begin{equation}\label{eq:multiplicative-HT-seq}
		0\to \cBun_{\G_m,\et}\to  \cBun_{\G_m,v}\xrightarrow{\wt{\mathcal H}} H^0(X,\wtOm^1_X)\otimes \G_a\to 0
	\end{equation}
	is short exact.
	In particular, both $\mathcal H$ and $\wtcH$ are torsors under $\cBun_{\G_m,\et}$.
\end{Theorem}

\subsection{Towards a $p$-adic Corlette--Simpson correspondence for $\pi^{\et}_1(X)$}
The moduli-theoretic approach also seems fruitful for studying representations of $\pi^{\et}_1(X)$: We show that the Hitchin morphism on the Betti side $\wtcH$ induces a geometric generalisation of the map 
$\HT:\Hom_{\cts}(\pi_1(X),K)\to H^0(X,\wtOm)$
in the Hodge--Tate sequence \cref{eq:HT-SES}, namely a morphism
\[ \wt {H}:\HOM(\pi_1^{\et}(X),G)\to \mathcal A_{G},\]
 from the representation variety parametrising continuous $G$-representations of $\pi_1^{\et}(X)$ to the Hitchin base. In contrast to 	$\wt{\mathcal H}$, this has the advantage that it is represented by a morphism of rigid spaces. We call this the ``\textbf{Hitchin--Hodge--Tate morphism}''. The name reflects that $\wt H$ is simultaneously an analogue of the Hitchin morphism and a generalisation of the Hodge--Tate map $\HT$. Indeed, for $G=\G_a$, it is the morphism of rigid spaces associated to $\HT$ by tensoring with $\G_a$.  As we will explain, the image of a representation $\pi^{\et}_1(X)\to \GL_n(K)$ under $\wt H$ is a close analogue of the Hodge--Tate--Sen weights of a local Galois representation.
 
 \medskip
 
 That $\wt{H}$ is rigid analytic is relevant as the exactness of \cref{eq:HT-SES} can now find its non-abelian generalisation in geometric properties of $\wt {H}$: For $G=\G_m$, \cite[Theorem 4.1]{heuer-geometric-Simpson-Pic} shows that \cref{eq:multiplicative-HT-seq} induces on coarse moduli spaces a short exact sequence of rigid groups
 	\begin{equation}\label{intro:seq-over-K}
 	0\to \mathbf{Pic}^{\tt}_{X,\et}\to \HOM(\pi_1^\et(X),\G_m)\xrightarrow{\wt H} H^0(X,\wtOm^1_X)\otimes \G_a\to 0,
 \end{equation}
where $\mathbf{Pic}^{\tt}_{X,\et}$ is the topological torsion Picard functor, so here $\wt {H}$ is indeed a fibration. For $G=\G_m$, this was our crucial input to answer in \cite[Theorem 1.1]{heuer-geometric-Simpson-Pic} the question which Higgs bundles correspond to continuous characters of $\pi_1^\et(X)\to \G_m(K)$ under the equivalence \Cref{eq:intro-naHT-correspondence}. We envision that the geometric study of $\wt H$ leads  to an answer of this question for general $G$.
 
 In summary, with the construction of the moduli stacks $\cBun_{G,v}$ and $\cHiggs_{G}$ and their Hitchin morphisms $\mathcal H$ and $\wtcH$, we have thus established the technical foundations for our moduli theoretic approach to $p$-adic non-abelian Hodge theory.

\subsection{Application: $G$-torsors on rigid spaces in the $v$-topology}
\cref{intro-first-version-of-main-Thm} is of independent interest beyond \Cref{eq:intro-naHT-correspondence}: The sequence \cref{eq:HT-seq-G} describes how far the fully faithful functor
\[ \{G\text{-torsors on }X_{\et}\}\hookrightarrow \{G\text{-torsors on }X_{v}\} \]
is from  being an equivalence, in terms of the explicit set $\Higgs_G(X)$.
This is already very interesting for $G=\GL_n$, where it describes the difference between \'etale and \mbox{$v$-vector} bundles on $X$, which had so far been insufficiently understood: The question when a $v$-vector bundle is already \'etale appears naturally for example in the context of automorphic sheaves defined via descent from perfectoid Shimura varieties \cite{CHJ}.
Using  \cref{eq:HT-seq-G}, we can now give satisfactory answers to this question.
For example, we deduce the following a priori surprising criterion:

\begin{Corollary}\label{c:intro-analyticity-criterion}
	Let $X$ be a smooth rigid space and  $V$ a $G$-torsor on $X_{v}$. Then  for any \'etale map $f:U\to X$ with Zariski-dense image,  $V$ is \'etale-locally trivial if and only if $f^{\ast}V$ is.
\end{Corollary}
This extends the case $G=\G_m$ from \cite{heuer-v_lb_rigid} which we had used to simplify the proof that the automorphic sheaves of \cite{CHJ} are analytic vector bundles. \cref{intro-first-version-of-main-Thm} now improves this application: Setting $G:=\O^{+,\times}$, we see that the natural integral structure on the automorphic sheaves descend as well, i.e.\ these are finite locally free $\O^+$-modules.

For applications to $G$-torsors on $X_v$, also the following more categorical result is useful:

\begin{Theorem}\label{t:local-paCS-intro}
	Any \'etale morphism $f:X\to \mathbb T^d$   induces an equivalence of categories
	\[\{\text{small $G$-torsors on $X_v$}\}\isomarrow \{\text{small $G$-Higgs bundles on $X_{\et}$}\}\]
	that is natural in $G$ and $f$, but in general depends on the choice of $f$.
\end{Theorem}
As we explain in \S\ref{s:rel-loc-corresp}, this is a generalisation of a rigid analytic version of Faltings' ``local $p$-adic Simpson correspondence'' \cite[Theorem~3]{Faltings_SimpsonI}. Other instances of such a correspondence have previously been given by Abbes--Gros \cite{AGT-p-adic-Simpson}, Tsuji \cite{Tsuji-localSimpson}, Wang \cite{Wang-Simpson} and Morrow--Tsuji \cite{MorrowTsuji}, all in the case of $G=\GL_n$. Here we call a $G$-torsor small if it has a reduction of structure group to a certain open subgroup of $G$ depending on $X$. As every $G$-torsor on $X_v$ becomes small \'etale-locally on $X$, \Cref{t:local-paCS-intro} always applies locally on $X_{\et}$.
Hence $p$-adic non-abelian Hodge theory is a good framework to study $G$-torsors on $X_v$.

	\subsection{Relation to Faltings' global $p$-adic Simpson correspondence} 
	We now elaborate on the discussion surrounding \Cref{eq:intro-naHT-correspondence}, and explain how our work fits into the historic context:
	The original goal of $p$-adic non-abelian Hodge theory is to study the $K$-linear representations in $G$ of the \'etale fundamental group $\pi^{\et}_1(X)$ of a smooth proper rigid variety $X$ with the methods of $p$-adic Hodge theory. As this has so far almost exclusively been done for $G=\GL_n$, we restrict to this case for the moment. Following Faltings \cite{Faltings_SimpsonI}, the basic idea is to embed representations into a category of ``generalised representations'', which can be interpreted as vector bundles on $X_{v}$ (\cite[\S2]{heuer-G-torsors-perfectoid-spaces}). Namely, there is a fully faithful functor
	\[\Big\{\begin{array}{@{}c@{}l}\text{fin.-dim.\ continuous $K$-linear}\\\text{representations of $\pi^{\et}_1(X)$} \end{array}\Big\}\hookrightarrow \Big\{\text{vector bundles on $X_{v}$}\Big\},\]
	defined by regarding a representation as a descent datum for the trivial bundle along the pro-finite-\'etale universal cover $\wt X\to X$, which is a $v$-topological $\pi^{\et}_1(X)$-torsor over $X$.
	
	Based on Faltings' influential work in the case of curves, we have the following:
	
	\begin{Theorem}[$p$-adic Simpson correspondence]\label{conj:p-adic-Simpson-global}
		Assume that $K$ is algebraically closed and $X$ is a smooth proper rigid space over $K$. Then there is an equivalence of categories
		\[\{\text{vector bundles on $X_{v}$}\}\isomarrow \{\text{Higgs bundles on $X_{\et}$}\}.\]
		It is non-canonical and depends on choices of a $B_{\mathrm{dR}}^+/\xi^2$-lift of $X$ and an exponential on $K$.
	\end{Theorem}
	We prove this in \cite{heuer-proper-correspondence} based on the preparations in this article, a crucial role being played by the Hitchin fibration on the Betti side $\wt {\mathcal H}$.
	The conjecture was previously known for curves \cite{Faltings_SimpsonI} and line bundles \cite{heuer-v_lb_rigid}, and there were many partial results e.g.\ in good reduction settings \cite{Wang-Simpson} and arithmetic situations over discretely valued fields \cite{LiuZhu_RiemannHilbert}\cite{MinWang22}.
	
	On the other hand, moduli spaces can be used to strengthen this statement: As a first example, by testing on perfectoid spaces associated to profinite sets, we can endow the sets of isomorphism classes $|\cBun_{\GL_n,v}(K)|$ and $|\cHiggs_{\GL_n}(K)|$ on either side with a natural topology. One can then hope to refine \cref{conj:p-adic-Simpson-global} in a topological way that is very close to the complex Corlette--Simpson correspondence:
	\begin{Conjecture}\label{conj:p-adic-Simpson-global-continuity}
		The equivalence of \cref{conj:p-adic-Simpson-global} induces a homeomorphism 
		\[|\cBun_{\GL_n,v}(K)|\isomarrow |\cHiggs_{\GL_n}(K)|.\]
	\end{Conjecture}
	When $X$ is a curve, this conjecture is proved in \cite[Theorem~1.1.1]{HX}: Indeed, the choices  in \cref{conj:p-adic-Simpson-global} induce a trivialisation of the torsor $\mathscr P$ from \Cref{t:part-II-twist-for-curves} on $K$-points. Note that  \Cref{t:part-II-twist-for-curves} is completely canonical, so this gives a conceptual explanation of the choices in \Cref{conj:p-adic-Simpson-global} in a geometric fashion. We believe that this perspective should generalise.
	 
	 \subsection{Relation to other previous works, and outlook}
	 There has recently been a great deal of  activity in $p$-adic non-abelian Hodge theory, and we now sketch how the results of this article relate to some of these recent works, further to the relation to the local $p$-adic Simpson correspondence discussed in detail in \cref{s:rel-loc-corresp}.
	 
	 \medskip
	 
	 The first works in the area were due to Deninger--Werner \cite{DeningerWerner_vb_p-adic_curves}\cite{DeningerWerner_Simpson}, who define a functor from certain Higgs bundles on $X$ with vanishing Higgs field to representations, thus giving a $p$-adic analogue to the complex theory of Narasimhan--Seshadri \cite{NarasimhanSeshadri}. This was extended by W\"urthen to the rigid setting \cite{wuerthen_vb_on_rigid_var}.  From our perspective, this treats those $v$-vector bundles that are trivial \'etale-locally, and thus go to $0$ under the Hitchin fibration. So this theory plays out in the fibre over $0$ of $\wt {\mathcal H}$,  the analogue of the ``nilpotent cone''.

	 The same is true for the works of Liu--Zhu \cite{LiuZhu_RiemannHilbert} and Min--Wang \cite{MinWang22}, who study $\Q_p$-local systems, respectively $v$-vector bundles, on a smooth rigid space $X_0$ over a discretely valued field $k$, and relate these to Higgs bundles over the base-change to the completed algebraic closure $K$ of $k$. In both of these cases, the associated Higgs bundles are nilpotent, which from our perspective can be explained as follows:  Functoriality in \cref{t:main-thm-2-intro} implies that $\HTlog$ is Galois-equivariant, and the same is true for the morphism  $\wt {\mathcal H}$. Since the Galois invariants of $K(-1)$ are trivial, it follows that  $\wt {\mathcal H}$ sends the associated $v$-vector bundles to $0$.
	 
	 \medskip
	 
	 Regarding $v$-vector bundles in $p$-adic Hodge theory, let us also mention the relation to Sen theory, which is concerned with \mbox{$v$-vector} bundles on rigid spaces over discretely valued fields $L$: The classical theory of Sen \cite{Sen_cts_cohom} essentially describes $v$-vector bundles on $\Spa(L)$. Recently, Sen theory has been further developed for rigid spaces over $L$ by Shimizu \cite[\S3]{Shimizu_Senoperator} based on \cite{LiuZhu_RiemannHilbert}, Min--Wang \cite{MinWang22}, Pan \cite[\S3.2]{PanLocallyAnalytic} and Rodr\'iguez Camargo \cite[\S5]{camargo-locally-analytic}. The relation of our  work to that of Rodr\'iguez Camargo will be explained in \cite[\S4.2]{heuer-proper-correspondence}.
	 
	 \medskip
	 
	 The generalisation from vector bundles to $G$-torsors in $p$-adic non-abelian Hodge theory  has, to the best of our knowledge, so far only been explored by one previous work: In the case of Higgs bundles on curves with Higgs field $\theta=0$, Hackstein \cite{hackstein2008principal} has generalised the functor of Deninger--Werner to a functor on $G$-bundles for a reductive group $G$, in analogy to the complex theory by Ramanathan \cite{Ramanathan_Gbundles}. As one considers reductive $G$ in the complex theory, it is arguably a bit surprising that Theorems~\ref{t:main-thm-2-intro} and \ref{t:local-paCS-intro}  allow any rigid group $G$.
	 
	 \medskip
	 
	 Regarding moduli functors in $p$-adic non-abelian Hodge theory, the only case that has been studied before is the case of $G=\G_m$ mentioned above, which we have studied in \cite{heuer-diamantine-Picard} in terms of the rigid analytic Picard variety, the coarse moduli space of the $v$-stack $\cBun_{\G_m,v}$.
	 
	 \medskip
	 
	 As already mentioned,
	 a major open question raised by Faltings is which Higgs bundles correspond to representations of $\pi^\et_1(X)$ under the $p$-adic Simpson correspondence. Cases in which this is known include line bundles  \cite{heuer-v_lb_rigid} and abeloid varieties \cite{HMW-abeloid-Simpson}. Very recently, Xu \cite{xu2022parallel} has extended the construction of Deninger--Werner in the case of curves, by constructing an equivalence between representations and ``potentially Deninger--Werner Higgs bundles'' over $\C_p$. 
	 Little is known beyond these cases, and it currently seems difficult to even formulate a conjecture in general. We have shown in \cite{heuer-diamantine-Picard} that for $G=\G_m$, the answer is already quite subtle and best described in terms of moduli spaces. For general $G$,  this suggest studying the locus of pro-finite-\'etale bundles inside $\cBun_{G,v}$, which is started in \cite[\S9]{HX}. This is another main motivation to study moduli spaces in this context.
	
	\medskip
	
	Regarding \cref{t:main-thm-2-intro}, already the case of $G=\mathbb G_a$  is new for smoothoid spaces and has interesting applications: As we will explain in detail in \cite{heuer-relative-HT}, it yields a relative version of the Hodge--Tate spectral sequence for smooth proper morphisms $f:\mathcal X\to \mathcal S$ of rigid spaces.
	Related results have been obtained by Abbes--Gros \cite{AbbesGros_relativeHodgeTate}, Caraiani--Scholze \cite[Corollary~2.2.4]{CaraianiScholze}, He \cite[Theorem 12.2]{he2022cohomological}, and Gaisin--Koshikawa \cite[Theorem 1.5]{Gaisin-Koshikawa}.
	
\subsection*{Acknowledgements}
We thank Peter Scholze for very helpful discussions, and especially for his suggestion after seeing an early version of \cref{t:main-thm-2-intro} that this could be used to define an analogue of the Hitchin morphism.
We also thank Johannes Ansch\"utz, Juan Esteban Rodr\'iguez Camargo, Yagna Dutta, David Hansen, Arthur-C\'esar Le Bras, Lucas Mann, Alexander Thomas,  Annette Werner, Matti W\"urthen, Daxin Xu, Mingjia Zhang and Bogdan Zavyalov for very helpful conversations and/or comments on an earlier version of this article. We thank Tongmu He for pointing out a mistake in a previous version.

This work was funded by Deutsche Forschungsgemeinschaft (DFG, German Research Foundation) under Germany's Excellence Strategy -- EXC-2047/1 -- 390685813. The author was supported by DFG via the Leibniz-Preis of Peter Scholze.

\section*{Setup, conventions and notation}
Let $K$ be a perfectoid field over $\Q_p$. 
One example that will appear is the field $\Q_p^{\cyc}$ obtained by adjoining to $\Q_p$ all $p$-power roots of unity and completing $p$-adically.
We fix a ring of integral elements $K^+\subseteq K$, e.g.\ the ring of integers $K^+=\O_K$.
Let $\m\subseteq K^+$ be the maximal ideal.
For any $\alpha\in \mathbb R_{\geq 0}$, we write $\varpi^{\alpha}\m$ for the subset of $x\in K$  with $|x|<|\varpi|^\alpha$. 

Throughout we work with analytic adic spaces over $(K,K^+)$ in the sense of Huber \cite{Huber-ageneralisation}. We take it as part of the definition that adic spaces are sheafy. In the very few cases where we consider non-sheafy affinoids, we use the functor of points of Scholze--Weinstein \cite[Definition~2.1.5]{ScholzeWeinstein}, and we call this a pre-adic space following \cite[Definition 8.2.3]{KedlayaLiu-rel-p-p-adic-Hodge-I}. 

By a rigid space we mean an adic space locally of topologically finite type over $(K,K^+)$, in the sense of \cite[\S3]{Huber-ageneralisation}. Since for a given $K$ the ring $K^+$ won't change throughout, we often suppress $K^+$ from notation, and simply speak of rigid and adic spaces over $K$.
Associated to any rigid space $X$ over $K$ we have the pro-\'etale site $X_{\proet}$ in the sense of \cite{Scholze_p-adicHodgeForRigid}, which is now sometimes called the ``flattened'' pro-\'etale site.
For any adic space over $K$, we denote by $X_{\et}$ the \'etale site of Kedlaya--Liu \cite[Definition 8.2.16/19]{KedlayaLiu-rel-p-p-adic-Hodge-I}(cf \cite[Definition~7.1]{perfectoid-spaces}): In particular, this may in general contain non-sheafy pre-adic spaces, but we usually work in the setting of sousperfectoid spaces \cite{HK_sheafiness}\cite[\S6.3]{ScholzeBerkeleyLectureNotes} where all objects of $X_{\et}$ are sheafy. We call a morphism $f:X\to Y$ of adic spaces standard-\'etale if $X$ and $Y$ are affinoid and $f$ is a finite chain of compositions of finite \'etale morphisms and rational opens. It is immediate from the definition of \'etale morphisms that standard-\'etale morphisms form a basis of $X_{\et}$.

We use perfectoid spaces in the sense of \cite{perfectoid-spaces} and write $\Perf_K$ for the category of affinoid perfectoid spaces over $K$, or equivalently of perfectoid $(K,K^+)$-algebras. On this we have the \'etale and $v$-topology in the sense of \cite[\S7]{etale-cohomology-of-diamonds}, and we denote the corresponding sites by $\Perf_{K,\et}$ and $\Perf_{K,v}$.
To any adic space $X$ over $K$, Scholze associates a diamond $X^\diamondsuit$  over $K$ \cite[\S11,\S15]{etale-cohomology-of-diamonds}: For our purposes, as we fix the perfectoid base field $K$, we can identify this with a sheaf on $\Perf_{K,v}$. The functor $-^{\diamondsuit}$ is fully faithful on all subcategories of adic spaces that we consider, and moreover identifies the \'etale sites \cite[Lemma~15.6]{etale-cohomology-of-diamonds}. We therefore often drop the $-^{\diamondsuit}$ from notation and switch back and forth freely between $X$ and $X^\diamondsuit$. 
In particular, we denote by $X_{\qproet}$ and $X_v$ the quasi-pro-\'etale and the $v$-site of $X^\diamondsuit$ defined in \cite[\S14]{etale-cohomology-of-diamonds}. If $X$ is perfectoid, we also use the site $X_{\proet}$ from \cite[\S8]{etale-cohomology-of-diamonds}.

\section{Smoothoid spaces}
In this section we introduce the class of smoothoid spaces and prove some basic properties.
\subsection{Definition of smoothoid spaces and toric charts}
We start by fixing the notion of smooth morphisms we work with, following Huber:
\begin{Definition}[{\cite[Corollary 1.6.10]{huber2013etale}}]A morphism of adic spaces $X\to S$ is \textbf{smooth} if locally on source and target it admits a factorisation
		$X\xrightarrow{h} S\times_K \B^d\xrightarrow{\pi_1}  S$
		where $\B^d$ is the unit ball over $K$ of some dimension $d$, and $h$ is an \'etale morphism.
\end{Definition}

It is possible to further extend the definition using the pre-adic \'etale site, but we shall in this section restrict attention to (sheafy) adic spaces. In particular, it is part of our definition that $S\times_K \mathbb B^d$ exists as an adic space. This always holds if $S$ is sousperfectoid.

\begin{Definition}\label{d:smoothoid-space}
	We say that an adic space over $\Spa(K,K^+)$ is \textbf{smoothoid}  if it admits an open cover by subspaces $U$ that admit a smooth morphism of adic spaces $U\to Y$ to a perfectoid space $Y$ over $K$. We call such a morphism a \textbf{smooth chart}.
	A morphism of smoothoid spaces is simply a morphism of adic spaces between smoothoid spaces.
\end{Definition}

\begin{Definition}
	We denote by $\Smd_{K,\et}$ the category of smoothoid spaces over $K$ endowed with the \'etale topology. Note that any adic space \'etale over a smoothoid is again smoothoid.
\end{Definition}

For any $d\in \N$, let $\mathbb T^d=\Spa(K\langle T_1^{\pm 1},\dots,T_d^{\pm 1}\rangle)$ be the $d$-dimensional affinoid torus over $(K,K^+)$. Let
\begin{equation}\label{eq:toric-cover}
 {\mathbb T}^d_\infty:=\Spa(K\langle T_1^{\pm 1/p^\infty},\dots,T_d^{\pm 1/p^\infty}\rangle)\to \mathbb T^d,
 \end{equation}
a pro-\'etale affinoid perfectoid cover.
If $\Q_p^\cyc\subseteq K$, this is Galois with group $\Delta:=\Z_p^d(1)=\varprojlim_{n}\mu_{p^n}^d(K)$, i.e.\ it is a pro-\'etale torsor under $\Delta$ regarded as a pro-finite adic group.

\begin{Definition}\label{d:toric-chart}
	\begin{enumerate}
		\item For a smoothoid space $X$, a \textbf{toric chart} is a standard-\'etale map $f:X\to \mathbb T^d\times Y$ where $Y$ is an affinoid perfectoid space. We call $X$ \textbf{toric} if it admits a toric chart. 
		As standard-\'etale maps form a basis of the \'etale site, any smoothoid space can be covered by toric open subspaces. 
		\item Let $f':X'\to \mathbb T^{d'}\times Y'$ be a toric chart for a second smoothoid space $X'$. Then by a morphism of toric charts we mean a commutative diagram  
		of morphisms of adic spaces 
		\[\begin{tikzcd}
			X' \arrow[d, "f'"] \arrow[r] & X \arrow[d, "f"] \\
			\mathbb T^{d'}\times Y' \arrow[r,"\phi"] & \mathbb T^{d}\times Y
		\end{tikzcd}\]
	where $\phi$ is a product of a morphism  $Y'\to Y$ with a homomorphism $\mathbb T^{d'}\to  \mathbb T^{d}$ of rigid groups. Then $\phi$ lifts canonically to a morphism of perfectoid spaces $\mathbb T^{d'}_{\infty}\times Y' \to \mathbb T^{d}_{\infty}\times Y$.
	\end{enumerate}
\end{Definition}

\begin{Lemma}\label{l:smoothoid-is-pullback-of-smooth-rigid}
	Any toric smoothoid space $X$ fits into a Cartesian diagram
	\[ \begin{tikzcd}
		X \arrow[d,"g"] \arrow[r] & Y \arrow[d] \\
		Z \arrow[r,"f"] & S
	\end{tikzcd}\]
	where $f$ is a smooth morphism of smooth rigid spaces and $Y$ is an affinoid perfectoid space. 
\end{Lemma}
\begin{proof}
	Let $h:X\to Y\times \mathbb T^d$ be a toric chart
	where $Y$ is affinoid perfectoid. Let $(Y_i)_{i\in I}$ be the inverse system of all smooth rigid spaces over $(K,K^+)$ with compatible morphisms from $Y$, then by \cite[Proposition~3.2]{heuer-diamantine-Picard}, we have $Y\sim \varprojlim Y_i$, and thus  $Y\times \mathbb T^d\approx \varprojlim Y_i\times \mathbb T^d$. As $h$ is standard-\'etale, hence quasi-compact quasi-separated, this implies by \cite[Proposition 11.23]{etale-cohomology-of-diamonds} that there is some $i\in I$ for which $h$ descends to an \'etale map $Z\to Y_i\times \mathbb T^d$.
\end{proof}

We now collect some technical properties saying that smoothoid spaces are well-behaved as adic spaces. These are all immediate applications of known results about adic spaces:

\begin{Lemma}\label{l:small-adic-space-is-diamantine}
	Let $X$ be a smoothoid adic space over $K$.
	\begin{enumerate}
		\item $X$ is sousperfectoid. If $X$ is affinoid, then the structure sheaf $\O$ is acyclic on $X_{\et}$. 
		\item Diamondification defines an equivalence $X_{\et}^{\diamondsuit}=X_{\et}$ that identifies the structure sheaves, where we see $X^{\diamondsuit}$ as a $v$-sheaf on $\Perf_K$ with structure sheaf induced from that of $\Perf_K$.
		\item For the natural map $\nu:X_{v}^{\diamondsuit}\to X_{\et}$, we have $\nu_{\ast}\O^+=\O^+$.
		\item Diamondification on smoothoid adic spaces defines a fully faithful functor
		\[ \Smd_K\to \{\text{diamonds over }\Spd(K)\}\]
		\item Any finite locally free $\O$-module on $X_{\et}$ is already locally trivial in the analytic topology.
	\end{enumerate}
\end{Lemma}

We may therefore freely switch between regarding $X$ as a diamond or as an adic space.
\begin{proof}
	Part 1 follows from \cite[Prop.\ 6.3]{ScholzeBerkeleyLectureNotes}, see also \cite[Remark~10.6]{HK_sheafiness}, or \cref{l:module-splitting-of-toric-tower} below for an explicit argument. Acyclicity then holds for any analytic sheafy affinoid adic space by \cite[Theorem 2.4.23]{KedlayaLiu-rel-p-p-adic-Hodge-I}. For 2, see \cite[Lemma 15.6]{etale-cohomology-of-diamonds}. The second part of the statement follows from \cite[Lemma~11.31]{etale-cohomology-of-diamonds}. Part 3 follows from part 1 by \cite[Proposition 11.3]{HK_sheafiness}.  Part 4 follows from part 2. For part 5, see \cite[Theorem 8.2.22.(d)]{KedlayaLiu-rel-p-p-adic-Hodge-I}.
\end{proof}
\subsection{Differentials on smoothoid spaces via Hodge--Tate comparison}
Our next goal is to show that there is an intrinsic notion of ``global'' K\"ahler differentials $\Omega_{X}^1$ on a smoothoid adic space $X$: From the perspective of Hodge cohomology, it makes sense to postulate that perfectoid spaces should have no differentials. This suggests that for a smooth morphism $X\to Y$ over a perfectoid space $Y$, the global differentials of $X$ should be given by the sheaf of relative differentials $\Omega^1_{X|Y}$ defined by Huber \cite[(1.6.2)]{huber2013etale}. 

The issue with this definition is that it is not immediately obvious that it is independent of the smooth chart of $X$, and hence that it glues. It is therefore better to have a more intrinsic definition of $\Omega^1_X$.  For this we build on 
 the following result  mentioned in the introduction, which is crucial in Scholze's perspective on the Hodge--Tate comparison isomorphism:
\begin{Proposition}[{\cite[Proposition~3.23]{Scholze2012Survey}\cite[Proposition.2.25]{heuer-v_lb_rigid}}]\label{p:Scholze-Prop-3.23}
	Let $X$ be a smooth rigid space over $K$ and consider the natural map $\nu:X_{\proet}\to X_{\et}$. Then there is for any $n\in\N$ a natural isomorphism
	\[ R^{n}\nu_{\ast}\O=\Omega^n_X\{-n\}.\]
	Here $\{-n\}$ denotes a Breuil--Kisin--Fargues twist, i.e.\ a Tate twist by $\Z_p(-n)$ if $\Q_p^\cyc\subseteq K$.
	The isomorphism identifies the cup product on $R\nu_{\ast}\O$ with the wedge product on  differentials.
\end{Proposition}
\begin{Remark}
	\begin{enumerate}
		\item One can always make choices to identify $\Omega^n_X\{-n\}\cong \Omega^n_X$. But it is often more natural to keep it, for example to keep track of Galois actions.
		\item Since $H^i_v(Y,\O)=0$ for any affinoid perfectoid space $Y$, one can equivalently formulate \cref{p:Scholze-Prop-3.23} in terms of any finer site over $X_{\proet}$, in particular for $X_{v}\to X_{\et}$. This is the perspective we use in the following as it can be adapted to smoothoid spaces.
	\end{enumerate}
\end{Remark}
The main goal of this subsection is to prove the following generalisation of \cref{p:Scholze-Prop-3.23}:

\begin{Proposition}\label{p:RnuO-for-small-dmd}
	Let $X$ be a smoothoid space. Let $\nu:X_v\to X_{\et}$ be the natural map. Then
	\begin{enumerate}
		\item 
	$R^n\nu_{\ast}\O$ is a vector bundle on $X_{\et}$.
	\item Suppose that there is a smooth morphism $f:Z\to S$ of rigid spaces (we do not require $Z$ and $S$ to be smooth over $K$) such that $X$ fits into a Cartesian diagram
	\[\begin{tikzcd}
		X \arrow[d,"g"] \arrow[r] & Y \arrow[d] \\
		Z \arrow[r,"f"] & S
	\end{tikzcd}\]
	where $Y$ is perfectoid. Then there is a canonical and functorial isomorphism 
	\[R^n\nu_{\ast}\O=g^{\ast}\Omega^n_{Z|S}\{-n\}.\]
	\end{enumerate}
\end{Proposition}

Before we give a proof, we discuss some consequences. Firstly, we can use this to define:

\begin{Definition}\label{d:differentials-on-smoothoids}
	Let $X$ be a smoothoid space. For any $n\in \N$, we set $\wtOm^n_{X}:=R^n\nu_{\ast}\O$. 
	
	We also set $ \wtOm_X:=\wtOm_X^1$.
It follows from the local description in \cref{p:RnuO-for-small-dmd} that $\wtOm^n_{X}$ enjoys all usual compatibilities of the K\"ahler differentials. For example, the cup product induces a natural isomorphism
$\wedge^n\wtOm_{X}\isomarrow\wtOm^{n}_{X}$
that we can use to define a wedge product 
\begin{equation}\label{eq:wedge-for-wtOm}
	\wedge:\wtOm^i_{X}\otimes \wtOm^j_{X}\to \wtOm^{i+j}_{X}.
\end{equation}
\end{Definition}
\begin{Definition}\label{d:differential-dimension}
	Let $X$ be a smoothoid adic space. Then the number $d:=\rk_{\O_X}\wtOm_{X}$ is locally constant on $X$, and we refer to it as the \textbf{smooth dimension} $\dim X$ of $X$.
\end{Definition}

A second advantage of the definition via $R^1\nu_{\ast}\O$ is that by the usual cohomological comparison between the big and small \'etale sites, we immediately obtain:

\begin{Lemma}\label{l:wtOm-is-sheaf-on-big-etale-site}
	The functor $\wtOm^n:X\mapsto H^0(X,\wtOm^n_X)$ is a sheaf on the big \'etale site $\Smd_{K,\et}$ of smoothoid spaces over $K$. Explicitly, it is given by $R^n\mu_{\ast}\O\{n\}$ for $\mu:\Smd_{K,v}\to \Smd_{K,\et}$.
\end{Lemma}
\begin{Lemma}[cotangent sequence]\label{l:cotangent-sequence-for-smoothoids}
	Let $f:X\to Y$ be a morphism of smoothoid spaces that is of topologically finite presentation. Then there is a natural right-exact sequence of sheaves on $X_{\et}$
	\[f^{\ast}\wtOm_{Y}\to \wtOm_X\to \Omega_{X|Y}\{-n\}\to 0\]
	where $\wtOm_{X|Y}$ is Huber's sheaf of relative K\"ahler differentials \cite[(1.6.2)]{huber2013etale}.
	If $f$ is smooth, this is an exact sequence of vector bundles. If $f$ is \'etale, the first map is an isomorphism.
\end{Lemma}
\begin{proof}
	The first map exists by adjunction via \cref{l:wtOm-is-sheaf-on-big-etale-site}. To prove the lemma, we can work locally and may thus assume by rigid approximation that $f$ is the base-change of a morphism $f_0:X_0\to Y_0$ of smooth rigid spaces along a morphism $Y\to Y_0$. For $f_0$, the lemma is clear from \cref{p:Scholze-Prop-3.23}. By \cref{p:RnuO-for-small-dmd} we get the desired sequence via pullback.
	If $f$ is smooth we can as in the proof of \cref{l:smoothoid-is-pullback-of-smooth-rigid} arrange for $f_0$ to be smooth.
\end{proof}

\subsection{Computations with toric charts}\label{s:comp-toric-charts}
For the proof of \cref{p:RnuO-for-small-dmd}, we can work locally and assume that we have a toric chart 
$f:X\to \mathbb T^d\times Y$ for some affinoid perfectoid space $Y$. We fix $f$ for the rest of this subsection.

We start by adapting some technical results from \cite[\S4-5]{Scholze_p-adicHodgeForRigid} to smoothoid spaces:
The chart $f$ induces a perfectoid cover $X_\infty\to X$ by pullback of $X$ along the toric cover \cref{eq:toric-cover}.  
Throughout this subsection, we use the following notation for the associated Huber pairs:
\[ \begin{tikzcd}
	X_\infty \arrow[d] \arrow[r] & {\mathbb T}^d_\infty\times Y \arrow[d]\arrow[r] &  {\mathbb T}^d_\infty\arrow[d]\\
	X \arrow[r] & \mathbb T^d\times Y\arrow[r] & \mathbb T^d
\end{tikzcd} 
\begin{tikzcd}
	{}\arrow[r,squiggly,"\O(-)"]&{}
\end{tikzcd}
 \begin{tikzcd}[column sep = {2.2cm,between origins},row sep = {1.2cm,between origins}]
	(R_\infty,R_\infty^+)& (B_\infty,B_\infty^+) \arrow[l]&(A_\infty,A_\infty^+) \arrow[l]\\
	(R,R^+) \arrow[u] & (B,B^+)  \arrow[u] \arrow[l] & (A,A^+). \arrow[u] \arrow[l]
\end{tikzcd}\]
 If $\Q_p^\cyc\subseteq K$, then the vertical maps are $\Delta$-torsors.
For $n\in \N$, let $\mathbb T^d_n$ be the torus in the variables $T_1^{1/p^n},\dots,T_d^{1/p^n}$, then $\mathbb T^d_\infty \sim \varprojlim_{n\in \N}\mathbb T^d_n$. Let $X_n=\Spa(R_n,R_n^+)$ be the pullback of $X\to \mathbb T^d$ to $\mathbb T^d_n$, then $X_\infty\sim \varprojlim X_n$.  As $X_\infty$ is standard-\'etale over ${\mathbb T}^d_\infty\times Y$, it is perfectoid.

Finally, we fix any real number $1\geq \alpha\geq \tfrac{1}{p-1}$.
This will be required to kill some torsion.
The following is the analogue of \cite[Lemmas~4.5 and 5.5]{Scholze_p-adicHodgeForRigid} in our setting.

\begin{Lemma}\label{l:Sch13-4.5/5.5}
	\begin{enumerate}
		\item 
		There is $\beta\in \mathbb R_{>0}$ such that the map
		$R^+\hotimes_{A^+}A_{\infty}^+\to R^+_\infty$
		is injective  with $p^\beta$-torsion cokernel. 
		\item Let $\epsilon>0$, then after replacing $X$ by the cover $X_n\to X$ for $n\gg 0$, we can take $\beta<\epsilon$.
		\item Assume that  $\Q_p^\cyc\subseteq K$. Then for any $s\in \N$ and $i\geq 0$, the kernel and cokernel of 
		\[ H^i_{\cts}(\Delta,R^+/p^s)\to H^i_{\cts}(\Delta,R^+_\infty/p^s)\]
		are killed by  $p^{\gamma}$ where $\gamma:=2\beta+\alpha$. The same is true for
		$H^i_{\cts}(\Delta,R^+)\to H^i_{\cts}(\Delta,R^+_\infty)$.
	\end{enumerate}
\end{Lemma}
\begin{proof}
	We first treat the case $X=\mathbb T^d\times Y$.
	Write $S^+=\O^+(Y)$ and consider the map $	g:S^+\hotimes_{K^+}A^+ \to B^+$.
	Its domain can be described as $S^+\langle T_1^{\pm},\dots,T_d^{\pm}\rangle$ which is a ring of definition of $B=S\langle T_1^{\pm},\dots,T_d^{\pm}\rangle$, so $g$ is an isomorphism after inverting $p$. It thus has bounded $p$-torsion cokernel as $B$ is uniform. On the other hand, after applying $\hotimes_{A^+}A_{\infty}^+$ we have maps
	\[S^+\hotimes_{K^+}A^+\hotimes_{A^+}A_{\infty}^+\to B^+\hotimes_{A^+}A_{\infty}^+\to B^+_\infty\aeq S^+\hotimes_{K^+}A^+_\infty\]
	where the last almost isomorphism comes from \cite[Proposition 6.18]{perfectoid-spaces}.
	The composition is clearly an isomorphism. Since the map $A^+/p^n\to A_\infty^+/p^n$ is flat for all $n$, the first map is still injective with bounded $p$-torsion cokernel. Again by flatness, the middle term is $p$-torsionfree. It follows that all of the above maps are almost isomorphisms, thus so is $g$.
	
	We now add a standard-\'etale map $X\to \mathbb T^d\times Y$. For this, the argument from \cite[Lemma~4.5]{Scholze_p-adicHodgeForRigid} goes through: Let $Z$ be a smoothoid for which we already know the statement, e.g.\ $\mathbb T^d\times Y$, and consider a standard-\'etale map $h:X\to Z$. Write $Z=\Spa(B,B^+)$ and $X=\Spa(R,R^+)$ with toric covers $ X_\infty=\Spa(R_\infty,R^+_\infty)$ and $ Z_\infty=\Spa(B_\infty,B^+_\infty)$. 
	\begin{Claim}
		The map $R^+\hotimes_{B^+}B^+_\infty\to R^+_\infty$
		has bounded $p$-torsion cokernel, and is an isomorphism after inverting $p$.
	\end{Claim}
	\begin{proof}
		If $h$ is a rational localisation defined by some $f_1,\dots,f_r,g\in B$, then a ring of definition $R_0:=B^+\langle \frac{f_1}{g},\dots,\frac{f_r}{g}\rangle$ of $(R,R^+)$ is given by the $p$-adic completion of the sub-algebra of $B^+[\frac{1}{g}]$ generated by $\frac{f_1}{g},\dots,\frac{f_r}{g}$. Similarly, $B^+_\infty\langle \frac{f_1}{g},\dots,\frac{f_r}{g}\rangle$ is a ring of definition of $(R_\infty,R_\infty^+)$. Consider now 
		$T_0:=R_0\hotimes_{B^+}B^+_\infty$
		and let $T^+$ be the integral closure of the image of $T_0$ in $T:=T_0\tf$, 
		then by construction $\Spa(T,T^+)=X\times_ZZ_\infty=\Spa(R_\infty,R_\infty^+)$. It follows that the natural map $T_0\to R_\infty^+$
		is an isomorphism after inverting $p$, and that both its image and $R_\infty^+$ are rings of definition of $R_\infty$, hence $T_0\to R_\infty^+$ has bounded $p$-torsion cokernel. The claim follows by considering the composition 
		$T_0\to  R^+\hotimes_{B^+}B^+_\infty\to  R_\infty^+$
		since $R_0\subseteq R^+$ has bounded $p$-torsion cokernel, both being rings of definition of $R$.
		
		If instead $h$ is finite \'etale, then $B_\infty\to R_\infty=R\otimes_{B}B_\infty$ is finite \'etale, and the image of  $R^+\hotimes_{B^+}B_\infty^+\to R_\infty$ is a ring of definition of $R_\infty$, so the statement follows by uniformity.
	\end{proof}
	Combining the first part of the proof and the claim, we see that in the composition
	\[ R^+\hotimes_{A^+}A^+_\infty\to  R^+\hotimes_{B^+}B^+_\infty\to R^+_\infty,\]
	both maps have bounded $p$-torsion cokernel and become an isomorphism after inverting $p$. Since the first term is $p$-torsionfree by the same flatness argument as before, it follows that the composition is injective already before inverting $p$. This proves part (i).	
	
	Part (ii) follows exactly as in \cite[Lemma~4.5]{Scholze_p-adicHodgeForRigid}:  By the explicit description of the integral subrings of rational localisations  \cite[Lemma 6.4]{perfectoid-spaces}, respectively finite \'etale extensions of a perfectoid algebra, there is for any $\epsilon>0$ a subalgebra $R_{\infty,\epsilon}^+\subseteq R_\infty^+$ that is topologically finitely generated over $B_\infty^+$ and such that the cokernel $R_\infty^+/R_{\infty,\epsilon}^+$ is $p^\epsilon$-torsion.

	To deduce part (iii) from (i), consider the diagram
	\[\begin{tikzcd}
		{H^i_{\cts}(\Delta,R^+/p^s)} \arrow[r] & {H^i_{\cts}(\Delta,R_\infty^+/p^s)} \\
		{H^i_{\cts}(\Delta,A^+/p^s)}\otimes_{A^+}R^+ \arrow[r] \arrow[u] & {H^i_{\cts}(\Delta,A_\infty^+\otimes_{A^+}R^+/p^s).} \arrow[u]
	\end{tikzcd}\]
	
	By part (i), the right vertical map has $p^{2\beta}$-torsion kernel and cokernel.
	The left vertical map is an isomorphism: $\Delta$ acts trivially on $R^+$ and $A^+$, so $H^i_{\cts}(\Delta,-)=\Hom_{\cts}(\wedge^i_{\Z_p}\Delta,-)$. 
	
	The bottom map has $p^{\alpha}$-torsion kernel and cokernel by \cite[Lemma~5.5]{Scholze_p-adicHodgeForRigid}, which says that
	\[H^i_{\cts}(\Delta,A_\infty^+\otimes_{A^+}R^+/p^s)=H^i_{\cts}(\Delta,A_\infty^+)\otimes_{A^+}R^+/p^s\]
	and that 
	$H^i_{\cts}(\Delta,A^+)\to H^i_{\cts}(\Delta,A^+_\infty)$
	is injective with $p^{\alpha}$-torsion cokernel.
	
	It follows that the top map has $p^{2\beta+\alpha}$-torsion kernel and cokernel.
	The same works for $H^i_{\cts}(\Delta,R^+_\infty)$ when we instead  use the last part of \cite[Lemma~5.5]{Scholze_p-adicHodgeForRigid}.
\end{proof}
As an immediate application, this shows more directly that $X$ is sousperfectoid:
\begin{Lemma}\label{l:module-splitting-of-toric-tower}
	The normalised traces $A^+_\infty\to A^+_n$ induce canonical $R_n$-linear and continuous splittings $\tr_n:R_\infty\to R_n$ such that for any $x\in R_\infty$, we have $\tr_n(x)\to x$ for $n\to \infty$.
\end{Lemma}
\begin{proof}
	By \cref{l:Sch13-4.5/5.5}.1, we get such splittings by applying $-\hotimes_{A^+}R^+$ and inverting $p$.
\end{proof}

\begin{proof}[Proof of \cref{p:RnuO-for-small-dmd}]
	For part 1,  we may work locally and assume that $X$ is toric. Let us first assume $\Q_p^\cyc\subseteq K$. Then the Cartan--Leray sequence \cite[Proposition~2.18]{heuer-v_lb_rigid} for  the $\Delta$-torsor $X_\infty\to X$  induces isomorphisms
	$H^n_v(X,\O)=H^n_{\cts}(\Delta,\O(X_\infty))$.
	By Lemma~\ref{l:Sch13-4.5/5.5}.3 and \cref{p:Scholze-Prop-3.23},
	\[H^n_{\cts}(\Delta,\O(X_\infty))=H^n_{\cts}(\Delta,\O(X))=H^n_{\cts}(\Delta,\O({\mathbb T}^d))\hotimes_{\O(\mathbb T^d)}\O(X)=H^0(\mathbb T^d,\wtOm^n)\hotimes \O(X).\]
	This in particular shows that the cup product induces an isomorphism
	\[\textstyle\bigwedge^nR^1\nu_{\ast}\O\isomarrow R^n\nu_{\ast}\O.\]
	
	Next, we prove part 2. For this we use the following Lemma:
	\begin{Lemma}[{\cite[\S3.4]{bhatt2017lecture}}]\label{l:bhatt-R1nuO-via-cotangent}
		Let $X$ be any stably uniform rigid space over $K$. Then there is a canonical and functorial morphism of sheaves on $X_\et$
		\[\mathcal H^0(\widehat{L_{\O^+_{X_\et}|\Z_p}})\tf\to R^1\nu_{\ast}\O\{1\}.\]
		When $X$ is a rigid space, the first term is $=\Omega_{X|K}$. Under this identification, when $X$ is moreover smooth, then this morphism agrees with the isomorphism from \Cref{p:Scholze-Prop-3.23}.
	\end{Lemma}
	\begin{proof}
		When $K$ is algebraically closed, this is proved in {\cite[\S3.4]{bhatt2017lecture}}:  The reference assumes that $X$ is a smooth rigid space, but the construction of the morphism goes through mutatis mutandis when we replace the formal model $\mathfrak X_{aff}$  by the site $(X_\an,\O^+)$.
		
		The case of general $K$ can be seen in the same way: We only need to see that 
		\[\widehat {L_{K^+|\Z_p}}\tf=K\{1\}[1].\]
		Let $C$ be the completion of an algebraic closure of $K$, then the transitivity triangle for $\Z_p\to K^+\to C^+$ yields a canonical isomorphism, compatible with Galois actions
		\[\widehat {L_{K^+|\Z_p}}\tf\otimes_{K}C=\widehat {L_{C^+|\Z_p}}\tf.\]
		The statement follows by taking Galois invariants, like in \cite[Proposition 2.25]{heuer-v_lb_rigid}.
	\end{proof}
	Applying \Cref{l:bhatt-R1nuO-via-cotangent} to $Z$, we obtain a natural morphism
	\[ g^\ast\Omega_{Z|K}^1\to g^\ast R^1\nu_\ast \O_Z\to R^1\nu_\ast \O_X,\]
	and similarly for $S$. By functoriality, the following diagram commutes:
	\[\begin{tikzcd}
		g^\ast\Omega_{Z|K}^1 \arrow[r] & R^1\nu_{\ast}\O_X \\
		g^{\ast}f^{\ast}\Omega^1_{S|K} \arrow[u] \arrow[r] & g^{\ast}R^1\nu_{\ast}\O_Y \arrow[u]
	\end{tikzcd}\]
	But $Y$ is perfectoid, so the bottom right term vanishes. By the cotangent sequence, we thus obtain a natural morphism
	\[	g^\ast\Omega_{Z|S}^1\to R^1\nu_{\ast}\O_X.\]
	To see that this is an isomorphism, we can work locally on $Z$ and thus assume that $f$ factors into an \'etale morphism $Z\to\mathbb T^d\times S$ and the projection $\mathbb T^d\times S\to S$. Then $X$ is \'etale over $\mathbb T^d\times Y$. The statement is then immediate from comparing Cartan--Leray sequences for the toric towers $\mathbb T^d_\infty\times S\to \mathbb T^d\times S$ and $\mathbb T^d_\infty\times Y\to \mathbb T^d\times Y$.
	
	\medskip
	
	For general $K$, we deduce the results via Galois descent, exactly as in \cite[Proposition~2.25]{heuer-v_lb_rigid}: Let $(C,C^+)$ be the completion of an algebraic closure of $(K,K^+)$, let $Q=\Gal(C|K)$ be the Galois group and $X_C$ the base-change. Then for any $i,j\geq 0$ we have
	\[ H^i_{\cts}(Q,H^j_v(X_C,\O))=\begin{cases}
		H^0(X,\wtOm_{X}^j)& i=0\\
		0 &i>0,
	\end{cases}\]
	since $H^j_v(X_C,\O)= \wtOm_{X_C}^j(X_C)\cong \O(X_C)^k$ as $Q$-modules for $k={d\choose j}$ by part 1, and the map
	$H^i_{\cts}(Q,\O(X_C)) \hookrightarrow H^i_{\cts}(Q,\O(X_{C,\infty}))=H^i_v(X_\infty,\O)=0$ for $i>0$ is injective  due to \cref{l:module-splitting-of-toric-tower}.
	The Cartan--Leray sequence for $X_C\to X$ now implies $H^n_v(X,\O)=H^0(X,\wtOm_{X}^n)$ for any $n\in\N$.
	Part 2 follows as the map $h$ is an isomorphism after base-change to $C$.
\end{proof}

For a toric smoothoid space $X$, this gives an explicit description of $\wtOm_X$ depending on $f$:
\begin{Lemma}\label{l:HT-isom-on-small-dmd}
	Assume $K$ contains $\Q_p^\cyc$. Then the chart $f$ induces an isomorphism
	\[ \HT_f:\Hom_{\cts}(\Delta,\O(X))\to H^1_{\cts}(\Delta,\O(X_\infty))\to H^1_{v}(X,\O)\to H^0(X,\wtOm_X)\]
	where the second map is from the Cartan--Leray sequence  (\cref{p:Cartan--Leray}) for $X_\infty\to X$ and the last is from the Leray sequence for the morphism $\nu:X_{v}\to X_{\et}$, using \cref{p:RnuO-for-small-dmd}.
\end{Lemma}
\begin{proof}
	The first map is the isomorphism by \cref{l:Sch13-4.5/5.5}.3.
	The second is an isomorphism as $H^1_v(X_\infty,\O)=0$.
	The third is an isomorphism as $H^i_{\an}(X,\O)=0$ for $i\geq 1$ by \cref{l:small-adic-space-is-diamantine}.
\end{proof}

\begin{Definition}\label{l:wtOm-integral-submodule-ind-by-f}
	Let $\wtOm^+_X\subseteq \wtOm_X$ be the finite free $\O^+$-submodule (depending on $f$) given by the image of
	$\HT_f:\Hom_{\cts}(\Delta,\O^+(X))\otimes_{\O^+(X)}\O^+\to \wtOm_X$.
\end{Definition}

\subsection{Non-cyclotomic base fields}\label{s:non-cyclotomic}
Rather than deducing the case of general $K$ from that where $K$ contains $\Q_p^\cyc$ by Galois descent, one can also give a more direct argument. We record it since we need it later: 

Assume that $K$ does not contain all $p$-power roots of unity and consider the cyclotomic extension $K^\cyc|K$ obtained by adjoining them. For simplicity, let us assume that $K$ contains at least $\zeta_p$. Then $\Spa(K^\cyc)\to \Spa(K)$ is a torsor under the profinite group $Q:=\Gal(K^\cyc|K)\subset \Z_p^\times$ such that $Q= 1+p^m\Z_p\cong \Z_p$ for some $m\in \N$.

Consider the base-change $X^\cyc\to \Spa(K^\cyc)$ of $X\to \Spa(K)$ along this extension. Denote by $(R^\cyc,R^{\cyc+})$ its global sections. Similarly, let  $(R^\cyc_\infty,R^{\cyc+}_\infty)$ be the global sections of the base-change $X_\infty^\cyc\to  \Spa(K^\cyc)$ of $X_\infty\to \Spa(K)$, as in the last section. Then
\[ X_\infty^\cyc\to \Spa(K^\cyc)\to\Spa(K)\]
is again a pro-finite-\'etale Galois torsor for a group $\Lambda$ which is canonically a split extension
\[0\to \Delta\to \Lambda\to Q\to 0,\]
where $\Delta$ is the Galois group of $X_\infty^\cyc\to \Spa(K^\cyc)$ as in the last section.
Namely, $\Lambda=\Delta\rtimes Q$ is the semi-direct product given by the action of $Q\subseteq \Z_p^\times$ by multiplication on $\Delta$. Writing elements of $\Lambda$ as tuples $(\gamma, \sigma)$, these act on $K^{\cyc}\langle T^{\pm 1/p^\infty}\rangle$ via
$(\gamma,\sigma)\cdot T^m= \sigma(a_m)\zeta^{m\gamma}T^m$.

We can now describe $\mathrm{R}\Gamma_{\cts}(\Lambda,R^{\cyc+})$, as in \cref{l:Sch13-4.5/5.5}.
While calculations like these are standard in $p$-adic Hodge theory, there is a major difference to the classical setting as in \cite[\S 3]{tate1967p}: For $0\neq i\in \Z$ we have $\Q_p^\cyc(i)^{\Z_p^\times}=0$, but for perfectoid $K$ we instead have:
\begin{Lemma}\label{l:cts-Lambda-cohom-of-Kcyc}
	For $i\in \Z$, we have $K^{+\cyc}(i)^Q\aeq K^{+}$ and $H^j_{\cts}(Q,K^{+\cyc}(i))\aeq 0$ for $j>0$.
\end{Lemma}
\begin{proof}
	 $K^{+\cyc}(i)$ with its $Q$-action defines an invertible $\O^+$-module $L_i$ on $X:=\Spa(K,K^+)$ by descent from $\Spa(K^\cyc)$, and $H^j_{\cts}(Q,K^{+\cyc}(i))$ computes $H^j_v(X,L_i)$ by the Cartan--Leray sequence. But any such $L_i$ is trivial by \cite[Corollary~2.28]{heuer-G-torsors-perfectoid-spaces}, so $H^j_v(X,L_i)\aeq 0$.
\end{proof}

\begin{Lemma}\label{l:non-cyclotomic-version-of-Sch13-4.5/5.5}
	\begin{enumerate}
		\item	There is $\gamma>0$ such that for any $s,i\geq 0$, the kernel and cokernel of
		\[ H^i_{\cts}(\Lambda,R^{\cyc+}/p^s)\to H^i_{\cts}(\Lambda,R^{\cyc+}_\infty/p^s)\]
		are annihilated by  $p^{\gamma}$. The same is true for the map
		$H^i_{\cts}(\Lambda,R^{+\cyc})\to H^i_{\cts}(\Lambda,R^{+\cyc}_\infty)$.

		\item The restriction
		$H^1_{\cts}(\Lambda,R^{\cyc+}/p^s)\to H^1_{\cts}(\Delta,R^{\cyc+}/p^s)^Q$
		is split surjective:  a splitting $r$ is given by sending $\rho:\Delta\to R^{\cyc+}/p^s$ to the $1$-cocycle $r(\rho):(\gamma,\sigma)\mapsto \rho(\gamma)$. In the limit over $s$, this also defines a splitting of
		$H^1_{\cts}(\Lambda,R^{+\cyc})\to H^1_{\cts}(\Delta,R^{+\cyc})^Q$.
		\item As an $R^+$-module, $H^1_{\cts}(\Delta,R^{+\cyc})^Q$ is almost finite free of rank $d$.
	\end{enumerate}
\end{Lemma}
\begin{proof}
	Let $A^{\cyc}=A\hotimes_KK^\cyc$.
	The main calculation is that for any $i,j\geq 0$ we have
	\[ H^j_{\cts}(Q,H^i_{\cts}(\Delta,A^{\cyc+}/p^s))\aeq \begin{cases}
		A^{+}/p^s & j=0,\\
		0&j>0.
	\end{cases}\]
	Indeed, writing $A^{\cyc+}/p^s=\oplus_{i\in \Z}K^{\cyc+}/p^s\cdot T^i$, this follows directly from \cref{l:cts-Lambda-cohom-of-Kcyc}.
	Let now $c>0$  be such that kernel and cokernel of the map
	$A^{\cyc+}\otimes_{A^+}R^+/p^s\to R^{+\cyc}/p^s$
	are killed by $p^c$ for all $s$. Then
	$H^j_{\cts}(Q,H^i_{\cts}(\Delta,R^{\cyc+}/p^s))$
	is $p^{2c}$-torsion for $j=1$ and any $i$, and vanishes for $j>1$ since $Q\cong \Z_p$. By  inflation-restriction, it follows that for any $n\geq 1$, the map
	\[H^n_{\cts}(\Lambda,R^{\cyc+}/p^s)\to H^{n}(\Delta,R^{\cyc+}/p^s)^Q\]
	is surjective with $p^{2c}$-torsion kernel.
	The exact same argument also works for $R^{\cyc+}_\infty/p^s$, so we can assume that there is $c>0$ independent of $s$ such that in the commutative diagram
	\[ \begin{tikzcd}
		{H^i_{\cts}(\Lambda,R^{+\cyc}/p^s)} \arrow[d] \arrow[r] & {H^i_{\cts}(\Delta,R^{+\cyc}/p^s)^Q} \arrow[d] \\
		{H^i_{\cts}(\Lambda,R_{\infty}^{+\cyc}/p^s)} \arrow[r] & {H^i_{\cts}(\Delta,R_{\infty}^{+\cyc}/p^s)^Q},
	\end{tikzcd}\]
	the top, bottom and right map have $p^c$-torsion kernel and cokernel, thus also the left map.
	
	To see part 3, we use that $\Hom_{\cts}(\Delta,R^{\cyc +})\cong R^{\cyc + }(-1)^d$
	as $R^+$-modules with $Q$-action. We can regard $R^{\cyc}(-1)$ with its $Q$-action as a descent datum for an invertible $\O^+$-module $L$ along the cover $X^\cyc\to X$ which has Galois group $Q$. In fact, since $R^{\cyc}(-1)=R\hotimes_KK^\cyc(-1)$, this comes via pullback from the invertible $\O^+$-module $L_0$ on $\Spa(K)_v$ associated to the descent datum $K^{\cyc+}(-1)$ for $\Spa(K^\cyc)\to\Spa(K)$. But $L_0$ is free by \cite[Corollary~2.28]{heuer-G-torsors-perfectoid-spaces}, hence the global sections $R^{\cyc}(-1)^Q$ of $L$ are a free $R^+$-module.
	
	It remains to see part 2: As the elements of $H^1_{\cts}(\Delta,R^{+\cyc}/p^s)^Q=\Hom_{\cts}(\Delta,R^{+\cyc}/p^s)^Q$ are precisely the $Q$-equivariant homomorphisms, this is the elementary observation:
\end{proof} 
\begin{Lemma}\label{l:section-for-Lambda}
	If $G$ is a topological group with continuous $\Lambda$-action such that $\Delta$ acts trivially, then $\res\colon H^1_{\cts}(\Lambda,G)\to  H^1_{\cts}(\Delta,G)^Q$
	has a natural section sending $\rho$ to $\sec(\rho)\colon(\gamma,\sigma)\mapsto \rho(\gamma)$.
\end{Lemma}

\subsection{Higgs bundles on smoothoid spaces}
Having introduced differentials on smoothoid spaces, we obtain a notion of Higgs bundles:
\begin{Definition}
	Let $X$ be a smoothoid space over $K$. A \textbf{Higgs bundle} on $X$ is a pair $(E,\theta)$ consisting of a vector bundle $E$ on $X_{\et}$ and
	$\theta\in H^0(X,\End(E)\otimes \wtOm^1_X)$
	such that $\theta\wedge\theta=0$ in $\End(E)\otimes \wtOm^2_X$, where $\wedge$ was defined in \cref{eq:wedge-for-wtOm}. Explicitly, this means that for any local basis of $\wtOm^1_X$, the coefficients of $\theta$ commute. Any such $\theta$ is called a \textbf{Higgs field} and can be written as a map $\theta:E\to E\otimes \wt\Omega^1$. A morphism of Higgs bundles is a morphism of the underlying bundles that commutes with the Higgs fields written in this form.
\end{Definition}

\begin{Definition}\label{d:Higgs-cohom}
	Let $X$ be a smoothoid space and let $(E,\theta)$ be a Higgs bundle on $X$. Again using the wedge product from \cref{eq:wedge-for-wtOm}, we can define the associated \textbf{Higgs complex}
	\[ \mathcal C^\ast_{\mathrm{Higgs}}(E,\theta):=\big [E\xrightarrow{\theta}E\otimes \wtOm^1\xrightarrow{\theta_1}E\otimes \wtOm^2\xrightarrow{\theta_2}\dots \xrightarrow{\theta_{n-1}} E\otimes \wtOm^n\big]\]
	where the transition maps are defined as
	$\theta_k:E\otimes \wtOm^k\xrightarrow{\theta\otimes\id } E\otimes \wtOm^1\otimes\wtOm^k\xrightarrow{\id\otimes \wedge} E\otimes \wtOm^{k+1}$.
	Passing to  the image of $\mathcal C^\ast_{\mathrm{Higgs}}(E,\theta)$ in the derived category $\mathcal D^b(X_{\et})$, we then set \[R\Gamma_{\mathrm{Higgs}}(X,(E,\theta)):=R\Gamma_{\et}(X,\mathcal C^\ast_{\mathrm{Higgs}}(E,\theta)).\]
	We call $H^k(X,(E,\theta)):=H^k(R\Gamma_{\Higgs}(X,(E,\theta)))$
	the \textbf{Dolbeault cohomology} of $(E,\theta)$.
\end{Definition}
\section{Rigid groups as $v$-sheaves}
In this short section, we briefly collect some background on rigid analytic group varieties:
\begin{Definition}
	A \textbf{rigid group} is a group object in the category of rigid spaces over $K$.
\end{Definition}

Since $K$ has characteristic $0$, a rigid group is automatically smooth \cite[Proposition~1]{Fargues-groupes-analytiques}. 

\begin{Example}
	\begin{enumerate}
		\item 	Any algebraic group over $K$ defines a rigid group via analytification.
		\item If $\mathcal G$ is a formal group scheme of topologically finite type over $K^+$, its adic generic fibre is a rigid group. A rigid group of this form is said to have \textbf{good reduction}.
	\end{enumerate}
\end{Example}
As always, we identify $G$ with the associated diamond, an abelian  $v$-sheaf on $\Perf_K$. 

\subsection{The Lie algebra and its exponential}
We now recall some basic constructions on rigid groups, and refer to \cite[\S3]{heuer-G-torsors-perfectoid-spaces} for more details.

Rigid groups are the non-archimedean analogues of complex Lie groups. One way in which this analogy manifests itself is by the $p$-adic Lie algebra Lie group correspondence:

Let $G$ be any rigid group, then its tangent space at the identity inherits the structure of a Lie algebra $\mathfrak g$ over $K$ of dimension $\dim_K\mathfrak g=\dim G$. We consider this as a $v$-sheaf via the associated vector group over $K$, i.e.\ we set $\mathfrak g(Y)=\mathfrak g\otimes_K \O(Y)$ on $\Perf_{K,v}$. Then there is a natural adjoint action
$\ad:G\to \End(\mathfrak g)$.
Sending a rigid group to its associated Lie algebra defines an equivalence of categories after localising at the class of open subgroups. Moreover, as for complex Lie groups, there is an exponential relating $G$ and $\mg$:

\begin{Proposition}[{\cite[Proposition 3.5]{heuer-G-torsors-perfectoid-spaces}}]\label{p:exp-on-Lie-alg}
	There is an open subgroup $\mg^\circ\subseteq \mg$, with underlying rigid space is isomorphic to $\B^d$, and a rigid open subgroup $G^\circ\subseteq G$ with an isomorphism  	of rigid spaces 
	\[ \exp:\mg^\circ\isomarrow G^\circ\]
that sends open subgroups onto open subgroups, and is functorial in $G$.
\end{Proposition}
\begin{Corollary}\label{c:open-subgroup-of-good-reduction}
Any rigid group $G$ has a neighbourhood basis $(U_k)_{k\in\N}$  of $1$ of open subgroups $U_k\subseteq G$ of good reduction whose underlying rigid space is isomorphic to $\B^d$.
\end{Corollary}
\begin{Lemma}[{\cite[Lemma~3.10]{heuer-G-torsors-perfectoid-spaces}}]\label{l:exp-log-commutativity}
	Let $Y$ be any adic space over $K$.
	\begin{enumerate}
		\item If  $A,B\in \mathfrak g^\circ(Y)$ satisfy $[A,B]=0$, then $\exp(A)$ and $\exp(B)$ commute and 
		\[\exp(A+B)=\exp(A)\exp(B).\]
		\item If $g,h\in G^\circ(Y)$ commute, then $[\log(g),\log(h)]=0$ and  $\log(gh)=\log(g)+\log(h)$.	
		\item If $\mathfrak g_1\subseteq\mathfrak g_2\subseteq \mg^\circ$ with images $G_1\subseteq G_2\subseteq G$ under $\exp$ and $g\in G_1(Y)$ are such that $\ad(g)(\mathfrak g_{1})\subseteq \mathfrak g_{2}$, then $g^{-1}G_{1}g\subseteq G_2$. For $A\in \mg_{1}(Y)$, we then have
		\[ \exp(\ad(g)(A))=g^{-1}\exp(A)g.\]
	\end{enumerate}
\end{Lemma}
\begin{Proposition}[{\cite[\S4.2]{heuer-G-torsors-perfectoid-spaces}}]\label{p:inductive-lifting-ses-for-G}
	If $G$ has good reduction, then the Lie algebra of the formal model induces a finite free $\O^+$-submodule $\mg^+\subseteq \mg$.  Let $\mg_k^+:=p^k\cdot\mathfrak m \cdot \mg^+$, then there is $\alpha>0$ such that $\exp$ is defined on $\mg_k^+$ for any $k>\alpha$, and $G_k:=\exp(\mg_k^+)$ is an open subgroup of $G$ such that for any  $\alpha<r<s\in \Q$ with $s\leq 2r-\alpha_0\in \Q$, the exponential induces isomorphisms \[\exp:\mg_{r}^+/\mg_{s}^+\isomarrow G_{r}/G_{s},\quad  \exp:\mg_{r}^+(X)/\mg_{s}^+(X)\isomarrow G_{r}(X)/G_{s}(X)\]
	of abelian sheaves on $\Smd_{K,\et}$, respectively of abelian groups for any $X\in \Smd_K$. 
\end{Proposition}
The $G_k$ can be described as the kernel of the reduction mod $p^k\mathfrak m$. In particular:

\begin{Lemma}[{\cite[Lemma 4.17]{heuer-G-torsors-perfectoid-spaces}}]\label{l:completeness-of-G}
If $G$ has good reduction, then we have $G=\varprojlim_{k\in \N} G/G_k$.
\end{Lemma}
\subsection{Torsors under rigid groups}
We now recall the definition of $G$-torsors on diamonds, which we studied in \cite[\S3]{heuer-G-torsors-perfectoid-spaces}.

\begin{Definition}
	Let $X$ be any diamond and let $\tau\in \{\et,v\}$. Then a \textbf{$G$-torsor} on $X_{\tau}$ is a sheaf $F$ on $X_{\tau}$ with a left action $m:G\times F\to F$ by $G$ such that there is a $\tau$-cover of $X'\to X$ where there is a $G$-equivariant isomorphism $G\times_XX'\isomarrow F\times_XX'$. A morphism of $G$-bundles on $X_{\tau}$ is a $G$-equivariant morphism of sheaves on $X_{\tau}$.
\end{Definition}
It is clear  that the set of isomorphism classes of $G$-torsors on $X_{\tau}$ is given by $H^1_{\tau}(X,G)$. 
\begin{Remark}\label{r:GLn-torsors-vs-vector-bundles}
Any morphism of $G$-torsors is an isomorphism. So the category of $\GL_n$-torsors has the same objects as  the category of vector bundles of rank $n$, but fewer morphisms.
\end{Remark}

\begin{Lemma}[{\cite[Proposition 3.16]{heuer-G-torsors-perfectoid-spaces}}]\label{l:fully-faithful-functor-torsors}
	There is a natural fully faithful functor
	\[ \{\text{$G$-torsors on $X_{\et}$}\}\hookrightarrow \{\text{$G$-torsors on $X_{v}$}\}\]
	which on isomorphism classes induces the natural map $H^1_{\et}(X,G)\to H^1_v(X,G)$.
\end{Lemma}

Finally, we recall the main technical results from \cite{heuer-G-torsors-perfectoid-spaces} about $G$-torsors on adic spaces:

\begin{Lemma}[{\cite[Lemma~4.26]{heuer-G-torsors-perfectoid-spaces}}]\label{l:small-bundles-on-perfectoid}
	If $X$ is an affinoid perfectoid space and $G$ is a rigid group of good reduction. Then for any $k>\alpha$, we have $H^1_v(X,G_k)=1$.
\end{Lemma}

\begin{Proposition}[{\cite[
		Proposition 4.8]{heuer-G-torsors-perfectoid-spaces}}]\label{p:reduction-of-structure-group}
	Let $G$ be a rigid group and let $U\subseteq G$ be a rigid open subgroup. Let $X$ be any sousperfectoid space and let $\nu:X_v\to X_{\et}$ be the natural map. Then the morphism
	\[R^1\nu_{\ast}U\to R^1\nu_{\ast}G\]
	is surjective. If $G$ is commutative, then
	$R^k\nu_{\ast}U= R^k\nu_{\ast}G$
	for all $k\geq 1$.
\end{Proposition}

\subsection{$G$-Higgs bundles}
Like in complex geometry, one can generalise the notion of Higgs bundles on smoothoid spaces from vector bundles to $G$-bundles for any rigid analytic group $G$ over $K$.

\begin{Definition}
	Let $X$ be a smoothoid space.
	For any $G$-bundle $E$ on $X_{\et}$, one defines
	the \textbf{adjoint bundle} of $E$ to be
	$\ad(E):=\mathfrak g\times^GE$, the associated bundle with respect to the adjoint action $\ad:G\to \GL(\mg)$.
	This has the natural structure of a vector bundle on $X_{\et}$.
\end{Definition}
\begin{Example}
		If $G=\GL_n$, then $\ad(E)=\End(E)$ is the endomorphism bundle.
		
		If $E=G$ is the trivial bundle, then $\ad(E)=\mathfrak g$. This holds for any $E$ if $G$ is commutative.
\end{Example}
Since the adjoint action of $G$ commutes with the Lie bracket on $\mathfrak g$, we obtain by functoriality a Lie bracket on $\ad(E)$. We can use this to define a natural map
\begin{eqnarray*}
	\wedge:\ad(E)\otimes \wtOm^1_X&\to &\ad(E)\otimes \wtOm^2_X,\\ \textstyle\theta=\sum_{i=1}^n A_i\otimes \delta_i&\mapsto& \theta\wedge\theta:=\textstyle\sum_{i<j} [A_i,A_j]\otimes \delta_i\wedge \delta_j.
\end{eqnarray*}

\begin{Definition}\label{d:G-Higgs}
	Let $X$ be a smoothoid space over $K$.
	\begin{enumerate}
		\item 
		A \textbf{$G$-Higgs bundle} on $X$ is a pair $(E,\theta)$ of a $G$-bundle $E$ on $X_{\et}$ and an element $\theta\in H^0(X,\ad(E)\otimes \wtOm_X)$ such that $\theta\wedge \theta=0$. Such $\theta$ are called \textbf{Higgs fields}.
		\item We denote by ${\mathrm{Higgs}}_G$ the sheafification of the presheaf of pointed sets on $X_{\et}$ of isomorphism classes of $G$-Higgs bundles. Using \cref{l:wtOm-is-sheaf-on-big-etale-site}, there is for any morphism of smoothoids $f':X'\to X$ a natural pullback map $\Higgs_G(X)\to \Higgs_G(X')$. We can therefore also regard $\Higgs_G$ as a sheaf on the big \'etale site $\Smd_{K,\et}$.
	\end{enumerate}
\end{Definition}
\begin{Remark}\label{r:higgs-for-commutative-G}
	If $G$ is commutative, then there is no interrelation between $E$ and $\theta$, and the Higgs field condition is vacuous. Therefore $\theta$ is in this case simply any section of $\mg \otimes \wtOm$.
\end{Remark}

\begin{Lemma}\label{l:sheaf-of-Higgs-is-sheaf-of-Higgs-on-trivial}
	There is a natural isomorphisms of sheaves of pointed sets
	\[ {\mathrm{Higgs}}_G=(\mathfrak g \otimes \wtOm_X)^{\wedge=0}/G\]
	given by interpreting the right hand side as a Higgs field on the trivial bundle. Here on the right we form the sheaf quotient by the adjoint action of $G$.
\end{Lemma}
\begin{proof}
	The map from right to left is injective for presheaves, thus also after the sheafification. It is surjective since for any $G$-Higgs bundle $(E,\theta)$, the $G$-bundle $E$ is trivial \'etale-locally.
\end{proof}
\begin{Example}
	Explicitly, for $G=\GL_n$, we have $\Higgs_G=( M_n(K)\otimes_K\wtOm_X)^{\wedge=0}/\GL_n$.
\end{Example}

The notion of Higgs bundles is functorial in $G$, namely for any homomorphism $\varphi:G\to G'$ of rigid groups there is a functor from $G$-Higgs bundles to $G'$-Higgs bundles defined by sending $(E,\theta)$ to $(G'\times^GE,\mathfrak g'\times^G\theta)$. This defines a morphism of sheaves
$\Higgs_G\to  \Higgs_{G'}$.

\begin{Example}\label{ex:cong-over-Q_p-but-not-Z_p}
	If $G\subseteq G'$ is an open subgroup, then this morphism
	is clearly surjective. But it might not be injective:
	For $X=\Spa(K\langle T\rangle)$, the Higgs fields $A_1dT$ and $A_2dT$ on $E=\O^2$ for 
	$A_1:=\left(\begin{smallmatrix} 1 & 1 \\ 0 & 1\end{smallmatrix}\right)$ and $A_2:=\left(\begin{smallmatrix} 1 & p \\ 0 & 1\end{smallmatrix}\right)$
	are conjugated over $G'=\GL_2(\O)$ via $\left(\begin{smallmatrix} p & p \\ 0 & 1\end{smallmatrix}\right)$, but not over $G=\GL_2(\O^+)$. Hence they are different elements in the same fibre of $\Higgs_G\to  \Higgs_{G'}$.
\end{Example}
Nevertheless, if we just consider the kernel, i.e.\ the fibre over $0$, we do have the following:

\begin{Lemma}\label{l:Zariski-dense-restriction-on-Higgs}
	\begin{enumerate}
		\item 
		If $f:X'\to X$ is a morphism of smoothoid spaces such that the map $f^{\ast}\colon \wtOm_X\to f_{\ast}\wtOm_{X'}$ on $X_{\et}$ is injective, then
		$\Higgs_G(X)\to \Higgs_{G}(X')$
		has trivial kernel. 
		\item If $\varphi:G\rightarrow G'$ is a homomorphism of rigid groups over $K$ such that $\mg\to \mg'$ is injective (e.g.\ if $\varphi$ is injective), then $\Higgs_G(X)\to \Higgs_{G'}(X)$
		has trivial kernel.
	\end{enumerate}
\end{Lemma}
\begin{Example}\label{ex:pullback-wtOm-inj}
	\begin{enumerate}
		\item If $f:X'\to X$ is an \'etale morphism with Zariski-dense image, then $\wtOm_{X'}=f^{\ast}\wtOm_X$ and $\O_{X}\to f_{\ast}\O_{X'}$ is injective, hence the conditions of the lemma hold.
		\item Let $X$ be a smooth rigid space and let $g:Y'\to Y$ be a $v$-cover of perfectoid spaces. Set $f=(\id,g):X\times Y'\to X\times Y$. Then again $\wtOm_{X'}=f^{\ast}\wtOm_X$, and $\O_{X}\to f_{\ast}\O_{X'}$ is injective by the $v$-sheaf property. Thus the condition of the lemma holds.
	\end{enumerate}
\end{Example}
\begin{proof}
	We first observe that the surjective morphism of sheaves
	$(\mathfrak g \otimes \wtOm_X)^{\wedge=0}\to \Higgs_G$
	has trivial kernel since $x\in \mg$ is conjugated to $0$ via the adjoint action if and only if $x=0$.
	
	Let now $x\in \Higgs_G(X)$ be in the kernel. After passing to an \'etale cover $\wt X\to X$ with pullback $\wt X'\to X'$, we can assume that $x$ lifts to $\wt x\in(\mathfrak g \otimes \wtOm_X)(\wt X)$. Chasing the diagram
	\[\begin{tikzcd}
		(\mathfrak g \otimes \wtOm_X)^{\wedge=0}(\wt X) \arrow[d] \arrow[r] & \Higgs_G(\wt X) \arrow[d] \\
		(\mathfrak g \otimes \wtOm_X)^{\wedge=0}(\wt X') \arrow[r] & \Higgs_G(\wt X')
	\end{tikzcd}\] 
	in which the map on the left is injective by assumption, we see that $\wt x=0$, whence $x=0$.
	
	Part 2 can be seen similarly, using that $(\mathfrak g \otimes \wtOm_X)^{\wedge=0}\to(\mathfrak g' \otimes \wtOm_X)^{\wedge=0}$
	is injective.
\end{proof}

\section{From $G$-Higgs bundles to $v$-topological $G$-bundles}\label{s:Higgs-to-v}
With the preparations of the last section, we can now state our main result, the sheafified correspondence between $v$-$G$-bundles and $G$-Higgs bundles:

\begin{Theorem}\label{t:main-thm-for-both-O-and-O^+}
	Let $K$ be a perfectoid field over $\Q_p$. Let
	$X$ be a smoothoid space over $K$ (for example a smooth rigid space) and let $\nu:X_v\to X_{\et}$ be the natural morphism of sites.  Let $G$ be a rigid group over $K$, regarded as a sheaf of groups $G=G(\O)$ on $X_v$. Let $\mathfrak g$ be the Lie algebra.
	Then there is a canonical isomorphism of sheaves of pointed sets on $X_{\et}$
	\[\HTlog:R^1\nu_{\ast}G\isomarrow \Higgs_{G}\]
	which is functorial in $G$, $X$ and $K$. Here $\Higgs_G=(\mg\otimes \wtOm_X)^{\wedge =0}/ G$ (see \cref{d:G-Higgs}). 
	
	If $G$ is commutative, we more generally have for any $n\geq 1$ an isomorphism
	\[\HTlog:R^n\nu_{\ast}G\isomarrow \mg\otimes \wtOm^n_X.\]
\end{Theorem}

\begin{Remark}\label{r:alternative topologies-in-main-thm}
	\begin{enumerate}
		\item The notions of $G$-torsors agrees on $X_{v}$ and $X_{\qproet}$, as well as on $X_{\proet}$ if $X$ is smooth or perfectoid, by \cite[Corollary~4.29]{heuer-G-torsors-perfectoid-spaces}. It follows that in the hierarchy
		\[ X_{v}\to X_{\qproet}\to X_{\proet}\to X_{\et},\]
		replacing $X_v$ with $X_{\qproet}$ (or $X_{\proet}$) gives an equivalent formulation of \cref{t:main-thm-for-both-O-and-O^+}. 
		\item If $G=\GL_n$, then by \cref{l:small-adic-space-is-diamantine}.5 also the $G$-torsors on $X_{\et}$ and $X_{\an}$ are equivalent, so we could replace $\nu$ by the projection to $X_{\an}$. But this is not true for general $G$.
		
		\item As mentioned in the introduction, the case of $G=\G_a$ and smooth rigid $X$ recovers Scholze's result that $R^n\nu_{\ast}\O=\wtOm^n_X$ for any $n\geq 1$. In \cite{heuer-Picard-good-reduction}, we have studied the case of $G=\G_m$. For all other $G$, the result is new already for smooth rigid $X$.
		\item  If $X$ is perfectoid, then $\wtOm_X^n=0$ for all $n\geq 1$, hence $\Higgs_G=0$. So this case recovers \cite[Theorem~1.1]{heuer-G-torsors-perfectoid-spaces}, which said that $R^1\nu_{\ast}G=0$ in this case. That said, we note that the cited result is used in the proof of \cref{t:main-thm-for-both-O-and-O^+}.
		
		\item At least for commutative $G$, we can informally remember \Cref{t:main-thm-for-both-O-and-O^+} as saying that one can compute $R^1\nu_{\ast}G(\O)$ by a ``chain rule'' applied to the ``composition'' $G(\O)$, with $\mathfrak g$  interpreted as the ``derivative'' of $G$, and $\wtOm_X^1$ as the ``derivative'' of $\O$.
		\item 
		The functoriality in $X$ and $K$ means in other words that the morphisms $\HTlog$ for varying $X$ can be assembled to an isomorphism of sheaves on the big site $\Smd_{K,\et}$.
		\item That $K$ is perfectoid is necessary: For example, if $K$ is instead discretely valued, then already $\Spa(K)$ has many non-trivial $v$-vector bundles, as described  by Sen theory.
		\item Already for algebraic $G$ like $G=\GL_n$ the proof goes by considering analytic open subgroups of $G$, so the perspective of rigid groups is natural already in this case.
		\item Rigid groups are the $p$-adic analogues of complex Lie groups. However, in the context of the complex Simpson correspondence (see \cref{s:rel-cpx}), one usually only works with reductive Lie groups, and we are not aware of any version without this assumption.
	\end{enumerate}
\end{Remark}

The proof will take us two sections: In this section, we construct a canonical and functorial morphism
\[  \Psi:{\mathrm{Higgs}}_G\to R^1\nu_{\ast}G\]
by exponentiating cocycles. In the subsequent section, we show that $\Psi$  is an isomorphism.

\medskip

To prepare the construction, let $\mathfrak g^\circ\subseteq \mathfrak g$ be an open subgroup of the Lie algebra isomorphic as a rigid group to $\O^{+d}$ that admits an exponential map of rigid spaces
$\exp:\mg^\circ\to G$.
It will in the following be irrelevant how large this subgroup $\mathfrak g^\circ$ is, as long as it is open and thus satisfies $\mg= \cup_{k\in \N}  p^{-k}\mg^\circ$. 
It will be convenient to take $\mg^\circ$ small enough so that $\exp$ still converges on $p^{-1}\mg^\circ$.
By \cref{p:exp-on-Lie-alg}, the image of $\mg^\circ$ under $\exp$ defines a rigid open subgroup 
\[G^\circ\subseteq G.\]
For any affinoid adic space $X$ over $K$, the set $G^\circ(X)$ inherits the structure of a topological group in a canonical way, with underlying topological space homeomorphic $\O^{+d}(X)$ via $\exp$.

\subsection{The morphism $\Psi$ in the commutative case}

By way of motivation, we first consider the much simpler case that $G$ is a commutative rigid group: Apart from $\G_a$ and $\G_m$ and their open subgroups, examples for such $G$ include abelian/abeloid varieties and analytic $p$-divisible groups in the sense of Fargues \cite{Fargues-groupes-analytiques}.

 If $G$ is commutative, then $\exp$ is a homomorphism and restricts to an isomorphism of rigid groups
$\exp:\mg^\circ\to G^\circ$.
For any $m\geq 0$ we derive from this an isomorphism
\[ \exp:R^m\nu_{\ast}\mg^\circ\isomarrow R^m\nu_{\ast}G^\circ.\]
\begin{proof}[Proof of \cref{t:main-thm-for-both-O-and-O^+} for commutative $G$]
	We begin by observing that by  \cref{p:reduction-of-structure-group} applied to $\O^+\subseteq \O$, we have
	$R^m\nu_{\ast}\O^+=R^m\nu_{\ast}\O=\wtOm^m$.
	The projection formula then shows \[R^m\nu_{\ast}\mg^\circ=\mg^\circ\otimes R^m\nu_{\ast}\O^+=\mg\otimes \wtOm^m.\]
	Composing with $\exp$, we obtain from this a canonical and functorial morphism
	\begin{equation}\label{eq:Psi-commutative-case}
		\Psi_G:\mg\otimes \wtOm^m= R^m\nu_{\ast}\mg^\circ\xrightarrow{\exp} R^m\nu_{\ast}G^\circ \to R^m\nu_{\ast}G.
	\end{equation}
	By \cref{p:reduction-of-structure-group} applied to $G^\circ\subseteq G$, also the last morphism is an isomorphism.
\end{proof}
\cref{t:main-thm-for-both-O-and-O^+} for commutative $G$ is already interesting for smooth rigid $X$, where as a consequence, we get a generalisation of the Hodge--Tate spectral sequence, with $G$-coefficients:

\begin{Corollary}
	The Leray sequence for $X_v\to X_{\et}$ induces a first quadrant spectral sequence
	\[ E_2^{ij}:=\begin{dcases}\begin{rcases}
			H^i_{\et}(X,G)& \text{ if }j=0\\
			H^i_{\et}(X,\wtOm_X^j)\otimes_K\mathfrak g& \text{ if }j>0	
	\end{rcases}\end{dcases}\Rightarrow H^{i+j}_{v}(X,G).\]
\end{Corollary}

Of course this simple construction of $\Psi_G$ cannot work for general $G$ since $\exp$ is then not a homomorphism, and cannot be derived. The basic idea, going  back to Faltings \cite{Faltings_SimpsonI}, is to represent both sides as group cohomology and apply $\exp$ to carefully chosen cocycles.

\subsection{Preparations and choices}\label{s:preps-and-choices}

As long as we later prove that our construction is canonical and functorial in $X$, it suffices to construct $\Psi_G$ locally. We may therefore assume that $X=\Spa(R,R^+)$ is a toric affinoid smoothoid space in the  sense of Definition~\ref{d:toric-chart}. Note that we do not fix a toric chart.

We begin by making some auxiliary choices, we later prove that the construction is independent of these.
We first choose any affinoid perfectoid pro-finite-\'etale Galois cover 
\[\wt X=\Spa(\wt R,\wt R^+)\to X=\Spa(R,R^+).\]

\begin{Remark}\label{ex:choice-of-cover}
	In practice, there are various concrete ways to choose this:
	\begin{enumerate}
		\item If $\Q_p^\cyc\subseteq K$, we can choose a toric chart and consider the toric tower as in \cref{s:comp-toric-charts}.
		\item If $\Spec(\O(X))$ is connected, then any choice of base point $x\in X(\overline K)$ induces an affinoid perfectoid universal pro-finite-\'etale cover 
		$\wt X\to X$, defined as the inverse limit over Zariski-connected finite \'etale covers of $X$
		together with a fixed lift $\tilde{x}$ of $x$.
	\end{enumerate}
However, the greater generality is useful to show that the construction is canonical and functorial. It also allows for a uniform treatment independent of whether $\Q_p^\cyc \subseteq K$ or not.
\end{Remark}

Let us denote the Galois group of $\wt X\to X$ by $\pi$, then $\wt X\to X$
is a pro-finite-\'etale $\pi$-torsor. As $H^1_v(\wt X,\O)=0$,
the Cartan--Leray sequence \cref{p:Cartan--Leray} induces an isomorphism
\[ H^1_{\cts}(\pi,\wt R )=H^1_v(X,\O)=H^0(X,\wtOm_X).\]
Since $X$ is toric, $H^0(X,\wtOm_X)$ is finite free over $R$. Moreover,
by \cref{l:Sch13-4.5/5.5} we can up to bounded torsion identify $H^1_v(X,\O^+)$ with an $R^+$-sublattice of rank $d$ of $H^0(X,\wtOm)$. Thus
\[H^0(X,\wtOm)^{\circ}:=\im\big (H^1_{\cts}(\pi,\wt R^+ )\aeq H^1_v(X,\O^+)\to H^0(X,\wtOm)\big)\]
is an open $R^+$-submodule of $H^0(X,\wtOm)$ that generates the whole module upon inverting $p$.
We now choose a basis $\delta=(\delta_1,\dots,\delta_d)$ of $H^0(X,\wtOm)$ as an $R$-module that already lies in $ \mathfrak m\cdot H^0(X,\wtOm)^{\circ}$.
Let
\[H^0(X,\wtOm)^{+}\subseteq H^0(X,\wtOm)^{\circ}\]
be the finite free $R^+$-sublattice spanned by $\delta$. We also write this as $H^0(X,\wtOm)^{+,\delta}$ to indicate the dependence on $\delta$.
Third, we now choose for each $\delta_i$ a representative  $\rho_i:\pi\to \wt R^+$ in the set of continuous 1-cocycles $\mathcal Z^1_{\cts}(\pi,\wt R^+)$ that maps to $\delta_i$ under the map
\[\HT\circ [-]:\mathcal Z^1_{\cts}(\pi,\wt R^+)\xrightarrow{[-]} H^1_{\cts}(\pi,\wt R^+)\to H^1_v(X,\O^+)\xrightarrow{\HT}  H^0(X,\wtOm)^\circ.\]

\begin{Choices}\label{choices}
	
	In summary, we have made the following choices:
	\begin{enumerate}
		\item an affinoid perfectoid pro-finite-\'etale cover $\wt X\to X$ that is Galois with group $\pi$,
		\item a basis $\delta=(\delta_1,\dots,\delta_d)$ of $H^0(X,\wtOm)$,
		\item a set of representative $1$-cocycles $\rho_i\in \mathcal Z^1_{\cts}(\pi,\O^+(\wt X))$ such that $\HT([\rho_i])=\delta_i$.
	\end{enumerate}
\end{Choices}
Of course the $\rho_i$ determine the $\delta_i$, but it is later helpful to see this choice as two steps.
\begin{Remark}
	If $\Q_p^\cyc\subseteq K$ and $\wt X\to X$ is a toric cover, then by \cref{l:Sch13-4.5/5.5} one can always choose $\rho_i$ of the form 
	$\Delta=\pi\to R^+$.
	But the proof of independence of the toric chart would lead back to more general $\rho_i$, which also allows to treat more general perfectoid $K$ over $\Q_p$.
\end{Remark}

\subsection{The integral morphism $\Phi^+$: exponentiating cocycles}
Having made Choices~\ref{choices}, we now construct as the first step  a map
\[\Phi^+_{\mathrm{grp}}:(H^0(X,\wtOm)^{+}\otimes_{R^+}\mg^\circ(R))^{\wedge=0}\to H^1_{\cts}(\pi,G^\circ(\wt R)).\]
Let $\theta$ be an element on the left. Since $\mg^\circ$ is a finite free $\O^+$-module, $H^0(X,\wtOm)^{+}\otimes_{R^+}\mg^\circ(R)$ is a finite free $R^+$-module. Therefore $\theta$ has a unique expansion
$\theta=\sum_{i=1}^d \delta_i\otimes A_i$
in terms of the basis $\delta$ of the free $R^+$-module $H^0(X,\wtOm)^{+}$
for some $A_i\in \mg^\circ(R)$. The condition $\theta\wedge \theta=0$ means precisely that the $A_i$ commute with each other.
We now first define a map
\begin{eqnarray*}
	\wt \Phi^+_{\mathrm{grp}}:H^0(X,\wtOm)^{+}\otimes_{R^+} \mg^\circ(R)&\to& \Map_{\cts}(\pi,G^\circ(\wt R))\\
	\sum_{i=1}^d \delta_i\otimes A_i&\mapsto &\Big(\gamma \mapsto \prod_{i=1}^d \exp(\rho_i(\gamma)\cdot A_i)\Big).
\end{eqnarray*}
This is well-defined as $\rho_i$ has image in $\wt R^+$ and $A_i\in \mg^\circ(R)$, so their product lies in $\mg^\circ(\wt R)$.

The commutativity condition in \cref{l:exp-log-commutativity} is precisely the reason why the definition is only really sensible when we restrict to Higgs fields, where the matrices $A_i$ commute:

\begin{Lemma}
	If $\theta\wedge\theta=0$, then
	$\wt \Phi^+_{\mathrm{grp}}(\theta)$ is a $1$-cocycle. Hence $\wt \Phi^+_{\mathrm{grp}}$ induces a map
	\[\Phi^+_{\mathrm{grp}}:(H^0(X,\wtOm)^{+}\otimes_{R^+} \mg^\circ(R))^{\wedge=0}\to H^1_{\cts}(\pi,G^\circ(\wt R)).\]
\end{Lemma}
\begin{proof}
	Using that $\rho_i\in\mathcal Z^1_{\cts}(\pi,\wt R^+)$, we have by Definition~\ref{d:non-ab-1-cocycles} for any $\gamma_1,\gamma_2\in \pi$
	\[\wt{\Phi}^+_{\mathrm{grp}}(\theta)(\gamma_1\cdot \gamma_2)=\textstyle\prod_i \exp(\rho_i(\gamma_1)A_i+ \gamma_1^\ast\rho_i(\gamma_2)A_i).\]
		Since $\rho_i(\gamma_1)$ and $\gamma_1^\ast\rho_i(\gamma_2)$ are scalars in $\wt R^+$, any two elements of the subset of $\mg^\circ(\wt R^+)$ given by   $\rho_i(\gamma_1)A_i$ and $\gamma_1^\ast\rho_i(\gamma_2)A_i$ for $i=1,\dots,d$ still commute. We may thus write this as
	\[=\textstyle\prod_i \exp(\rho_i(\gamma_1)A_i) \cdot \gamma_1^{\ast}\prod_i \exp(\rho_i(\gamma_2)A_i)=\wt\Phi^+_{\mathrm{grp}}(\theta)(\gamma_1)\cdot \gamma_1^\ast\wt\Phi^+_{\mathrm{grp}}(\theta)(\gamma_2).\qedhere\]
\end{proof}

Via the map
$H^1_{\cts}(\pi,G^\circ(\wt R))\to H^1_v(X,G^\circ)$
of the Cartan--Leray sequence, we get a map
\[\Phi^+:(H^0(X,\wtOm)^{+}\otimes_{R^+} \mg^\circ(R))^{\wedge=0}\xrightarrow{\Phi^+_{\mathrm{grp}}} H^1_{\cts}(\pi,G^\circ(\wt R))\to  H^1_v(X,G^\circ)\to H^1_v(X,G).\]
If we want to indicate $X$, $G$ or the choice of $\delta$ and $\rho$, we add this as a subscript, e.g.\
$\Phi^+_{X,\delta}$.

\medskip

At this point, we have associated to any ``small'' $G$-Higgs field on the trivial bundle a ``small'' $v$-$G$-bundle. This construction is in fact functorial, but not independent of our choices. 
However, we will later prove that the construction becomes independent of choices after sheafifying.
As an intermediate step, we already need the following weaker statement:

\begin{Lemma}\label{l:indep-of-choice-of-cocycle-+}
	For each $i=1,\dots,d$, let $\rho_i'\in \mathcal Z^1_{\cts}(\pi,\wt R^+)$ be such that $[\rho_i']=[\rho_i]$ in $H^1_{\cts}(\pi,\wt R^+)\aeq H^1_v(X,\O^+)$. Then 
	$\Phi^+$ is the same whether it is computed using $\rho_i$ or $\rho_i'$.
\end{Lemma}
This is weaker than independence of the choice of $\rho_i$, as $\rho_i$ is by definition also a preimage of $\delta_i$ under the map
$\HT:H^1_v(X,\O^+)\to H^0(X,\wtOm)$
which is in general not injective.
\begin{proof}	
	That $[\rho_i]=[\rho_i']$ means that there is $x_i\in \wt R^+$ such that $\rho_i'(\gamma)=\gamma^\ast x_i+\rho_i(\gamma)-x_i$ for all $\gamma \in \pi$.
	Since $\gamma^\ast x$ and $x$ are scalars in $\wt R^+$, all the $\gamma^\ast xA_i$, $xA_i$ and $\rho'_i(\gamma)A_i$ in $\mg^\circ(\wt R)$  still commute with each other for all $i$.
	Consequently, for any $\gamma \in \pi$ we have
	\[\textstyle\prod_i \exp(\rho_i'(\gamma) A_i)=\gamma^\ast\big (\textstyle\prod_i\exp(x_i  A_i)\big)\cdot  \textstyle\prod_i \exp(\rho_i(\gamma) A_i)\cdot (\textstyle\prod_i \exp(x_i A_i))^{-1}.\]
	
	Setting $y:=\prod_i\exp(x_i  A_i)\in G^\circ(\wt R)$, we see that $\prod_i \exp(\rho_i(\gamma) A_i)$ and $\prod_i \exp(\rho'_i(\gamma) A_i)$ agree up to the $\gamma$-conjugation by $y$ and thus have the same image in $H^1_{\cts}(\pi,G^\circ(\wt R))$.
\end{proof}
We will frequently use the following simple observation to make cocycles small:
\begin{Lemma}\label{l:rescaling-rho}
	Let $0\neq a \in K^+$.
	Suppose that each $\rho_i$ has image in $a\wt R^+$. Then $a^{-1}\delta$ is a basis of $H^0(X,\wtOm)$ contained in the image of $H^1_v(X,\O^+)$, and $a^{-1}\rho$ is a choice of integral cocycle representatives for $a^{-1}\delta$. Then on $(H^0(X,\wtOm)^{+,\delta}\otimes_{R^+} \mg^\circ(R))^{\wedge=0}$ we have $\Phi^+_{a^{-1}\rho}=\Phi^+_{\rho}$,
	but $\Phi^+_{a^{-1}\rho}$ is now defined on the larger space $(H^0(X,\wtOm)^{+,\delta}\otimes_{R^+} a^{-1}\mg^\circ(R))^{\wedge=0}$.
\end{Lemma}
\begin{proof}
	Clear from the definition by writing $A_i\cdot \rho_i=aA_i\cdot a^{-1}\rho_i$.
\end{proof}

\subsection{Extension to all Higgs field}
Next, we extend $\Phi^+$ to a morphism on all of  $(H^0(X,\wtOm)\otimes \mg)^{\wedge=0}$. For this, we first note that already the case of $G=\G_m$ and $X=\B^2$ shows that  $\HTlog\colon H^1_v(X,\G_m)\to H^0(X,\wtOm)$ is in general not surjective, see \cite[\S6]{heuer-v_lb_rigid}. We can therefore only expect $\Phi^+$ to extend to all Higgs fields after \'etale sheafification on $X$, i.e.\ after passing from $H^1_v(X,G)$ to $R^1\nu_{\ast}G$.

We therefore now compose $\Phi^+$ with the map
from the Leray sequence of $\nu:X_v\to X_{\et}$
\[ \Psi^+:(H^0(X,\wtOm)^{+}\otimes \mg^\circ(K))^{\wedge=0}\xrightarrow{\Phi^+} H^1_v(X,G)\to  R^1\nu_{\ast}G(X).\]
As before, we denote this by $\Psi^+_{X}$ or $\Psi^+_{X,\delta}$ etc.\ if we want to indicate $X$, $G$, or our choices. 
This is the map that we now extend to all Higgs fields:
The basic idea for doing so is that every Higgs fields becomes small on a finite Galois cover, by the following two lemmas.

\begin{Lemma}\label{l:induced-choices-on-cover}
	If $\wt X\to X'\to X$ is a sub-cover such that $f:X'\to X$ is finite \'etale, then $f^{\ast}\wtOm_X=\wtOm_{X'}$ by \cref{l:cotangent-sequence-for-smoothoids}. Therefore Choices~\ref{choices} induce natural choices for $X'$: 
	\begin{enumerate}
		\item 
		For the Galois cover we use $\wt X\to X'$, whose Galois group is an open subgroup $\pi'\subseteq \pi$. 
		\item For the basis of $ H^0(X',\wtOm)$ we use the image of $\delta$ under $H^0(X,\wtOm)\to H^0(X',\wtOm)$.
		\item  For the representatives we use the restriction of the $\rho_i$ to $\pi'$.
	\end{enumerate}
\end{Lemma}
\begin{Lemma}\label{l:Higgs-field-becomes-small-on-cover}
	Let $k\in \N$. For any small enough normal open subgroup $\pi'\subseteq \pi$ corresponding to a finite \'etale Galois cover
	$f:X'\to X$, the induced cocycles $\rho_i$ from \cref{l:induced-choices-on-cover} have image in $p^k\O^+(\wt R)$. In particular, for any $\theta\in (H^0(X,\wtOm)\otimes_{K} \mg(K))^{\wedge=0}$, we can arrange that \[f^{\ast}\theta\in H^0(X',\wtOm)^{+,\delta'}\otimes_{K^+}\mg^\circ(K) \quad \text{ where }\delta':=p^{-k}\delta.\]
\end{Lemma}
\begin{proof}
	Since $\rho_i\colon\pi\to \O^+(\wt X)$ is continuous, $\rho_i^{-1}(p^k\O^+(\wt X))\subseteq \pi$ is an open neighbourhood of the identity. It therefore contains an open normal subgroup $\pi'\subseteq \pi$. This corresponds to a finite \'etale cover $X'\to X$ on which the restriction of $\rho_i$ to $\pi'$ maps into $p^k\O^+(\wt X)$.
\end{proof}
The idea is now to pass to such a cover $X'\to X$ and apply $\Psi_{X'}^+$ there.
In order to be able to use this to extend $\Psi_X^+$ to $(H^0(X,\wtOm)\otimes_{K} \mg(K))^{\wedge=0}$, we need to see that we can descend the image of $\Psi_{X'}^+$ back to $X$.
For this we use a first instance of functoriality of $\Psi$ in $X$:

\begin{Lemma}\label{l:Phi_U'^+-is-G-equivariant}
	Let $f:\wt X\to X'\to X$ be a subcover such that $X'\to X$ is finite \'etale and Galois with group  $Q$. Let $\delta'$ and $\rho'$ be the choices induced by \cref{l:induced-choices-on-cover}. Then the map $ \Phi_{X',\delta',\rho'}^+:(H^0(X',\wtOm)^{+}\otimes_{K^+}\mg^\circ(K))^{\wedge=0}\to H^1_v(X',G)$ is equivariant for the natural $Q$-actions on either side. In particular, so is $\Psi_{X'}^+$.
\end{Lemma}
\begin{proof}
	Write $\pi'$ for the Galois group of $\wt X\to X'$, then $Q=\pi/\pi'$. 
	Consider the restriction
	\[ \res:H^1_{\cts}(\pi,\wt R^+)\to H^1_{\cts}(\pi',\wt R^+)^Q,\]
	where we recall that the natural action of $Q$ on classes in $H^1_{\cts}(\pi',\wt R^+)$ is given for $g\in Q$ with any lift $\tilde g$ to $\pi$ by $[\rho]\mapsto g[\rho]:=[\tilde g\rho(\tilde g^{-1}-\tilde g)]$. Using  that $[\rho_i]=g[\rho_i]$ in $H^1_{\cts}(\pi',\O^+(\wt X))^Q$, we may  by~\cref{l:indep-of-choice-of-cocycle-+} use either of $\rho$ or $\tilde g\rho(\tilde g^{-1}-\tilde g)$ to compute $\Phi_{X'}^+$.
	
	Let $\theta =\sum \delta_i\otimes A_i$ be in the domain of $\Phi_{X',\delta',\rho'}^+$, then for any $\wt g\in \pi$ with image $g$ in $Q$,
	\begin{align*}
		{\Phi}_{X'}^+(g\theta)&=[\gamma\mapsto\textstyle \prod_i \exp(\rho_i(\gamma)\cdot gA_i)].\\
		&=[\gamma\mapsto \textstyle\prod_i \exp(\tilde g\rho_i(\tilde g^{-1}\gamma \tilde g)\cdot \tilde gA_i)]=g \Phi_{X'}^+(\theta)
	\end{align*}
	because $A_i\in \mg(X')\subseteq \mg(\wt X)$, so we can write $gA_i=\wt gA_i$ inside $\mg(\wt X)$.
\end{proof}

We are now prepared to define $\Psi$:
For any $\theta=\sum \delta_i\otimes A_i\in (H^0(X,\wtOm)\otimes_{K} \mg(K))^{\wedge=0}$, let $k\in \N$ be such that $p^kA_i\in \mg^\circ(X)$ for all $i$. Using Lemma~\ref{l:Higgs-field-becomes-small-on-cover} we find a Galois subcover $\wt X\to X'\to X$ with Galois group $Q$ where all $p^{-k}\rho_i$ become integral with respect to $p^{-k}\delta$, so that $\theta$ is in the domain of $\Psi_{X',p^{-k}\delta}^+$.
Then since $\theta\in (H^0(X',\wtOm)\otimes_K \mg(K))^Q$, Lemma~\ref{l:Phi_U'^+-is-G-equivariant} implies that
\[ \Psi_{X',p^{-k}\delta}^+(\theta)\in  (R^1\nu_{\ast}G(X'))^Q=R^1\nu_{\ast}G(X)\]
as $R^1\nu_{\ast}G$ is  a sheaf on $X_{\et}$.
Using Lemma~\ref{l:rescaling-rho}, we see that $\Psi_{X',p^{-k}\delta}^+(\theta)=\Psi_{X,\delta}^+(\theta)$ for any $\theta\in 	(H^0(X,\wtOm)^{+,\delta}\otimes_{K^+} \mg^\circ(K))^{\wedge=0}$, so this construction is compatible with the integral one.
Since any two finite \'etale sub-covers of $\wt X\to X$ have a common finite \'etale cover under $\wt X$, this also shows that for any two covers $X'\to X$ and $X''\to X$ on which $\theta$ becomes small enough, we have
$\Psi_{X'}^+(\theta)=\Psi_{X''}^+(\theta)$
inside of  $R^1\nu_{\ast}G(X)$.
Thus the following is well-defined:
\begin{Definition}
	In the colimit over finite \'etale sub-covers $\wt X\to X'\to X$, we obtain a map
	\[\Psi_{X}:H^0(X,\wtOm\otimes \mg)^{\wedge=0}\to R^1\nu_{\ast}G(X),\quad 
		\theta\mapsto \varinjlim_{X'\to X}\Psi^+_{X'}(\theta).\]
\end{Definition}

The following lemma summarises the above discussion of compatibility for varying $X'$:
\begin{Lemma}\label{l:check-independence-after-finite-etale-cover}
	Let $\wt X\to X'\to X$ be a finite \'etale sub-cover. Then using the induced choices of \cref{l:induced-choices-on-cover} for $X'$ to compute $\Psi_{X'}$,  the following diagram commutes:
	\[\begin{tikzcd}[row sep = {0.3cm,between origins}]
		H^0(X',\wtOm\otimes \mg)^{\wedge=0}	\arrow[r,"\Psi_{X'}"] &  R^1\nu_{\ast}G(X')\\\\
		\rotatebox[origin=c]{90}{$\quad\subseteq$}&
		\rotatebox[origin=c]{90}{$\quad\subseteq$}\\
		H^0(X,\wtOm\otimes \mg)^{\wedge=0}	\arrow[r,"\Psi_{X}"]  &  R^1\nu_{\ast}G(X).
	\end{tikzcd}\]
\end{Lemma}

\subsection{Independence of choice, and functoriality}
While our definition of $\Psi_X$ a priori depends on \cref{choices}, we now check that the result is canonical.
We begin by showing that $\Psi_X$ is independent of \cref{choices}.3, strengthening Lemma~\ref{l:indep-of-choice-of-cocycle}:

\begin{Lemma}\label{l:indep-of-choice-of-cocycle}
	$\Psi_X$ does not depend on the choice of integral $1$-cocycles $\rho_i$ such that  $[\rho_i]=\delta_i$.
\end{Lemma}
\begin{proof}
	Let $\rho_i'\in \mathcal Z^1_{\cts}(\pi,\wt R^+)$ be such that $[\rho_i]=[\rho_i']$ in $H^1_{\cts}(\pi,\wt R)$. Then there is $x\in \wt R$ with
	\begin{equation}\label{eq:rho-cmp-rho'}
		\rho_i'(\gamma)=\gamma^\ast x+\rho_i(\gamma)-x \quad \text{ for all $\gamma \in \pi$.}
	\end{equation}
	
	\begin{Claim}
		For any $\epsilon>0$, there is a finite \'etale subcover $\wt X\to X'\to X$ with Galois group $\pi'\subseteq \pi$ on which we can find $x\in \wt R$ such that \cref{eq:rho-cmp-rho'} holds for  $\gamma\in\pi'$ and such that $ p^{\epsilon}x\in\wt R^+$.
	\end{Claim}
	\begin{proof}
		Since \cref{eq:rho-cmp-rho'} implies $(\gamma-1)x\in \wt R^+$ for all $\gamma \in \pi$, there is $k\in \N$ such that $y:=p^kx\in \wt R^+$ satisfies $(\gamma-1)y\in p^k\wt R^+$ for all $\gamma$,
so the image $\overline{y}$ of $y$ in $H^0(\wt X,\O^+/p^k)$ is $\pi$-invariant.
		
		By \cite[Corollary~2.16]{heuer-G-torsors-perfectoid-spaces},
		$H^0(\wt X,\O^+/p^k)^{ \pi}=H^0(X,\O^+/p^k)$.
		The short exact sequences
		\[ 0\to \O^+(X')/p^k\to \O^+/p^k(X')\to H^1(X',\O^+)[p^k]\to 0\]
		in the colimit over all $\wt X\to X'\to X$ show that  
		$\varinjlim H^1(X',\O^+)[p^k]=H^1(\wt X,\O^+)[p^k]\aeq 0.$
		Hence there is $X'$ such that $p^{\epsilon}\overline{y}$ lifts to $\O^+(X')$. Then 
		$p^{\epsilon}y\in \O^+(X')+p^k\O^+(\wt X)$
		which implies $p^{\epsilon}x\in \O(X')+\O^+(\wt X)$. As $x$ is only determined by \cref{eq:rho-cmp-rho'} up to the difference of an element in $\O(X')$, we can thus change $x$ to arrange for $p^{\epsilon}x$ to be in $\wt R^+=\O^+(\wt X)$.
	\end{proof}
	By Lemma~\ref{l:check-independence-after-finite-etale-cover}, we may check the independence on $X'\to X$. Let now $\theta=\sum \delta_i\otimes A_i \in (H^0(X,\wtOm)^{+}\otimes_{R^+} \mg^\circ(R))^{\wedge=0}$, then by the claim, we can assume that $x\cdot A_i,\gamma^{\ast}x\cdot A_i\in p^{-\epsilon}\mg^\circ(\wt X)$ where $\exp$ still converges. Now the statement follows exactly as in Lemma~\ref{l:indep-of-choice-of-cocycle-+}.
\end{proof}

\begin{Lemma}\label{l:independence-of-Choice-of-Galois-cover}
	$\Psi_X$ does not depend on the choice of Galois cover $\wt X\to X$.
\end{Lemma}
\begin{proof}
	We first note that if $\wt X'\to \wt X\to X$ is any dominating pro-finite-\'etale Galois cover, then it is clear from functoriality of the Cartan--Leray sequence that the construction for $\wt X$ is compatible with that of $\wt X'$. Therefore, if  $\wt X_1\to X$ and $\wt X_2\to X$ are any two pro-finite-\'etale Galois cover such that there exists a pro-finite-\'etale Galois cover $\wt X_3\to X$ that dominates both $\wt X_1$ and $\wt X_2$, then  the construction for either agrees with that for $\wt X_3$.
	It now suffices to observe that such a simultaneous dominating cover always exists in the case that $\wt X_2\to X$ is the universal cover from \cref{ex:choice-of-cover}.2: But here any finite \'etale cover of $\wt X_2$ is split, hence again Galois over $X$, from which it is clear that we can find such a cover $\wt X_3$.
\end{proof}
\begin{Remark}\label{r:indep-of-base-point-is-subtle}
	Given that the proof is easy, and that $\wt X\to X$ could simply be induced by a choice of base-point, part 1 may look like the most innocuous part of \cref{choices}. However, it is this choice that keeps us from upgrading the construction to a \textit{functor} from all Higgs bundles to $v$-bundles: The crucial point is that there is no \textit{canonical} choice of  $\wt X_3$ in the above proof: Suppose we are given two base points $x_1$ and $x_2$ of $X$. Let $\wt X_1\to X$ and $\wt X_2\to X$ be the associated pro-finite-\'etale universal covers from \cref{ex:choice-of-cover}.2.
	Any choice of lift $\tilde{x}_2$ of $x_1$ to $\wt X_2$ induces an isomorphism $\phi:\wt X_1\isomarrow \wt X_2$ over $X$ sending $\tilde x_1$ to $\tilde x_2$, and an isomorphism of Galois groups
	$\pi_1(X,x_1)\isomarrow \pi_1(X,x_2)$ for which $\phi$ is equivariant.
	Let $\rho_i:\pi_1(X,x_2)\to \O^+(\wt X_2)$ be any representative cocycles with respect to $x_2$ for the given basis $\delta_i$. Then $
		\phi^{\ast}\rho_i:=[\gamma\mapsto \phi\circ\rho_i(\phi\circ\gamma\circ \phi^{-1})]$
	is a cocycle in $\mathcal Z^1_{\cts}(\pi_1(X,x_1),\O^+(\wt X_1))$ that represent the $\delta_i$ with respect to $x_1$. For this choice, we have $\Phi_{X,\phi^{\ast}\rho,x_1}^+=\phi^{\ast}\circ\Phi_{X,\rho,x_2}^+$ by the Lemma.
	However, the explicit cocycle $\phi^{\ast}\rho_i$ depends on the choice of $\phi$, and thus on $\wt{x}_2$. Only the associated element in group cohomology is independent of this choice.
\end{Remark}

\begin{Lemma}\label{l:indep-of-choice-of-basis}
	$\Psi_X$ is independent of the choice of basis $\delta$ of $H^0(X,\wtOm)$, \cref{choices}.2.
\end{Lemma}
\begin{proof}
	Let $\delta'=(\delta'_1,\dots,\delta'_d)$ be any other basis lying in $H^0(X,\wtOm)^{\circ}$. Choose $\rho'_j\in\mathcal Z^1_{\cts}(\pi,\wt R^+)$ representing the $\delta'_j$.
	Then $\delta_i=\sum_{j}b_{ij}\delta'_j$ for some $b_{ij}\in \O(X)$, and
	$\sum_i \delta_i\otimes A_i=\sum_j\delta'_j\otimes A_j'$
	for $A_j':=\sum_{i}b_{ij}A_i$.
	To see that $\Psi_{X,\delta}=\Psi_{X,\delta'}$, we may by \cref{l:check-independence-after-finite-etale-cover}  pass to a finite \'etale cover $X'\to X$,  so by \cref{l:Higgs-field-becomes-small-on-cover} we may assume that $b_{ij}\rho_j'\in \mathcal Z^1_{\cts}(\pi',\wt R^+)$ for all $i,j$. Then $\rho_i:=\sum_{j}b_{ij}\rho'_j\in \mathcal Z^1_{\cts}(\pi',\wt R^+)$ represents $\delta_i$, so by \cref{l:indep-of-choice-of-cocycle} we may use it to compute $\Psi_{X',\delta}$. 
	Now since $b_{ij}\rho'_j(\gamma) A_i\in \mg^\circ(\wt R)$ for any $\gamma$ in the Galois group of $\wt X\to X'$, we have
	\[\textstyle\Phi^+_{\rho,\delta}(\theta)(\gamma)=\prod_i \exp(\rho_i(\gamma) A_j)
		= \prod_i \exp(\sum_{j}b_{ij}\rho'_j(\gamma) A_i)
	=\prod_j \exp(\rho'_j(\gamma)  A_i')
		=\Phi^+_{\rho',\delta'}(\theta)(\gamma).\]
	It follows that the same is true for $\Psi^+_{X'}$ and thus for $\Psi_X$.
\end{proof}
In summary, Lemmas \ref{l:indep-of-choice-of-cocycle}, \ref{l:independence-of-Choice-of-Galois-cover} and \ref{l:indep-of-choice-of-basis}  show  that $\Psi_X$ is independent of \cref{choices}, hence canonical. Next, we show functoriality in $X$:
\begin{Lemma}\label{l:functoriality-of-Psi}
	For any morphism $f\colon X_2\to X_1$ of toric smoothoid spaces, the following diagram commutes:
	\[ \begin{tikzcd}
		{H^0(X_1,\wtOm\otimes \mg)^{\wedge=0}} \arrow[r,"\Psi_{X_1}"] \arrow[d,"f^{\ast}"]&  { R^1\nu_{\ast}G(X_1)} \arrow[d,"f^{\ast}"] \\
		{H^0(X_2,\wtOm\otimes \mg)^{\wedge=0}} \arrow[r,"\Psi_{X_2}"] & { R^1\nu_{\ast}G(X_2)}
	\end{tikzcd}\]
\end{Lemma}
\begin{proof}
	 Choose a sequential pro-finite-\'etale Galois cover $\wt X_1=\varprojlim_{n\in\N} X_{1,n}\to X_1$ with group $\pi_1$ as in \cref{choices}. The pullback $\wt X_1\times_{X_1}X_2\to X_2$ is a pro-finite-\'etale $\pi$-torsor over $X_2$, but not necessarily perfectoid. But we can find a sequential affinoid perfectoid Galois cover $\wt X_2\to X_2$ with group $\pi_2$
	  that dominates $\wt X_1\times_{X_1}X_2$. We thus get a morphism $\wt f:\wt X_2\to \wt X_1$ over $f:X_2\to X_1$, and a quotient map
	$\phi:\pi_2\to \pi_1$
	with respect to which $\wt f$ is equivariant.
	By functoriality of the Cartan--Leray sequence, this induces a commutative diagram
	\begin{equation}\begin{tikzcd}\label{eq:functoriality-of-CL}
			{H^1_\cts(\pi_1,\O(\wt X_1))} \arrow[d, " f^{\ast}"] \arrow[r]& {H^1_v(X_1,\O)} \arrow[d, "f^{\ast}"]\arrow[r] & {H^0(X_1,\wtOm)}  \arrow[d, "f^{\ast}"]\\
			{H^1_{\cts}(\pi_2,\O(\wt X_2))} \arrow[r]& {H^1_v(X_2,\O)}\arrow[r] & {H^0(X_2,\wtOm)}
		\end{tikzcd}
	\end{equation}
	where for any 1-cocycle $\rho$ in the top left, we define $f^{\ast}\rho:\pi_2\xrightarrow{\phi} \pi_1\xrightarrow{\rho} \O(\wt X_1)\xrightarrow{\wt f^\ast} \O(\wt X_2)$.
	
	Let $\delta_1$ be any basis in $H^0(X_1,\wtOm)^{\circ}$ and $\delta_2$ any basis in $H^0(X_2,\wtOm)^{\circ}$, and choose representative integral cocycles $\rho_1$ and $\rho_2$.  Then we can find $b_{ij}\in \O(X_2)$ such that
	$f^{\ast}\delta_{1,i}=\sum_jb_{ij}\delta_{2,j}$.
	For any $\theta=\sum_{i}\delta_{1,i}\otimes A_i$ with $A_i\in \mg(X_1)$ set $A_j':=\sum_i b_{ij}f^{\ast}A_i$, then 
	$f^{\ast}\theta=\sum_j\delta_{2,j}\otimes A_j'$.
	
	Choose now a finite \'etale sub-cover $X'_1\to X_1$ such that we can take $A_i\in \mg^\circ(X_1')$. Let $X_2''\to X_2$ be the pullback. Then by Lemma~\ref{l:Higgs-field-becomes-small-on-cover}, we can find a finite \'etale sub-cover $\wt X_2\to X'_2\to X_2''$ of $X_2$ with Galois group $\pi'_2$ such that $b_{ij}\rho_{2,j}\in \mathcal Z^1_{\cts}(\pi_2',\O^+(\wt X_2))$ for all $i,j$. The composition
	$f':X_2'\to X_2'' \to X'_1$
	is an intermediate morphism over $f$ and under $\wt f$.
	
	By definition, we can now use $\Phi_{X'_1,\rho_1}^+$ to compute $\Psi_{X_1}(\theta)$: This yields
	\[ f'^{\ast}\Phi^+_{X'_1,\rho_1}(\theta)=[\gamma\mapsto \textstyle\prod_i\exp(f'^{\ast}\rho_{1,i}(\gamma)f^{\ast}A_i)].\]
	By the commutative diagram \cref{eq:functoriality-of-CL} (applied to $f'\colon X'_2\to X'_1$), the $1$-cocycle $f'^{\ast}\rho_{1,i}$, defined as the restriction of $f^{\ast}\rho_{1,i}$ to $\pi'_2\subseteq \pi_2$, represents $f'^{\ast}\delta_{1,i}\in H^0(X_2',\wtOm)$. The same is true for 
	$\rho_{2,i}':=\sum_{j}b_{ij}\rho_{2,j}\in \mathcal Z^1_{\cts}(\pi_2',\O(\wt X_2))$.
	It follows that there is $x_i\in \O(\wt X_2)$ such that
	\[ (f'^{\ast}\rho_{1,i})(\gamma)=\gamma^{\ast}x_i+\rho'_{2,i}(\gamma)-x_i \quad \text{  for all $\gamma\in\pi_2'$}.\]
	
	Let $k\in \N$ be such that $p^kx_i\in \O^+(\wt X_2)$ and thus $p^k\gamma^{\ast}x_i\in \O^+(\wt X_2)$ for all $\gamma$.
	After possibly increasing $X'_1\to X_1$ and $X'_2\to X_2$, which by Lemmas~\ref{l:Higgs-field-becomes-small-on-cover} and \ref{l:indep-of-choice-of-basis} allows us to replace $\delta_1$ by $p^{-k}\delta_1$, and thus $A_i$ by $p^kA_i$, we can then assume that $\gamma^{\ast}x_i\cdot f^{\ast}A_i \in \mg^\circ(\wt X_2)$ for all $\gamma$. Similarly, we can assume that the $b_{ij}f^\ast A_i$ are in $\mg^\circ(\wt X_2)$, and then so are the  $A'_j$.
	
	The lemma now follows from  the same calculation as in the previous lemmas: By the usual commutativity argument we see that inside $H^1_{\cts}(\pi_2',G^\circ(\wt X_2))$,  we have
	for any $\gamma\in \pi_2'$:
	\[\textstyle f'^{\ast}\Phi^+_{X'_1,\rho_1}(\theta)=\big[\gamma\mapsto\prod_i \exp(\rho'_{2,i}(\gamma)\cdot f^{\ast}A_i)\big]=  \big[\gamma\mapsto  \prod_j\exp(\rho_{2,j}(\gamma)\cdot A_j')\big]=\Phi^+_{X'_2,\rho_2}(f^{\ast}\theta)\]
 The functoriality of $\Psi$ follows in the colimit over sub-covers of $\wt X_{1}\to X_{1}$ and $\wt X_{2}\to X_{2}$.
\end{proof}
As a consequence of Lemma~\ref{l:functoriality-of-Psi}, we see that for any smoothoid $X$, the maps $\Psi_{X'}$ for the basis of toric $X'\to X$ in $X_{\et}$  glue to a morphism of sheaves of pointed sets on $X_{\et}$
\[ \Psi_X:(\wtOm\otimes \mg)^{\wedge=0}\to R^1\nu_{\ast}G.\]

\begin{Lemma}
	$\Psi_X$ factors  through the quotient by the adjoint action of $G$ on $(\wtOm\otimes \mg)^{\wedge=0}$.
\end{Lemma}
\begin{proof}
	It suffices to prove this locally. Let thus $X$ be any toric smoothoid space, let $B\in G(X)$, and let $\theta=\sum_{i=1}^d \delta_i\otimes A_i$. Let $k\in \N$ be large enough such that $p^{k}\ad(B)(A_i)\in \mg^\circ(X)$ for all $i$. After choosing $X'\to X$ large enough, we may by \cref{l:indep-of-choice-of-basis} replace $\delta$ by $p^{-k}\delta$ and thus $A_i$ by $p^kA_i$, so that $\ad(B)(A_i)\in  \mg^\circ(X)$. Then by \cref{l:exp-log-commutativity}.3,
	\[\Phi_{X',\rho}^+(\ad(B)(\theta ))(\gamma)= \textstyle\prod_i \exp(\rho_i(\gamma)\cdot \ad(B)(A_i ))=B	\Phi_{X',\rho}^+(\theta)(\gamma)B^{-1}\]
	for all $\gamma \in \pi'$, where $\pi'$ is the Galois group of $\wt X\to X'$.
	Since $B\in G(X)$, we have $\gamma^{\ast}B=B$ for all $\gamma \in \pi'$. If $B\in G^\circ(R)$, this shows that $\Phi_{X',\rho}^+(\ad(B)(\theta ))$ is equal to $\Phi_{X',\rho}^+(\theta)$ in $H^1_{\cts}(\pi',G^\circ(\wt X))$. In general, one easily verifies directly that the images in $H^1_v(X,G)$ agree: Namely, $B$ defines an isomorphism between the associated $v$-$G$-torsors on $X$.
\end{proof}
Using \cref{l:sheaf-of-Higgs-is-sheaf-of-Higgs-on-trivial}, this finally induces the desired morphism of sheaves on $X_{\et}$
\begin{equation}\label{eq:candidate-for-Psi}
\Psi:\Higgs_{G}\to R^1\nu_{\ast}G.
\end{equation}
The functoriality of $\Psi$ in $G$ is clear from functoriality of the exponential. In summary:

\begin{Proposition}\label{p:Psi-summary}
	The morphism
	$\Psi$ is canonical and functorial in $G$ and $X$.
\end{Proposition}
In particular, we can assemble the $\Psi$ to a canonical morphism on $\Smd_{K,\et}$. 
At this point, we have constructed a candidate for the isomorphism in \cref{t:main-thm-for-both-O-and-O^+}.

\section{From $v$-topological $G$-bundles to $G$-Higgs bundles}\label{s:abelian-v-G-bundles}
Let $X$ be any smoothoid space over $K$. The goal of this section is to prove that the morphism $\Psi$ from \cref{eq:candidate-for-Psi} is an isomorphism. For this, we may again localise and assume that $X$ is toric.

Following Faltings in the case of $\GL_n$, the basic idea is as follows: Assume that $X$ admits a pro-finite-\'etale Galois cover $X_\infty\to X$ with some \textit{abelian} Galois group $\Delta$. Let now $\mathfrak u\subseteq \mg^\circ$  be any open subgroup and let $U\subseteq G^\circ$ be its image  under $\exp$. Then $U(X)\subseteq G^\circ(U)$ inherits from $\mg^\circ(U)$ the structure of a topological group. We consider the continuous homomorphisms
$\rho:\Delta\to U(X)$.
The inclusion $U(X)\subseteq  U(\wt X)$ and the Cartan--Leray map for $X_\infty\to X$ define natural maps
\[ \Hom_{\cts}(\Delta,U(X))\to H^1_{\cts}(\Delta,U(X_\infty))\to H^1_v(X,U)\to H^1_v(X,G)\]
that associate a $v$-$G$-bundle $V$ on $X$ to $\rho$.
On the other hand, the fact that $\Delta$ is abelian allows us to associate a Higgs bundle to $\rho$ by composing with the logarithm map
\[ \rho\mapsto [\log\circ\rho:\Delta\xrightarrow{\rho} U(X)\xrightarrow{\log}\mg(X)]\in H^1_{\cts}(\Delta,\mg(X))= H^0(X,\wtOm\otimes \mg).\]
By \cref{l:exp-log-commutativity}, as $\Delta$ is abelian, this satisfies the Higgs field condition, thus defines a Higgs field on the trivial $G$-bundle. This is the Higgs bundle that we would like to associate to $V$.

However, it is not clear  whether
any $v$-$G$-bundle arises from such an ``abelian'' cocycle  $\rho$ (surjectivity of $\Psi$), and whether this construction is independent of the choice of $\rho$  (injectivity of $\Psi$).
The goal of this section is to show that both hold after sheafification.

\subsection{Abelian $G$-bundles}\label{s:abelian-G-bundles}
In order to carry out this strategy, let us first assume that $\Q_p^\cyc\subseteq K$. We relax this condition in the next subsection. We fix a toric chart, by \cref{d:toric-chart} this is a  standard-\'etale map 
\[f:X\to \mathbb T^d\times Y\]
where $Y$ is an affinoid perfectoid space and $\mathbb T^d$ is some rigid torus. Using notation as in \cref{s:comp-toric-charts}, $f$ induces a perfectoid cover $X_\infty\to X$ which is an affinoid perfectoid $\Delta:=\Z_p^d(1)$-torsor due to the assumption on $K$. We write $X=\Spa(R,R^+)$ and $X_\infty=\Spa(R_\infty,R_\infty^+)$. 

\begin{Definition}\label{d:W_f}
	Let us write $\mathcal Z^{\ab}(U)$ for the sheaf on $X_{\et}$ that sends $Y\in X_{\et}$ to the set $\Hom_{\cts}(\Delta,U(\O(Y)))$. We call this the ``sheaf of abelian cocycles on $X_{\et}$''.
	Consider the map
	\[W_{f,U}:\mathcal Z^{\ab}(U)(X)\to H^1_{\cts}(\Delta,U(R))\to H^1_{\cts}(\Delta,U(R_\infty))\to  H^1_v(X,G).\]
	Since the action of $\Delta$ on $U(R)$ is trivial, the first map can be described as the quotient  with respect to the conjugation action by  $U(R)$. The third map is the Cartan--Leray map.
	
	Composing $W_f$ with the sheafification on $X$, we obtain a morphism of sheaves on $X_{\et}$
	\[ \mathcal W_{f,U}:\mathcal Z^{\ab}(U)\to R^1\nu_{\ast}G.\]
	It is clear that $W_{f,U}$ is functorial in $f$ and $U$ (for morphisms as in \cref{d:toric-chart}.2), thus so is $\mathcal W_{f,U}$. We will often drop the subscript $U$ from notation when it is clear from context.
\end{Definition}

We now explain how $\mathcal W_f$ allows us to define partial inverses of $\Psi$ on the image of $\mathcal W_f$:
Recall from \cref{l:HT-isom-on-small-dmd} that $f$ induces an isomorphism
$\HT_f:\Hom_{\cts}(\Delta,\O(X))\to H^0(X,\wtOm)$.
By tensoring with the Lie algebra, this induces an isomorphism of sheaves on $X_{\et}$
\[ \HT_f:\Hom_{\cts}(\Delta,\mg)\isomarrow \mg\otimes\wtOm.\]
Under $\HT_f$, the subsheaf of Higgs fields $(\mg\otimes\wtOm)^{\wedge=0}$ gets identified with the subsheaf 
\[\mathcal Z^{\ab}(\mathfrak g)\subseteq \Hom_{\cts}(\Delta,\mg)\]
consisting of those $\varphi:\Delta\to \mg(Y)$ for $Y\in X_\et$ with $[\varphi(x),\varphi(y)]=0$ for all $x,y\in \Delta$. We thus have an isomorphism
\[\HT_f:\mathcal Z^{\ab}(\mathfrak g)\isomarrow (\mg\otimes\wtOm)^{\wedge=0}.\]
Consider now the subsheaf $\mathcal Z^{\ab}(\mathfrak u)\subseteq \mathcal Z^{\ab}(\mathfrak g)$ of homomorphisms $\varphi$ with image in $\mathfrak u$.
On any such homomorphism $\varphi$, the exponential $\exp\circ \varphi$ converges, and the commutativity condition is by \cref{l:exp-log-commutativity} equivalent to $\exp\circ \varphi$ being a homomorphism. We thus have a map
\[\exp:\mathcal Z^{\ab}(\mathfrak u)\isomarrow	\mathcal Z^{\ab}(U)\]
which  has an inverse defined by $\log$. We will use this to compare $\Psi_G$ to the map $\W_f$.

\subsection{A non-cyclotomic variant}\label{s:ab-coc-non-cyc-setup}
In order to allow for more general $K$, we also need a non-cyclotomic variant: Assume that $\Q_p^\cyc \nsubseteq K$, but $\zeta_p\in K$. With notation as in \cref{s:non-cyclotomic}, there is then a $\Lambda$-torsor
\[ X^\cyc_\infty=\Spa(R^\cyc_\infty,R^{\cyc+}_\infty)\to X^\cyc=\Spa(R^\cyc,R^{\cyc+})\to X=\Spa(R,R^+)\]
given by base-change to the cyclotomic extension $K^\cyc|K$. \cref{l:section-for-Lambda} gives a natural section
\[ \sec:\Hom_{\cts}(\Delta,U(R^\cyc))^Q\to  H^1_{\cts}(\Lambda,U(R^\cyc_\infty))\]

\begin{Definition}\label{d:W_f-nc}
	Write $\mathcal Z^{\ab}(U)(X)$  for the set $\Hom_{\cts}(\Delta,U(R^\cyc))^Q$.
	Composing $\sec$ with the Cartan--Leray map for the cover $X^\cyc_\infty\to X$, we obtain a map $W_{f,U}$ defined as
	\[W_{f,U}:\mathcal Z^{\ab}(U)(X)\xrightarrow{\sec} H^1_{\cts}(\Lambda,U(R^\cyc_\infty))\to H^1_v(X,U)\to  H^1_v(X,G).\]
	Replacing $X$ by objects in $X_{\et}$, the  $\mathcal Z^{\ab}(U)$  define a sheaf of pointed sets $\mathcal Z^{\ab}(U)$ on $X_{\et}$, and
	as before, $W_{f,U}$ defines upon sheafification on $X_{\et}$ a morphism $\mathcal W_{f,U}:\mathcal Z^{\ab}(U)\to R^1\nu_{\ast}G$.
\end{Definition}

Second, for the Lie algebra, we observe that we have an isomorphism
\[ \HT_f:\Hom_{\cts}(\Delta,\mathfrak g(R^\cyc))^Q\isomarrow H^1_{\cts}(\Lambda,\mathfrak g(R^\cyc))\isomarrow \mathfrak g\otimes \wtOm,\]
this is clear for $\mg=\O$ and in general follows by tensoring with $\mg$.
As before, we then write $\mathcal Z^{\ab}(\mathfrak g)(X)\subseteq \Hom_{\cts}(\Delta,\mathfrak g(R^\cyc))^Q$
for those $\varphi:\Delta\to \mathfrak g(R^\cyc)$ such that $[\varphi(x),\varphi(y)]=0$ for all $x,y\in \Delta$.  This defines a sheaf $\mathcal Z^{\ab}(\mathfrak g)$ that $\HT_f$ identifies with $(\mathfrak g\otimes \wtOm)^{\wedge =0}$. Let
$\mathcal Z^{\ab}(\mathfrak u)\subseteq \mathcal Z^{\ab}(\mathfrak g)$
be the sheaf of those homomorphism with image in $\mathfrak u(R^\cyc)$. With these definitions, composing with the exponential again defines an isomorphism
\[ \exp:\mathcal Z^{\ab}(\mathfrak u)\isomarrow \mathcal Z^{\ab}(U).\]
As the  cocycles in $\mathcal Z^1_{\cts}(\Lambda,\mathfrak u(R^\cyc))$ and $\mathcal Z^1_{\cts}(\Lambda,U(R^\cyc))$ associated via $\sec$ still satisfy the commutativity condition, $\exp$ and $\log$ commute with passing from $\Delta$-cocycles to $\Lambda$-cocycles.

\subsection{Computing $\Psi_G$ with abelian cocycles}
Assume from now on that we are in either of the two setups of the last two subsections.
We can now use abelian cocycles depending on the choice of toric chart to compute $\Psi_G$:
\begin{Proposition}\label{p:master-diagram-Wf-vs-Psi}
	The following diagram commutes:
	\begin{equation}\label{eq:master-diagram-proof-main-thm}
		\begin{tikzcd}
			\mathcal Z^{\ab}(\mathfrak u) \arrow[d, "\HT_f"'] \arrow[r, "\exp","\sim"'] & \mathcal Z^{\ab}(U) \arrow[d, "\W_{f,U}"'] \\
			(\mg\otimes \wtOm)^{\wedge =0}/ G \arrow[r, "\Psi_G"] & R^1\nu_{\ast}G
		\end{tikzcd}
	\end{equation}
	The morphism $\HT_f$ on the left has trivial fibre over $0$.
\end{Proposition}

\begin{Remark}
	Diagram~\cref{eq:master-diagram-proof-main-thm} is the reason for the name $\HTlog$ for the inverse of $\Psi_G$.
\end{Remark}

The role of the morphism $\mathcal W_{f,U}$ is therefore that it allows us to define a partial inverse of $\Psi$ on the image of $\mathcal W_{f,U}$ by applying $\log$ to abelian cocycles $\rho\in \mathcal Z^{\ab}(U)$.

\begin{proof}
	We first assume that $\Q_p^\cyc\subseteq K$:
	 We may then compute $\Psi$ using the following choices: For the Galois cover, we take the toric cover $\wt X:=X_\infty\to X$. Fix an isomorphism $\Delta\cong \Z_p^d$, then $\Hom_{\cts}(\Delta,\O^+(X))= \O^+(X)^d$. 
	This induces an integral basis $\delta$ of $H^0(X,\wtOm)$ via the map $\HT_f$ from \cref{l:HT-isom-on-small-dmd}.
	Moreover, the left hand side defines canonical representatives $\rho_i\in \mathcal Z^1_{\cts}(\Delta,\O^+(X_\infty))$ of the $\delta_i$. By \cref{l:indep-of-choice-of-cocycle}, we can now use $\Phi_{X,\rho}^+$ to compute $\Psi_X^+$.
	
	With these choices, any $\theta=\sum \delta_i\otimes A_i\in (\mathfrak u\otimes \wtOm)^{\wedge =0}$ is sent by $\Phi_{X,\rho}^+$ to the cocycle
	\[\Delta\to U(X),\quad \gamma\mapsto \textstyle\prod_i \exp(\rho_i(\gamma)A_i)=\exp(\sum_i \rho_i(\gamma)A_i)\]
	The associated $G$-torsor is exactly $\W_{f,U}(\exp(\theta))$ by definition, so the square \Cref{eq:master-diagram-proof-main-thm} commutes.
	
	If $\Q_p^\cyc\nsubseteq K$,  we use the cover $X_\infty^\cyc\to X$ from \cref{s:ab-coc-non-cyc-setup}. By \cref{l:non-cyclotomic-version-of-Sch13-4.5/5.5}.3, we have
	$\Hom_{\cts}(\Delta,\O^+(R^{\cyc}))^Q\stackrel{a}{\cong} \O^+(R)^d$
	as $\O^+(R)$-modules. Tensor with $\m$ and let $\rho'_1,\dots,\rho'_d$ be the images of the standard basis of $p\O^+(R)^d$ on the left hand side, then via the map
	\[	\Hom_{\cts}(\Delta,\O^+(R^{\cyc}))^Q\xrightarrow{\sec}H^1_{\cts}(\Lambda,\O^+(R^{\cyc}_\infty))\to H^1_v(X,\O^+)\xrightarrow{\HT}H^0(X,\wtOm)^\circ,\]
	the images of  the $\rho_i'$ define an integral basis. Let $\rho_1,\dots,\rho_d$ be the images of the $\rho'_i$ in $H^1_{\cts}(\Lambda,\O^+(R^{\cyc}_\infty))$. Computing $\Psi_G$ for these choices proves the statement, as before.
\end{proof}
\subsection{Surjectivity of $\Psi$}
The map $W_{f,U}$ cannot in general be surjective as the map $H^1_v(X,U)\to H^1_v(X,G)$ is not in general surjective. However, it turns out that for $U$ small enough, this is the only obstruction, namely we now prove that any $v$-$G$-bundle on $X$ becomes abelian on an \'etale cover. For this  we adapt Faltings' method from \cite[\S2]{Faltings_SimpsonI}, using the technical means  prepared in \S2.
\begin{Proposition}\label{p:part-2-of-Faltings-Lemma}
	Assume that $\zeta_p\in K$. Let $G$ be a rigid group of good reduction over $K$. Then there is $c\geq 0$ such that with $U:=G_c$ as defined in \cref{p:inductive-lifting-ses-for-G}, the following hold for the map
	$W_f:=W_{f,U}$ from \cref{d:W_f} (if $\Q_p^\cyc\subseteq K$) or \cref{d:W_f-nc} (if $\Q_p^\cyc\nsubseteq K$):
	\begin{enumerate}
		\item $W_f$ has the same image as the natural map $H^1_v(X,G_c)\to H^1_v(X,G)$.
		\item Assume that $\Q_p^\cyc\subseteq K$. Then for any two homomorphisms $\varphi_1$,$\varphi_2:\Delta\to G_c(X)$ with $W_f(\varphi_1)=W_f(\varphi_2)$, there is $g\in G(X)$ such that $\varphi_1(\gamma)=g\varphi_2(\gamma)g^{-1}$ for all $\gamma\in \Delta$.
	\end{enumerate}
	Moreover, we can choose $c$ uniformly among all finite \'etale subcovers of $X_\infty \to X$.
\end{Proposition}
We will see at the end of \cref{s:construction-of-local-correspondence} that part 2 also holds for general $K$. 
\begin{proof}
	Let us first explain the argument in the simpler case that $\Q_p^\cyc\subseteq K$:
	Let $\alpha,\beta$ be as in \cref{l:Sch13-4.5/5.5} and let $\gamma=2\beta+\alpha$. We claim that any
	$c\geq 5\gamma$
	does the job. To see this, we can argue essentially as in \cite[Lemma~1.(i)]{Faltings_SimpsonI}, except that further care needs to be taken in our much more general setting, as the ``isomorphism up to torsion''-results are weaker.
	
	 It is clear from the construction that $W_f$ factors through the image of $H^1_v(X,G_c)$. Conversely, let $x\in H^1_v(X,G_c)$, then
	by \cref{l:small-bundles-on-perfectoid}, this $x$ becomes trivial in $H^1_v(X_\infty,G_c)$ and thus comes from an element in $H^1_{\cts}(\Delta,G_c(R_\infty))$ via the Cartan--Leray sequence. Let 
	\[ \rho:\Delta\to G_c(R_\infty)\subseteq G(R_\infty)\]
	be any $1$-cocycle representing $x$. In the following, for any $s\geq 0$ let us write $\rho_s$ for the image of $\rho$ in $H^1_{\cts}(\Delta,G/G_s(R_\infty))$. Then $\rho_c=1$.
	We consider for varying $s\geq c$ the natural map 
	\[ h_s:\Hom_{\cts}(\Delta,G(R)/G_s(R))\to H^1_{\cts}(\Delta,G/G_s(R_\infty))\]
	as well as for any $c\leq k<s$ the reduction map
	\[ r_k:\Hom_{\cts}(\Delta,G(R)/G_{s}(R))\to \Hom_{\cts}(\Delta,G(R)/G_k(R)).\]
	
	Set $t:=3\gamma$. We will inductively construct  a compatible system of continuous homomorphisms $\varphi_s:\Delta\to G(R)/G_s(R)$ for any $s\geq c$ such that $h_s(\varphi_s)=\rho_s$ and $r_t(\varphi_s)=1$.
	
	We start with $s=c$, for which we can take $\varphi_s:=1$ to be the trivial representation. 
	
	For the induction, we  claim that we can find a lift of the reduction $r_{s-2\gamma}(\varphi_s)$ to a homomorphism $\varphi'_{s+\gamma}:\Delta\to G(R)/G_{s+\gamma}(R)$ such that $h_{s+\gamma}(\varphi'_{s+\gamma})=\rho_{s+\gamma}$. Note that since $s\geq c\geq 5\gamma$, this still satisfies $r_{t}(\varphi'_{s+\gamma})=r_t(\varphi_{s})=1$, so that we can then lift $\varphi_s$ inductively.
	
	For the claim, the conditions imply $s+t<2s-\alpha_0$, so that by \cref{p:inductive-lifting-ses-for-G} we have a short exact sequence
	\[ 0\to \overline{\mg}^+_t\to G/G_{s+t}\to G/G_{s}\to 1,\]
	where $\overline{\mg}^+_t:=\mg_0^+/\mg_t^+$ is isomorphic to $(\m\O^+/p^t\m\O^+)^{\dim G}$. It therefore stays exact upon evaluation at the affinoid perfectoid $X_\infty$. For the same reason, $\overline{\mg}^+_t(R_\infty)={\mg}^+_0(R_\infty)/p^t$.
	
	On the other hand, by the second part of \cref{p:inductive-lifting-ses-for-G}, we also have an exact sequence
	\begin{equation}\label{eq:ses-2-in-proof-of-Faltings-lifting-lemma}
	0\to \mg^+_0(R)/p^t\to G(R)/G_{s+t}(R)\to G(R)/G_s(R)\to 1
	\end{equation}
	Taking continuous $\Delta$-cohomology of both sequences, we thus obtain a  commutative diagram
	\[ \begin{tikzcd}[column sep = 0.3cm]
		{ H^1_{\cts}(\Delta,{\mg}_0^+(R_\infty)/p^t)} \arrow[r] & { H^1_{\cts}(\Delta,G/G_{s+t}(R_\infty))} \arrow[r] & {H^1_{\cts}(\Delta,G/G_{s}(R_\infty))} \\
		{ \Hom_{\cts}(\Delta,{\mg}_0^+(R)/p^t)} \arrow[r]\arrow[u] & 	{ \Hom_{\cts}(\Delta,G(R)/G_{s+t}(R))} \arrow[r]\arrow[u,"h_{s+t}"'] & {\Hom_{\cts}(\Delta,G(R)/G_{s}(R))}.\arrow[u,"h_{s}"']
	\end{tikzcd}\]

	 By induction hypothesis, the image of $\rho_{s+t}$ in the top right has a preimage $\varphi_s$ under $h_s$. Hence there is $b=b_s\in G/G_{s}(R_\infty)$ such that
	$b^{-1}\cdot \rho_s(\gamma)\cdot \gamma^{\ast}b=\varphi_s(\gamma) $ for all $\gamma\in \Delta$.
	
	We now repeatedly use that by \cref{l:Sch13-4.5/5.5}, the map
	\begin{equation}\label{eq:H^i(R)->H^i(wt R)-in-lifting-claim}
		H^i_{\cts}(\Delta,{\mg}_0^+(R)/p^t)\to H^i_{\cts}(\Delta,{\mg}_0^+(R_\infty)/p^t)
	\end{equation}
	has $p^\gamma$-torsion kernel and cokernel for any $i\geq 1$.
	
	The bottom row of the above diagram is still left-exact, but not necessarily right-exact (it is not clear that one can lift the images to $G(R)/G_{s+t}(R)$ in such a way that they commute). However, as the term on the left in \cref{eq:ses-2-in-proof-of-Faltings-lifting-lemma} is an abelian group, the obstruction to lifting $\varphi_s$ to the middle of the bottom row is a class in $H^2_{\cts}(\Delta,{\mg}_0^+(R)/p^t)$ by \cref{p:non-ab-les}.3. This class is mapped to $0$ under the map  \cref{eq:H^i(R)->H^i(wt R)-in-lifting-claim} for $i=2$
	since $\rho_s$ does lift and the obstruction class is functorial. As the kernel of \cref{eq:H^i(R)->H^i(wt R)-in-lifting-claim} is killed by $p^\gamma$, it follows that after reducing the whole diagram mod $p^\gamma$, we can find a lift $\wt{\varphi}$ of $\varphi_s$ to the middle of the bottom row of the diagram
		\[ \begin{tikzcd}[column sep = 0.3cm]
		{ H^1_{\cts}(\Delta,{\mg}_0^+(R_\infty)/p^t)} \arrow[r] & { H^1_{\cts}(\Delta,G/G_{s+t-\gamma}(R_\infty))} \arrow[r] & {H^1_{\cts}(\Delta,G/G_{s-\gamma}(R_\infty))} \\
		{ \Hom_{\cts}(\Delta,{\mg}_0^+(R)/p^t)} \arrow[r]\arrow[u] & 	{ \Hom_{\cts}(\Delta,G(R)/G_{s+t-\gamma}(R))} \arrow[r]\arrow[u,"h_{s+t-\gamma}"'] & {\Hom_{\cts}(\Delta,G(R)/G_{s-\gamma}(R))}.\arrow[u,"h_{s-\gamma}"']
	\end{tikzcd}\]

	Let us denote by $\phi$ the image of $\wt{\varphi}$ under $h_{s+t-\gamma}$. Then by commutativity of the diagram, $\phi$ and $\rho_{s+t-\gamma}$ both define a lift of $\rho_{s-\gamma}$ (i.e.\ the image of $\rho_\gamma$ in the top right) along the top right map.
	By \cref{p:non-ab-les}.2, these differ by a class $\delta$ in 
	$H^1_{\cts}(\Delta,{}_{\phi}({\mg}_0^+(R_\infty)/p^t))$,
	where ${}_{ \phi}(\dots)$ denotes the module with the $\Delta$-action twisted by $\phi$:
	Explicitly, let $\wt b$ be any lift of $b$ to $G/G_{s+t-\gamma}(R_\infty)$. Then there is a unique cocycle $\delta:\Delta\to {}_{\phi}({\mg}_0^+(R_\infty)/p^t)$ such that
	\begin{equation}\label{eq:cocycle-battle}
		\wt b^{-1} \cdot \rho(\gamma) \cdot \gamma^{\ast}\wt b = \delta(\gamma)\cdot \phi(\gamma) \text{ in } G/G_{s+t-\gamma}(R_\infty)
	\end{equation}
	under the identification ${\mg}_0^+(R_\infty)/p^t=G_{s-\gamma}(R_\infty)/G_{s+t-\gamma}(R_\infty)$. Since $\rho\equiv 1 \bmod p^t$, we have $\phi\equiv 1 \bmod p^t$. This shows that on the level of cohomology sets, we have
	\[H^1_{\cts}(\Delta,{}_{\phi}({\mg}_0^+(R_\infty)/p^t))=H^1_{\cts}(\Delta,{\mg}_0^+(R_\infty)/p^t).\]
	Therefore, we can again use \cref{eq:H^i(R)->H^i(wt R)-in-lifting-claim}, this time for $i=1$: The image of $\delta$ in the cokernel of this map is annihilated by $p^{\gamma}$. We deduce that after reducing the whole diagram mod $p^{\gamma}$ once again, 
	the cocycle $\delta$ in the top left can be lifted to a class $\delta'$ in the bottom left.
	Explicitly, this means that there is $a\in {\mg}_0^+(R_\infty)/p^t= G_{s-2\gamma}(R_\infty)/G_{s+t-2\gamma}(R_\infty)$ such that
	\[\delta'(\gamma):=a^{-1}\cdot \delta(\gamma) \cdot \gamma \ast_{\phi}a=a^{-1}\cdot \delta (\gamma)\cdot \phi(\gamma)\cdot \gamma^{\ast}a \cdot \phi(\gamma)^{-1}\]
	defines a homomorphism $\delta':\Delta\to G_{s-2\gamma}(R)/G_{s+t-2\gamma}(R)$.
	Then
	\[\wt\varphi':=\delta'\cdot \widetilde\varphi:\Delta\to G(R)/G_{s+t-2\gamma}(R)\] is a homomorphism which also lifts $\varphi_{s-2\gamma}$ due to exactness of the bottom row, and whose image $\phi':=h_{s+t-2\gamma}(\wt\varphi')$ is equivalent to $\rho_{s+t-2\gamma}$ in	$H^1_{\cts}(\Delta,G/G_{s+t-2\gamma}(R_\infty))$: Namely, we can again be more explicit in terms of cocycles. By definition, we have for any $\gamma\in \Delta$
	\[\phi'(\gamma)=\delta'(\gamma)\cdot \phi(\gamma)= a^{-1}\cdot \delta (\gamma)\cdot \phi(\gamma)\cdot \gamma^{\ast}a \cdot \phi(\gamma)^{-1}\phi(\gamma)=
		\stackrel{\eqref{eq:cocycle-battle}}{=}(\wt ba)^{-1} \cdot \rho(\gamma) \cdot \gamma^{\ast}(\wt ba).\]
	Note that $b_{s+t-2\gamma}:=\wt ba$ is still a lift of the image $b_{s-2\gamma}$ of $b$ in $G/G_{s+t-2\gamma}(R_\infty)$ as $a$ is trivial mod $G_{s-2\gamma}(R_\infty)$. Using $s+t-2\gamma=s+\gamma$,
	this shows that $\varphi_{s+\gamma}:=\wt\varphi'$ defines a preimage of $\rho_{s+\gamma}$ under $h_{s+\gamma}$ with the desired properties, completing the induction step.
	
	\medskip
	
	In the limit $s\to \infty$, the $\varphi_{s}:\Delta\to G(R)/G_s(R)$ define by \cref{l:completeness-of-G} a homomorphism $\varphi:\Delta\to G(R)$, and the $b_s$ define an element $b\in G(R_\infty)$ such that
	\[ \varphi(\gamma)=b^{-1}\rho(\gamma)\gamma^{\ast}b.\]
	This shows that $\rho$ and $\varphi$ define the same class in $H^1_{\cts}(\Delta,G(R_\infty))$. Thus $W_f(\varphi)$ agrees with the image of $x$ in $H^1_v(X,G)$.
	This proves part 1 in the case that $\Q_p^\cyc\subseteq K$.
	
	\medskip
	
	To deduce part 2, we run the above argument with $\rho:=\varphi_1$ and $\wt\varphi:=\varphi_2$. Inductively, we then see that we may choose $b_s\in G(R)/G_s(R)$: For the induction start $c=s$ this is clear. In the induction step, we may then find $a\in \mg^+_0(R)/p^t$, and thus $\wt b\in G(R)/G_s(R)$ such that 
	$\varphi_{1,s}=\wt b^{-1}\varphi_{2,s}\wt b$.
	Again using \cref{l:completeness-of-G}, this defines in the limit $s\to \infty$ the desired element $b\in G(R)$. This finishes the proof in the case that $K$ contains $\Q_p^\cyc$.
	
	\medskip

	In the general case, we can argue in exactly the same way, but instead use the diagram
		\[ \begin{tikzcd}[column sep = 0.3cm]
		{ H^1_{\cts}(\Lambda,{\mg}_0^+(R^\cyc_\infty)/p^t)} \arrow[r] & { H^1_{\cts}(\Lambda,G/G_{s+t}(R^\cyc_\infty))} \arrow[r] & {H^1_{\cts}(\Lambda,G/G_{s}(R^\cyc_\infty))} \\
		{ H^1_{\cts}(\Lambda,{\mg}_0^+(R^\cyc)/p^t)} \arrow[r]\arrow[u] & 	{ H^1_{\cts}(\Lambda,G(R^\cyc)/G_{s+t}(R^\cyc))} \arrow[r]\arrow[u,"h_{s+t}"'] & {H^1_{\cts}(\Lambda,G(R^\cyc)/G_{s}(R^\cyc))}.\arrow[u,"h_{s}"']
	\end{tikzcd}\]
The map on the the left, and the map on obstruction classes in $H^2$, have $p^\gamma$-torsion cokernel by \cref{l:non-cyclotomic-version-of-Sch13-4.5/5.5}. The same inductive lifting procedure as in the cyclotomic case then shows that we can lift $\rho$ to an element $\varphi\in H^1_{\cts}(\Lambda,G(R^\cyc))$. 
In fact, we can arrange for $\varphi$ to be in the image of the map $\sec$ of  \cref{l:section-for-Lambda}:
 As in the first part of the proof, we find a lift $\wt{\varphi}$ of $\varphi_s$ to the middle of the bottom row. In an additional step, we now apply $\sec\circ \res$ to this:
 Since in the induction start $\varphi_s$ is trivial, and $\sec$ and $\res$ are functorial, this shows inductively that $\sec \circ \res(\wt{\varphi})$ is also a lift of  $\varphi_s$.
Second, on $H^1_{\cts}(\Lambda,{\mg}_0^+(R^\cyc_\infty)/p^t)$, the map $\res$ is an isomorphism up to the $p^\gamma$-torsion kernel, which is a direct factor. We can therefore reduce modulo $p^\gamma$ once more to arrange that the difference $\delta'$ from the last proof is already in the image of $\sec$. This way we ensure inductively that $\varphi_{s+\gamma}$ is in the image of $\sec$.
\end{proof}
We can now prove the next instance of \cref{t:main-thm-for-both-O-and-O^+}, the case that $G$ has good reduction:

\begin{Proposition}\label{t:Psi-G-isomorphism-good-reduction}
	Let $K$ be a perfectoid field over $\Q_p$.
	Let $G$ be a rigid group over $K$ of good reduction. Then for any smoothoid space $X$ over $K$, the map $\Psi_G$ is surjective. If moreover $\Q_p^\cyc\subseteq K$, then $\Psi_G$ is an isomorphism.
\end{Proposition}
\begin{proof}
	We may work locally in $X_{\et}$ and assume that there is a toric chart $f:X\to \mathbb T^d\times Y$, as well as $\zeta_p\in K$. We then have the commutative diagram \cref{p:master-diagram-Wf-vs-Psi}. 
	
	To see that $\Psi_G$ is surjective, let $x\in H^1_v(X,G)$. Let $k\in \N$ and consider the subgroup $G_k\subseteq G$ from  \cref{p:inductive-lifting-ses-for-G}. By \cref{p:reduction-of-structure-group}, every $v$-$G$-bundle on $X$ admits a reduction of structure group to $G_k$ on some \'etale cover $X'\to X$. Going up the toric tower of $X'$, we can for $k\gg 0$ arrange by \cref{l:Sch13-4.5/5.5}.2 that $k>c$ where $c$ is the constant of \cref{p:part-2-of-Faltings-Lemma} for $X'$. Now $x$ is in the image of $H^1_v(X',G_c)\to H^1_v(X',G)$.
	By \cref{p:part-2-of-Faltings-Lemma}.1, this implies that $x\in \im(W_f)$. It follows from \cref{p:master-diagram-Wf-vs-Psi} that $x\in \im(\Psi_G)$.
	
	To see that $\Psi_G$ is injective, let $\theta_1,\theta_2\in (\mg\otimes \wtOm)^{\wedge =0}(X)$ be such that $\Psi_G(\theta_1)=\Psi_G(\theta_2)$. By \cref{p:master-diagram-Wf-vs-Psi}, we can pass to an \'etale cover $X'\to X$ where we can lift these to homomorphisms $\wt \theta_1, \wt \theta_2\in \mathcal Z^{\ab}(\mathfrak u)$. Then by commutativity, we have 
	$W_f(\exp(\wt \theta_1))=W_f(\exp(\wt \theta_2))$,
	which by  \cref{p:part-2-of-Faltings-Lemma}.2 implies that $\exp(\wt \theta_1)$ and $\exp(\wt \theta_2)$ are conjugated via $G(X')$. By \cref{l:exp-log-commutativity}.3, it follows that already $\wt \theta_1$, $\wt \theta_2$ are conjugated via $G(X')$, thus so are $\theta_1,\theta_2$.
\end{proof}

\begin{Corollary}
	Let $K$ be any perfectoid field over $\Q_p$.
	Let $G$ be any rigid group over $K$.  Then $\Psi_G$ is surjective on any smoothoid space over $K$.
\end{Corollary}
\begin{proof}
	By \cref{c:open-subgroup-of-good-reduction}, there is an open subgroup $G^\circ\subseteq G$ of good reduction. The diagram
	\[\begin{tikzcd}
		\Higgs_{G^\circ} \arrow[d]\arrow[r,"\Psi_{G^\circ}"] & R^1\nu_{\ast}G^\circ \arrow[d] \\
		\Higgs_{G} \arrow[r,"\Psi_G"] & R^1\nu_{\ast}G
	\end{tikzcd}\]
	commutes by functoriality of $\Psi$. By \cref{p:reduction-of-structure-group}, the right vertical map  is surjective. The top morphism is surjective by \cref{t:Psi-G-isomorphism-good-reduction}. Thus the bottom map is surjective.
\end{proof}

\subsection{Injectivity of $\Psi$}
Throughout this subsection, we assume that $\Q_p^\cyc\subseteq K$ and retain the notation of
\cref{s:abelian-G-bundles}, i.e.\ $X=\Spa(R,R^+)$ is a toric smoothoid space with $\Delta$-torsor $X_\infty= \Spa(R_\infty,R_\infty^+)\to X$. In order to prove that $\Psi_G$ is injective, we use the map
$W_f\colon \Hom(\Delta,U(R))\to H^1_{\cts}(X,G(R_\infty))$
from \cref{s:abelian-G-bundles}. Unravelling \cref{d:non-ab-1-cocycles}, we see that we have the following:
\begin{Lemma}\label{l:eq-cond-Wf-injective}
	For any two $\rho_1,\rho_2\colon \Delta\to U(R)$ in $\mathcal Z^{\ab}(G)(X)$, the following are equivalent:
	\begin{enumerate}
		\item $W_f(\rho_1)=W_f(\rho_2)$
		\item There is $A\in G(R_\infty)$ such that 
		\begin{equation}\label{eq:inj-of-W_f}
			\rho_1(\gamma)=A^{-1}\cdot \rho_2(\gamma)\cdot \gamma^{\ast}A \quad \text{for all $\gamma\in\Delta$.}
		\end{equation}
	\end{enumerate}
\end{Lemma}

	Equivalently, \cref{eq:inj-of-W_f} expresses that there is an isomorphism between the $G$-torsors on $X_v$ associated to $\rho_1$ and $\rho_2$. As we will explain in detail in the next section, the following Proposition  says that this can be translated into an isomorphism of $G$-Higgs bundles:

\begin{Proposition}\label{l:fully-faithful-case-of-good-reduction}
	Assume that there is $A\in G(R_\infty)$ such that \cref{eq:inj-of-W_f} holds.
	Then $A\in G(R)$.
\end{Proposition}

For $G=\GL_n$, this is the analogue of \cite[Lemma~1.(ii)]{Faltings_SimpsonI} in our setting. However, our argument is different to Faltings', which uses the embedding of $\GL_n$ into the ambient ring $M_n$, for which there is no analogue for general rigid groups $G$.
\begin{proof}
	The proof consists of two steps: A ``decompletion'' showing that $A\in G(X_n)$ for some subcover $X_\infty\to X_n\to X$, and then the descent further to $G(X)$. We begin with the latter:

\begin{Claim}\label{cl:from-decompletion-to-descent}
	Assume that \cref{eq:inj-of-W_f} holds with $A\in G(R_n)$ for some $n\in \N$. Then $A\in G(R)$.
\end{Claim}
\begin{proof}
	$A$ is fixed by $\Delta_n:=p^n\Delta$, hence $\rho_1(\gamma^{p^n})=A^{-1}\cdot \rho_2(\gamma^{p^n})\cdot A$ for all $\gamma\in \Delta$.
Let $x:=\rho_1(\gamma)$ and $y:=\rho_2(\gamma)$. Consider the subgroups $V:=U(R)\cap AU(R) A^{-1}\subseteq U(R)$ and $\mathfrak v:=\mathfrak u(R)\cap  \ad(A^{-1})(\mathfrak u(R))\subseteq \mathfrak u(R)$, then the above equation shows that $y^{p^n}\in V$. Let moreover $V':=U(R)\cap A^{-1}U(R) A\subseteq U(R)$ and $\mathfrak v':=\mathfrak u(R)\cap  \ad(A)(\mathfrak u(R))\subseteq \mathfrak u(R)$, then
by functoriality of the logarithm, we have a commutative diagram of bijections
\[ \begin{tikzcd}[column sep = 1.5cm]
	V \arrow[d, "\log"'] \arrow[r, "g\mapsto A^{-1}gA"] & V' \arrow[d, "\log"] \\
	\mathfrak v \arrow[r, "\ad(A)"] & \mathfrak v'.
\end{tikzcd}\]
It follows that
$ p^n\log(x)=\log(x^{p^n})=\log(A^{-1}y^{p^n}A)=\ad(A)(\log(y^{p^n}))=p^n\ad(A)\log(y)$ in $\mathfrak v'\subseteq \mg(R)$, 
which implies $\log(x)=\ad(A)(\log(y))$ in the $K$-vector space $\mathfrak g(R)$. It follows from this that we still have $\log x\in \mathfrak v'$. We can therefore reverse the calculation using $\exp$, 
\[ x=\exp(\ad(A)(\log(y)))=A^{-1}\cdot y\cdot A.\]
Since $x=A^{-1}y\gamma^{\ast}A$ by assumption, this shows that $\gamma^{\ast}A=A$. It follows that $A\in G(R)$.
\end{proof}

\begin{Claim}\label{c:fully-faithful-for-G-of-good-reduction}
	If $G$ has good reduction, there is $c\geq 0$ depending on $X$ such that for any $\rho_{1},\rho_2:\Delta\to G_c(R)$ in $\mathcal Z^{\ab}(G_c)(X)$ and $A\in G(R_\infty)$ such that \cref{eq:inj-of-W_f} holds, 
	we have $A\in G(R)$.
\end{Claim}
\begin{proof}
	By \cref{p:part-2-of-Faltings-Lemma}, we already know that for $c\geq 0$, there is $B\in G(R)$ such that $\rho_1(\gamma)=B^{-1}\rho_2B$. Replacing $A$ by $B^{-1}A$. We may therefore assume that $\rho_1=\rho_2=:\rho$.
	
	Consider $G(R_\infty)$ endowed with the action of $\Delta$ twisted by $\rho$ in the sense of \cref{d:twisted-action}, i.e.\ $\gamma\ast_{\rho} x=\rho(\gamma)\gamma^{\ast}x\rho(\gamma)^{-1}$, we write this $\Delta$-module as ${}_\rho G(R_\infty)$. Then \cref{eq:inj-of-W_f} can be expressed as
	$A\in H^0(\Delta,{}_\rho G(R_\infty))$,
	and we wish to see that this is already in the image of
	\[ H^0(\Delta,{}_\rho G(R))\to H^0(\Delta,{}_\rho G(R_\infty)).\]
	This can be seen by the same argument as in \cref{p:part-2-of-Faltings-Lemma} but in cohomological degree $0$: 
	
	Let $\gamma$ be as in \cref{l:Sch13-4.5/5.5}, then $c\geq t:= 3\gamma$.
	We first observe that in $G/G_c=\overline{G}_c$, where $\rho_1$ becomes trivial, the image $\overline{A}_c$ of $A$ satisfies $\gamma^{\ast}\overline{A}_c=\overline{A}_c$ and is thus contained in $\overline{G}_c(R_\infty)^\Delta=\overline{G}_c(R)$. By \cite[Proposition 4.16]{heuer-G-torsors-perfectoid-spaces}, we may find an \'etale cover $X'\to X$ on which $\overline{A}_c$ lifts to $G(X')/G_c(X')$. Moving up the toric tower of $X'$, we may by \cref{l:Sch13-4.5/5.5}.2 assume after replacing $\Delta$ by an open subgroup $\Delta_n$ that $c$ also works for $X'$. Replacing $X$ by $X'$, this first shows that $A\in G(X'_n)$, and then by \cref{cl:from-decompletion-to-descent} that $A\in G(X')$ is already $\Delta$-invariant, thus $A\in G(X)$.  We may therefore assume that $\overline{A}_c$ lifts to $G(X)/G_c(X)$.
	
	\medskip
	
	We now show inductively that for any $s\geq c$, the image $\overline{A}_s$ of $A$ in $\overline{G}_s$ is $\Delta$-invariant and lifts to $G(R)/G_s(R)$:
	If we know this for $s$, then by \cref{p:inductive-lifting-ses-for-G} we have exact sequences
	\[ \begin{tikzcd}[column sep = 0.25cm]
		0\arrow[r] &{ ({}_\rho {\mg}_0^+(R_\infty)/p^t)}^{\Delta} \arrow[r] & { ({}_\rho G/G_{s+t}(R_\infty))^\Delta} \arrow[r] & {({}_\rho G/G_{s}(R_\infty))^\Delta} \arrow[r] &	{ H^1_{\cts}(\Delta,{}_\rho {\mg}_0^+(R_\infty)/p^t)}\\
		0 \arrow[r] &{ ({}_\rho {\mg}_0^+(R)/p^t)})^\Delta \arrow[r]\arrow[u,"h_{t}"']& 	{ ({}_\rho G(R)/G_{s+t}(R))^{\Delta}} \arrow[r]\arrow[u,"h_{s+t}"'] & {({}_\rho G(R)/G_{s}(R))^{\Delta}}\arrow[r] \arrow[u,"h_{s}"']&	{ H^1_{\cts}(\Delta,{}_\rho {\mg}_0^+(R)/p^t)}\arrow[u,"h_{t}"'].
	\end{tikzcd}\]
	As the first and last vertical maps are isomorphisms up to $p^\gamma$-torsion by \cref{l:Sch13-4.5/5.5}, the same chase as in \cref{p:part-2-of-Faltings-Lemma} shows that there is a preimage of $\overline{A}_{s+t-2\gamma}$ under $h_{s+t-2\gamma}$ that lifts to $G(R)$.
	This shows that $A\in \varprojlim_{s\in \N} G(X)/G_s(X)=G(X)$.
\end{proof}
We now return to the case of general $G$. To prove \cref{l:fully-faithful-case-of-good-reduction}, it suffices by \cref{cl:from-decompletion-to-descent} to prove that $A$ is fixed by an open subgroup of $\Delta$. For the proof, we may replace  $X$ by an \'etale cover. Keeping this in mind, consider the image of $A$ in $G/U(R_\infty)$. 
	By \cite[Proposition 4.1]{heuer-G-torsors-perfectoid-spaces}, there is $n\in \N$ such that this comes from $G/U(R_n)$. After replacing $X$ by an \'etale cover, we can assume that this lifts to an element $A_n\in G(R_n)$. Then $A^\circ:=A\cdot A_n^{-1}\in U(R_\infty)$. Consider $\rho_3(\gamma):=A_n\rho_2(\gamma)A_n^{-1}$ defined on the open subgroup $\Delta_n\subseteq \Delta$ that fixes $R_n\subseteq R_\infty$. On the level of Lie algebras we see that we can find a small rigid open subgroup $H\subseteq U$ such that $A_nH(R_\infty) A_n^{-1}\subseteq U(R_\infty)$. After  increasing $n$ to replace $\Delta$ by an  open subgroup such that $\rho_2$ has image in $H$, we can therefore arrange for $\rho_3$ to also be a continuous homomorphism of the form $\rho_3:\Delta\to U(R_n)$. Then for any $\gamma\in \Delta_n$,
	\[\rho_1(\gamma)=A^{-1}\cdot \rho_2(\gamma)\cdot \gamma^{\ast}A=A^{-1}\cdot A_n^{-1}\rho_3(\gamma) A_n\cdot \gamma^{\ast}A=A^{\circ -1}\rho_3(\gamma)\gamma^{\ast}A^{\circ}.\]
	We now apply \cref{c:fully-faithful-for-G-of-good-reduction} to $\rho_1$ and $\rho_3$ to deduce that $A^\circ\in U(R_n)$, hence $A\in G(R_n)$.
\end{proof}

\begin{Proposition}\label{p:Psi-G-injective}
	Let $K|\Q_p^\cyc$ be a perfectoid field. 
	Let $X$ be any smoothoid over $K$ and let $G$ be any rigid group over $K$. Then
	$\Psi_G$ is injective.
\end{Proposition}
\begin{proof}
	We may assume that $X$ is toric and choose a toric chart. We consider diagram \cref{eq:master-diagram-proof-main-thm}:
	
	Let $\theta_1,\theta_2\in (\mg\otimes \wtOm)^{\wedge =0}(X)$ be such that $\Psi_G(\theta_1)=\Psi_G(\theta_2)$. As in the last part of the proof of  \cref{t:Psi-G-isomorphism-good-reduction}, we can pass to an \'etale cover $X'\to X$ where we can lift these to homomorphisms $\wt \theta_1, \wt \theta_2$ in the top left, so that $W_f(\exp(\wt \theta_1))=W_f(\exp(\wt \theta_2))$
	inside $H^1_v(X',G)$.
	By \cref{l:eq-cond-Wf-injective}, there is $A\in G(R_\infty)$ such that 
	$\exp(\wt \theta_1(\gamma))=A^{-1}\exp(\wt \theta_2(\gamma))\gamma^{\ast}A$. By \cref{l:fully-faithful-case-of-good-reduction}, we have $A\in G(X')$, so $\exp(\wt \theta_1)$ and $\exp(\wt \theta_2)$ are conjugated via $A$. By \cref{l:exp-log-commutativity}.3 it follows that already $\wt \theta_1$, $\wt \theta_2$ are conjugated via $G(X')$, thus so are $\theta_1,\theta_2$.
\end{proof}
This finishes the proof the $\Psi_G$ is an isomorphism if $K$ contains $\Q_p^\cyc$. 
One could prove the case of non-cyclotomic perfectoid $K$ directly with some more work, but what we have shown so far is already enough to deduce this case by descent in the next section.

\section{The local $p$-adic Simpson correspondence for $G$}
The technical work of the previous section also proves a new instance of Faltings' ``local $p$-adic Simpson correspondence'': Via the identification of ``generalised representations'' with $v$-vector bundles (\cite[Proposition~2.6]{heuer-G-torsors-perfectoid-spaces}), we can now generalise this from $\GL_n$ to rigid groups $G$ and from rigid spaces to perfectoid families of such.
For the formulation, we need a notion of ``smallness'' on either side of the correspondence. In the setting of rigid groups, this is not always intrinsic to $G$. Rather, it depends on an integral structure on $G$: The choice of a rigid open subgroup $G^+\subseteq G$ of good reduction, which always exists by \cref{c:open-subgroup-of-good-reduction}. 
\begin{Example}
		For $G=\GL_n$, the canonical choice $G^+=\GL_n(\O^+)$ recovers the classical setting of the local Simpson correspondence. More generally, there is a canonical choice for $G^+$ when $G$ is the analytification of an algebraic group over $K$ that extends to a group scheme $\mathcal G$ over $\O_K$: Namely, we can then take $G^+:=\mathcal G(\O^+)$, which is represented by the generic fibre of the $p$-adic completion of $\mathcal G$. This works for any split reductive groups.
\end{Example} 
Recall from \cref{p:inductive-lifting-ses-for-G} that given $G^+$, we get a canonical system of open subgroups $G^+_k$ for $k\geq 0$, as well as the integral subgroups $\mg^+_k\subseteq \mg$ of the Lie algebra. By \cref{p:exp-on-Lie-alg}, there is $\alpha>0$ such that for any $k>\alpha$, we have an exponential map
$\exp:\mg^+_k\isomarrow G^+_k$.

\begin{Definition}\label{d:small}
Let $X$ be a toric smoothoid space  and fix a toric chart $f:X\to \mathbb T^d\times Y$ (\cref{d:toric-chart}). Let $G$ be a rigid group and fix an open subgroup $G^+\subseteq G$ of good reduction. 
\begin{enumerate}
\item 
A $v$-$G$-bundle $V$ on $X_v$ is called \textbf{small} (with respect to $G^+$) if it admits a reduction of structure group to $G^+_c$ for some $c$ such that \cref{p:part-2-of-Faltings-Lemma} holds.

\item A $G$-Higgs bundle $(E,\theta)$ on $X$ is called \textbf{small}  (with respect to $G^+$) if for some $c>\alpha$ such that \cref{p:part-2-of-Faltings-Lemma} holds, there is a $G^+_c$-torsor $E^+$ with an isomorphism $E=G\times^{G^+_c}E^+$ with respect to which $\theta$ is a section of the integral $\O^+$-submodule
\[\ad(E)^+\otimes_{\O^+} \wtOm^+\subseteq \ad(E)\otimes_\O \wtOm,\]
where $\ad(E)^+:=\mg^+_c\times^{G_c^+}E^+\subseteq \ad(E)$ is the $\O^+$-sublattice induced by $E$, and $\wtOm^+\subseteq \wtOm$ is the locally free integral submodule induced by the chart $f$, see \cref{l:wtOm-integral-submodule-ind-by-f}.
\end{enumerate}
\end{Definition}
For $G=\GL_n$ and $G^+:=\GL_n(\O^+)$, we also define a $v$-vector bundle $V$ to be small if it is the change of fibre of a $v$-topological $1+p^c\m M_n(\O^+)$-bundle $V_c$ to $\G_a^d$, i.e.\ $V\cong \G_a^d\times^{G_c}V_c$.
By a small Higgs bundle, we then mean the change of fibre of a small Higgs bundle to $\G_a^d$.
\begin{Lemma}\label{l:etale-small-is-trivial}
	Any small $v$-$G$-bundle on $X_{v}$ that is \'etale-locally trivial is already globally trivial on $X$. In particular, the $G$-bundle underlying any small Higgs bundle is trivial.
\end{Lemma}
\begin{proof}
	By \cref{p:part-2-of-Faltings-Lemma}, any element $x$ in the image of $H^1_v(X,G_c^+)\to H^1_v(X,G)$ is in the image of $W_f$. On the other hand, if $x$ comes from $H^1_{\et}(X,G)$, it is sent to $0$ in $R^1\nu_{\ast}G(X)$, thus the image under $\mathcal W_f:\mathcal Z^{\ab}(G_c)(X)\to R^1\nu_{\ast}G(X)$ is trivial.  But the fibre over $0$ of this map  is trivial by \cref{p:master-diagram-Wf-vs-Psi} and \cref{c:Psi-G-trivial-kernel}. Hence $x=0$.
\end{proof}

\begin{Lemma}\label{l:v-bundle-locally-small}
	Any $v$-$G$-bundle $V$ on $X_v$ becomes small on an \'etale cover of $X$.
\end{Lemma}
\begin{proof}
	By \cref{p:reduction-of-structure-group}, we can after replacing $X$ by an \'etale cover find a reduction of structure group of $V$ to $G^+_k$ for some $k>\alpha$. This does not prove the lemma yet since passage to an \'etale cover might change the required bound $c$, which depends on $X$. However, we now can argue exactly as in the proof of \cite[Lemma~2.30]{heuer-G-torsors-perfectoid-spaces}: By \cref{l:small-bundles-on-perfectoid}, the $v$-$G$-bundle $V$ becomes trivial on $X_\infty\to X$. It therefore corresponds to a class $[\rho]$ in $H^1_{\cts}(\Delta,G^+_k(X_\infty))$. Going up the toric tower, which does not change $c$ by \cref{p:part-2-of-Faltings-Lemma}, we can replace $\Delta$ by an open subgroup that is sent into $G^+_c(X_\infty)$ by $\rho$. 
\end{proof}

\begin{Theorem}[Local $p$-adic Simpson correspondence for $G$-bundles in perfectoid families]\label{t:local-paCS}
Let $X$ be a toric smoothoid space over $K$ and fix a toric chart $f:X\to \mathbb T^d\times Y$ (see \cref{d:toric-chart}). Let $G$ be a rigid group over $K$. Fix an open subgroup $G^+\subseteq G$ of good reduction. This induces a notion of smallness of $G$-bundles (see \cref{d:small}).
\begin{enumerate}
\item There is an equivalence of categories
\[ LS_{f}:\{\text{small $G$-bundles on $X_v$}\}\isomarrow \{\text{small $G$-Higgs bundles on $X_{\et}$}\} \]
that is natural in $G$, but not in general independent of the chart $f$.
\item In the case of $G=\GL_n$, this extends to an exact equivalence of categories
\[ \{\text{small vector bundles on $X_v$}\}\isomarrow \{\text{small Higgs bundles on $X_{\et}$}\}.\]
\item Let $g:X'\to X$ be any morphism of smoothoids with toric charts $f$ of $X$ and $f'$ of $X'$. Then any morphism of toric charts $\widetilde{g}:f'\to f$ over $g$ induces a natural equivalence 
\[
\begin{tikzcd}
	\{\begin{array}{@{}c@{}l}\text{small $G$-bundles on $X'_v$} \end{array}\} \arrow[r, "LS_{f'}"]  & \{\begin{array}{@{}c@{}l}\text{small $G$-Higgs bundles on $X'_\et$}\end{array}\}       \\
	\{\begin{array}{@{}c@{}l}\text{small $G$-bundles on $X_v$} \end{array}\} \arrow[u,"g^{\ast}",dashed] \arrow[r, "LS_{f}"] & 	\{\begin{array}{@{}c@{}l}\text{small $G$-Higgs bundles on $X_\et$} \end{array}\} \arrow[u,"g^\ast",dashed]\arrow[lu,Rightarrow,"\wt g"',shorten >=7ex,shorten <=7.0ex]
\end{tikzcd}
\]
where the vertical functors are defined only on the respective subcategories of small objects with respect to $X$ for which the pullback along $g$ is small with respect to $X'$.

\item Let $V$ be a small $v$-vector bundle on $X_v$ and let $LS_f(V)=(E,\theta)$. Then there is a natural isomorphism 
$R\Gamma_v(X,V)=R\Gamma_{\Higgs}(X,(E,\theta))$
in $\mathcal D(\O(X))$. In particular, for $\nu:X_v\to X_{\et}$ we get a natural isomorphism  $R\nu_{\ast}V=R\Gamma_{\Higgs}(E,\theta)$ in $\mathcal D(X_{\et})$.
\end{enumerate}
\end{Theorem}
\begin{Remark}\label{r:explaining-localpaCS}
\begin{enumerate}
\item For functoriality in $f$ in part 3, we use the notion of morphisms of toric charts from
\cref{d:toric-chart}.2. We caution that for a given morphism $X'\to X$, it is not always possible to find charts $f$ and $f'$ that admit a  morphism of charts between them.
\item Our notion of ``smallness'' is more restrictive than that in other setups like \cite{Faltings_SimpsonI} \cite{AGT-p-adic-Simpson} \cite{Tsuji-localSimpson} \cite{Wang-Simpson}. This seems necessary to treat general rigid groups $G$.
We discuss the precise comparison to these other works in \cref{s:comparison-to-AGT}.
\item The last sentence in (3) is due to the fact that in general, $g^\ast$ may not preserve smallness, as the implicit constant $c$ in \Cref{p:part-2-of-Faltings-Lemma} depends on $X$. However, in many concrete situations of interest, $g^\ast$ does preserve smallness, for example when $X=Y$ and $g$ describes some Galois action.
\item For $G=\GL_n$, apart from exactness, the difference between \Cref{t:local-paCS}.(1) and (2) is that in (2), we allow morphisms that are not necessarily isomorphisms, see \cref{r:GLn-torsors-vs-vector-bundles}.
\item Even if $G$ is reductive and $X$ is a rigid space, we do not see how this could simply be deduced from $G=\GL_n$ by the Tannakian formalism, due to the smallness conditions.
\item The cohomological comparison in \Cref{t:local-paCS}.(4) is the generalisation to smoothoids of results that are known for smooth rigid spaces at least in the arithmetic setup: For Faltings local correspondence, this is due to Faltings \cite[p852]{Faltings_SimpsonI}, see also \cite[IV]{AGT-p-adic-Simpson}. It is also closely related to \cite[Theorem 2.1.(v)]{LiuZhu_RiemannHilbert} and \cite[Theorem 3.13]{MinWang22} which give analogous comparison results in the analytic setting over discretely valued fields. Our proof of (4) follows a similar strategy. Apart from the different base field, our  result can be interpreted as giving a relative version for families of rigid spaces. 
\end{enumerate}
\end{Remark}

\subsection{Proof of the local correspondence for general $G$}\label{s:construction-of-local-correspondence}
\begin{proof}[Proof of \cref{t:local-paCS}]
We first assume that $\Q_p^\cyc\subseteq K$, then the toric cover $X_\infty\to X$ associated to $f$ in \cref{s:comp-toric-charts} is an affinoid perfectoid $\Delta$-torsor.
In this setting, our technical preparations so far allow us to essentially follow Faltings' construction \cite[Theorem 3]{Faltings_SimpsonI}:

We will define a functor $LS^{-1}_f$ from right to left, and show that this is fully faithful and essentially surjective.
Let $(E,\theta)$ be a small $G$-Higgs bundle and let $E^+\subseteq E$ be a reduction of structure group to $G_c^+$ with respect to which $\theta$ has coefficients in $\ad(E)^+$. As explained in the beginning of \cref{s:abelian-v-G-bundles}, $\theta$ can be written via Lemma~\ref{l:HT-isom-on-small-dmd} as a continuous homomorphism
$\rho:\Delta\to \ad(E)^+(X)$ with commutative image. By \cref{l:etale-small-is-trivial} and \cref{l:exp-log-commutativity}, we can thus apply $\exp$ to turn this into a continuous $1$-cocycle
$\exp(\rho):\Delta\to \underline{\Aut}_{G_c^+}(E^+)(X)$.
This defines via Cartan--Leray a small $v$-$G$-bundle $V_{\rho}$ on $X_v$ that sends $W\in X_v$ to
\[ V_{\rho}(W):=\{s\in E(X_\infty\times_XW)|\gamma\cdot  s=\exp(-\rho(\gamma))s \text{ for all }\gamma\in \Delta\}.\]
(The sign in front of $\rho(\gamma)$ is usually required due to our conventions on cocycles \cref{d:non-ab-1-cocycles}, to make the cocycle condition translate into $\gamma_1\gamma_2\cdot x=\gamma_1^\ast\exp(-\rho(\gamma_2))\exp(-\rho(\gamma_1))x$. Here it is of minor importance as $\rho$ has commutative image). We now set $LS^{-1}(E,\theta):=V_{\rho}$.

We note that $E$ is a trivial $G$-bundle on $X$ by \cref{l:etale-small-is-trivial}, and $V_{\rho}$ is given by regarding $\exp(\rho)$ as a descent datum for $E$ along $X_\infty\to X$. In particular, $V_{\rho}$ becomes trivial on $X_\infty$.

We claim that this construction is functorial: A morphism of $G$-Higgs bundles $(E_1,\theta_1)\to (E_2,\theta_2)$ is the same as a morphism of $G$-bundles $\varphi:E_1\to E_2$ such that for the homomorphisms $\rho_1:\Delta\to \ad(E_{1})^+(X)$ and $\rho_2:\Delta\to \ad(E_{2})^+(X)$ associated to $\theta_1,\theta_2$ we have 
\begin{equation}\label{eq:local-paCS-ff-1}
\varphi^{-1}\rho_2(\gamma)\varphi=\rho_1(\gamma)\quad \forall \gamma\in\Delta.
\end{equation}
By \cref{l:exp-log-commutativity}.3, it follows that inside $E_1(X)$, we have
\begin{equation}\label{eq:local-paCS-ff-2} \varphi^{-1}\exp(-\rho_2(\gamma))\varphi=\exp(-\rho_1(\gamma))\quad \forall \gamma\in\Delta,
\end{equation}
which implies that for any $s\in E(X_\infty\times_XW)$, we have 
$\gamma\cdot \varphi(s)=\exp(-\rho_2(\gamma))\varphi(s)$ for all $\gamma\in\Delta$.
This shows that the natural homomorphism
$\varphi:E_1(X_\infty\times_XW)\to E_2(X_\infty\times_XW)$
restricts to the desired map $V_{\rho_1}(W)\to  V_{\rho_2}(W)$.

	\medskip

\textbf{$LS^{-1}$ is essentially surjective:} Let $V$ be a small $v$-$G$-bundle on $X$. Then by \cref{p:part-2-of-Faltings-Lemma}.1, we can find $\rho:\Delta\to G_c^+(X)$ such that the class of $V$ in $H^1_v(X,G)$ is equal to $W_f(\rho)$. Then the $G$-Higgs bundle $(G,\theta)$ with $\theta:=\log(\rho)$ is such that $LS^{-1}(G,\theta)\cong V$.

	\medskip

\textbf{$LS^{-1}$ is fully faithful:} Let $(E_1,\theta_1)$ and $(E_2,\theta_2)$ be two small $G$-Higgs bundles on $X$. By \cref{l:etale-small-is-trivial}, we can find isomorphisms $E_1=G$ and $E_2=G$.
Let $\phi:V_{\rho_1}\to V_{\rho_2}$ be any morphism between the associated $v$-topological $G$-torsors. Since $V_{\rho_1}$ and $V_{\rho_2}$ become trivial on $X_{\infty}$ by construction, 
this is by descent the same as a $\Delta$-invariant $G$-linear homomorphism
$\phi:V_{\rho_1}(X_\infty)\to V_{\rho_2}(X_\infty)$.
Choosing generators $x$ and $y$ on either side, this is a $\Delta$-equivariant homomorphism
$\phi:G(X_\infty)x\to G(X_\infty)y$.
Let $A\cdot y$ be the image of $1\cdot x\in G(X_\infty)x$, then the $\Delta$-equivariance means that for all $\gamma\in \Delta$, we have:
\begin{alignat*}{5}
	\gamma\cdot \phi(x)&=&&\phi(\gamma\cdot x)	\quad &\Rightarrow\quad&& \gamma^\ast A\exp(-\rho_2(\gamma))\cdot y &=&&\exp(-\rho_1(\gamma))A\cdot y\\
	&&&&\Rightarrow\quad&& A\exp(\rho_2(\gamma))\gamma^\ast A^{-1} &=&&\exp(\rho_1(\gamma))A
\end{alignat*}
By \cref{l:fully-faithful-case-of-good-reduction}, this implies $A\in G(X)$. 
We thus have \cref{eq:local-paCS-ff-2} with $\varphi=A$, which again by \cref{l:exp-log-commutativity}.3 is equivalent to \cref{eq:local-paCS-ff-1}, expressing that $A$ defines a uniquely determined morphism of Higgs bundles $A:(E_1,\theta_1)\to (E_2,\theta_2)$. Hence $LS^{-1}$ is an equivalence.

	\medskip

\textbf{Naturality in $f$ and $G$:}  Naturality in $G$ is clear.
Functoriality in $f$ can be seen exactly as in \cref{l:functoriality-of-Psi}: Any morphism $\wt g$ of toric charts between $f'$ and $f$
as in  \cref{d:toric-chart} induces a morphism $X'_\infty \to X_\infty$ which is equivariant with respect to the induced map $\Delta'\to \Delta$ between the Galois groups. The pullback of $G$-Higgs bundles is then given by sending $\rho$ to
$\rho':\Delta'\to \Delta\xrightarrow{\rho} \ad(E^+)(X)\to \ad(E^+)(X')$
and $\exp$ of this agrees with the cocycle
\[ \exp(\rho'):\Delta'\to \Delta\xrightarrow{\exp(\rho)} \underline{\Aut}_{G_c^+}(E^+)(X)\to \underline{\Aut}_{G_c^+}(E^+)(X').\]
Thus the pullback of $V_{\rho}$ along $X'\to X$ agrees with $V_{\rho'}$ defined with respect to $f'$.

Part 2 can be seen exactly like part 1: We just need to replace \cref{eq:local-paCS-ff-1} by the equation
$\rho_2(\gamma)\varphi=\varphi\rho_1(\gamma)$ for all $ \gamma\in\Delta$
which is equivalent to
$\exp(\rho_2(\gamma))\varphi=\varphi\exp(\rho_1(\gamma))$, this time
by \cite[Lemma~3.11]{heuer-G-torsors-perfectoid-spaces}.
Finally, exactness in part 2 follows from the case of parabolic subgroups $G\subseteq \GL_n$, and functoriality of the local correspondence in $G$.

This finishes the proof of \cref{t:local-paCS}.1 and 2.

\medskip

We end this subsection with two remarks on globalisation: 
\begin{Remark}One key  difference between \Cref{t:main-thm-for-both-O-and-O^+} and \Cref{t:local-paCS} is that the former works for any smooth rigid space $X$, while the latter depends on a toric chart $f:X\to \mathbb T^d$. For general smooth $X$, one might try to globalise the construction by choosing a cover by affinoids that admit a toric chart. However, \Cref{t:local-paCS}.3 says that it is only possible to compare the equivalences on overlaps in a canonical way after the choice of a morphism of toric charts. It is usually not possible to find these in a way that preserves gluing data.
	
	This should not be regarded as a flaw of the local correspondence, but rather as a meaningful conceptual barrier: In fact, the global correspondences of Faltings \cite{Faltings_SimpsonI}, Abbes-Gros and Tsuji \cite{AGT-p-adic-Simpson}, Wang \cite{Wang-Simpson} and others (see \Cref{s:comparison-to-AGT}) all rely on an additional choice, namely the choice of an $A_{\inf}/\xi^2$-lift of a semi-stable model of $X$. Second, they require a stronger smallness condition than the one that works for the local correspondence.
\end{Remark}
\begin{Remark}
	In this light, it is interesting to ask to what extent it is possible to generalise \Cref{t:local-paCS} to a more general class of perfectoid Galois covers, thus allowing better gluing properties in some special cases. We will therefore now sketch how one can use the more general preparations  from \Cref{s:Higgs-to-v} to generalise the local correspondence by axiomatising the class of Galois covers for which the proof still goes through. We will apply this in \cite{HMZ} to prove a small $p$-adic Simpson correspondence for $G$-torsors on abeloid varieties.
	
	 As in \S\ref{s:preps-and-choices}, let $X$ be an affinoid smoothoid adic space over a perfectoid field $K$ containing $\Q_p^\cyc$ and let
	\[f:\wt X=\Spa(R_\infty,R^+_\infty)\to X=\Spa(R,R^+)\]
	be a pro-\'etale affinoid perfectoid cover. We suppose that $f$ satisfies the following axioms:
	\begin{enumerate}[label=(\arabic*)]
		\item $f$ is a pro-finite-\'etale Galois cover, and the Galois group $\Delta$ is a finite free $\Z_p$-module.
		\item There is $\gamma>0$ such that for any $s\in \N$ and $i\in \{0,1,2\}$, the kernel and cokernel of 
		\[ H^i_{\cts}(\Delta,R^+/p^s)\to H^i_{\cts}(\Delta,R^+_\infty/p^s)\]
		are killed by  $p^{\gamma}$. Moreover, the same is true for
		$H^i_{\cts}(\Delta,R^+)\to H^i_{\cts}(\Delta,R^+_\infty)$.
		\end{enumerate}
	Note that (1) and (2) imply that $\Delta\cong \Z_p^d$ where $d$ is the smooth dimension of \Cref{d:differential-dimension}. Part (2) replaces \Cref{l:Sch13-4.5/5.5}.3. The case of $i=1$ implies that the Cartan--Leray map of $f$
	 \[ \Hom_\cts(\Delta,\O(R))\to H^1_{\cts}(\Delta,\O(R_\infty))\isomarrow H^1_v(X,\O)\xrightarrow{\HT}H^0(X,\wtOm^1)\]
	 is an isomorphism: Indeed, the second map is an isomorphism because $H^1_v(\wt X,\O)=0$ due to the assumption that $\wt X$ is affinoid perfectoid, and the map $\HT$ is an isomorphism because $X$ is affinoid.
	 We can thus deduce from (2) the analogue of \Cref{l:HT-isom-on-small-dmd} in this setting.
	 
	 Then the analogue of \Cref{p:part-2-of-Faltings-Lemma} holds for $f$, namely the exact same proof goes through with $c:=5\gamma$ by replacing \Cref{l:Sch13-4.5/5.5} with axiom (2). We can use this to define a notion of smallness as in \Cref{d:small}. However, we will have to slightly modify $c$ further:
	 
	 We also require an analogue of \Cref{l:fully-faithful-case-of-good-reduction} for $f$, but for this we need to be slightly more careful. The proof of this goes through verbatim except for the step in \Cref{c:fully-faithful-for-G-of-good-reduction} where \cref{l:Sch13-4.5/5.5}.2 is invoked, for which we do not have a direct analogue for $f$.
	  However, this step of the proof can still be generalised with some more work, by replacing the \'etale localisation argument with a lifting argument at the expense of increasing $c$: 
	  
	  In the notation of the proof, we need to show that $\overline{A}_c\in G/G_c(X)$ lifts to $G(X)$. The obstruction to lifting defines a class in $H^1_v(X,G_c)$. We claim that this class vanishes if we replace $U=G_c$ in \Cref{c:fully-faithful-for-G-of-good-reduction} by $U:=G_{2c}$. Indeed, if $\rho_1,\rho_2$ factor through $G_{2c}(R)$, then we already have $\overline{A}_{2c}\in G/G_{2c}(X)$ and the lifting obstruction thus lies in the image of a map
	 \[ h:H^1_v(X,G_{2c})\to H^1_v(X,G_c).\]
	 It thus suffices to see that $h=1$. But this follows from \Cref{l:etale-small-is-trivial}: Unravelling the definitions, we see that for the group $G':=G_c$, a $v$-$G'$-bundle on $X$ is small if it admits a reduction of structure group to $G'_c=G_{2c}$. Hence $h=1$, so $\overline{A}_{c}$ lifts uniquely to $G(X)/G_c(X)$ as desired. From here the proof of \Cref{l:fully-faithful-case-of-good-reduction} goes through without further changes.
	 
	 Adapting the notion of smallness accordingly, so that a $v$-$G$-bundle is small if it admits a reduction of structure group to $G_{2c}$, it follows that also \Cref{l:etale-small-is-trivial} holds in this setting.
	 
	 With these results at hand, the construction of the functor
	 \[ LS_{f}:\{\text{small $G$-bundles on $X_v$}\}\to \{\text{small $G$-Higgs bundles on $X_{\et}$}\} \]
	 now goes through verbatim: Indeed, the first part of the proof of \Cref{t:local-paCS} work in the same way using the analogue of \Cref{l:etale-small-is-trivial}. The proof of essential surjectivity works verbatim using the analogue of  \Cref{p:part-2-of-Faltings-Lemma}, and the proof of fully faithfulness works verbatim using the analogue of \Cref{l:fully-faithful-case-of-good-reduction}.
	 
	 All in all, this shows that we get a more general version of part 1, 2, 3 of \Cref{t:local-paCS} for not necessarily toric covers $\wt X\to X$ satisfying axioms (1) and (2), but at the expense of a more restrictive smallness assumption (defined in terms of $2c$ instead of $c$, where $c=5\gamma$). In some special cases where one has global perfectoid covers, this allows for better globalisation than one has for toric charts. 
	 That being said, the cohomological correspondence, \Cref{t:local-paCS}.4, will require more input specific to toric charts. This is our next goal.
\end{Remark}

\subsection{Cohomology of $v$-vector bundles}
Continuing the proof of \cref{t:local-paCS}, we now move on to part 4. 
For this the main computation is a more general version of part of  \cref{l:Sch13-4.5/5.5} (which was in turn based on  \cite[Lemma~5.5]{Scholze_p-adicHodgeForRigid}), describing the $v$-cohomology of small $v$-vector bundles, \cref{l:Sch13-4.5/5.5} being the case of the trivial bundle $\O$. As before, we write $X=\Spa(R,R^+)$ and $X_\infty=\Spa(R_\infty,R^+_\infty)$.

\begin{Lemma}\label{l:ScholzeLemma-5.5-small-v-bundles}
	Let $\rho:\Delta\to \GL_n(R^+)$ be a continuous representation such that $\rho\equiv 1\bmod p^\alpha$. Let either $M:=R^{+n}$ or $M:=R^{+n}/p^k$ for some $k\in \N$, endowed with $\Delta$-action via $\rho$. Endow $M\otimes_{R^+}R^+_\infty$ with the diagonal $\Delta$-action and  consider for any $m\geq 0$ the natural map
	\[ H^m_{\cts}(\Delta,M)\to H^m_{\cts}(\Delta,M\otimes_{R^+}R^+_\infty).\]
	Then we can find $c>0$ independent of $k$ such that the  kernel and cokernel of this map are annihilated by $p^{c}$. For $m=0$ and $M=R^{+n}$, the map is an isomorphism. 
\end{Lemma}
\begin{Remark}
For $m=0$, this give an alternative proof of \cref{l:fully-faithful-case-of-good-reduction} for $G=\GL_n$, by setting $M=\End(\rho_1,\rho_2)$. This is essentially Faltings' original proof in \cite[Lemma~1]{Faltings_SimpsonI}.
\end{Remark}
\begin{proof}
	We use \cref{l:Sch13-4.5/5.5} and its notation: As $R^+\hotimes_{A^+}A_{\infty}^+\to R^+_\infty$ is injective with $p^\beta$-torsion cokernel, it suffices to see the statement with $R^+\to R_\infty^+$ replaced by $A^+\to A_\infty^+$.
	For $m=0$ we then get upon inverting $p$ an isomorphism
	$M^\Delta\otimes_{R^+}R=(M\otimes_{R^+}R_\infty)^\Delta$,
	and using that $R^+=R^+_\infty\cap R$, we see that it is already an isomorphism before inverting $p$.
	
	By a limit argument, it thus suffices to prove the statement for $M/p^k$ for any $k\in \N$. Here we can argue as in \cite[p205-206]{Faltings_AlmostEtale}: Fix a compatible system of $p$-power roots of unity $\zeta$. Write $\gamma_1,\dots,\gamma_d$ for the induced generators of $\Delta=\Z_p^d$.
	We then have a decomposition
	\[M\otimes_{A^+}A_\infty^+/p^k=\bigoplus_{i=(i_1,\dots,i_d)\in [0,1)\cap \Z[\frac{1}{p}]}(M/p^k) \cdot T_1^{i_1}\cdots T_d^{i_d}  \]
	as a $\Delta$-module,
	where $(M/p^k) \cdot T_1^{i_1}\cdots T_d^{i_d}$ is the $R^+$-module $M/p^k$ with the action of $\rho$ twisted by the character  $\gamma_j\mapsto \zeta^{i_j}$. It thus suffices to prove that already
	$H^m_{\cts}(\Delta,(M/p^k) \cdot T_1^{i_1}\cdots T_d^{i_d})$
	is $p^{c}$-torsion unless $i=0$. 
	Via the inflation-restriction sequence and \cref{l:cts-grp-cohom-of-Zp}, we can reduce to $\Z_p$-cohomology.
	We are thus left to describe $R\Gamma_{\cts}(\Z_p,M/p^k)$ with $1\in \Z_p$ acting as multiplication by $\zeta^i\rho(\gamma_j)$ on $M/p^k$. By \cref{l:cts-grp-cohom-of-Zp}, this is computed by the complex
	\[ 0\to M/p^k\xrightarrow{\zeta^i\rho(\gamma_j)-1}  M/p^k\to 0.\]
	By assumption we have $\rho(\gamma_j)=1+p^\alpha A$ for some $A\in M_n(R^+)$, so $(\zeta^i\rho(\gamma_j)-1)/(\zeta^i-1)$ is a unit for $i\neq 0$. Hence the above map has the same kernel and cokernel as $\zeta^i-1$.
\end{proof}

\begin{Lemma}\label{l:Higgs-cohom-as-grp-cohom}
Let $\rho:\Delta\to \GL_n(R^+)$ be a continuous representation  such that $\rho\equiv 1\bmod p^\alpha$.
Let $M:=R^n$ endowed with $\Delta$-action via $\rho$. Then for $E:=M\otimes \O$ and $\theta:=\log(\rho)$ interpreted as a section of $M_n\otimes \wtOm$ as in \cref{s:construction-of-local-correspondence}, there is a natural isomorphism
\[ R\Gamma_{\cts}(\Delta,M)\isomarrow R\Gamma_{\mathrm{Higgs}}(X,(E,\theta)).\]
\end{Lemma}
\begin{proof}
	Let $\gamma_1,\dots,\gamma_d\in \Delta$ be topological generators and let $A_1,\dots, A_d$ be the images of $\gamma_1,\dots,\gamma_d$ under $\theta$. Via $\Hom(\Delta,\O(X))=\wtOm(X)$, the dual basis of $\gamma$ induces a basis $\delta_1,\dots,\delta_d$ of $\wtOm(X)$ and thus an isomorphism $\wtOm(X)=R^d$. The Higgs complex now evaluates on $X$ to
\[\mathcal C^\ast_{\mathrm{Higgs}}=\big [M\to M\otimes_RR^d\to  M\otimes_R\wedge^2R^d\to \dots \to M\otimes_R\wedge^dR^d\big]\]
where the $k$-th transition map sends $m \otimes e_{i_1}\wedge\dots\wedge e_{i_k}\mapsto\sum_j A_j(m)\otimes e_{i_1}\wedge\dots\wedge e_{i_k}\wedge e_j$

Comparing the complex $\mathcal C^\ast_{\mathrm{Higgs}}$ to $\mathcal C^\ast_{\mathrm{grp}}$ from \cref{l:cts-grp-cohom-of-Zp}, we see that both have the same terms, but different transition maps. We claim that there is an $R$-linear isomorphism
$u:\mathcal C^\ast_{\mathrm{grp}}\to \mathcal C^\ast_{\mathrm{Higgs}}$.
For this we first observe that inside $\End_R(M)$, when we identify $\gamma_i$ with its image under $\rho$, then
$A_i/(\gamma_i-1)=u_i:=1+(\gamma_i-1)\sum_{w=2}^\infty\frac{(-1)^{w+1}}{w}(\gamma_i-1)^{w-2}\in \End_R(M)$ is invertible as $\gamma_i-1$ is topologically nilpotent in $\End_R(M)$. Writing each term in $\mathcal C^{\ast}_{\mathrm{grp}}$ as
$M\otimes_R\wedge^kR^d=\bigoplus_{i_1<\dots<i_k}M \cdot e_{i_1}\wedge\dots \wedge e_{i_k}$,
we now define $u$ on each direct summand as
\[ u:M\cdot e_{i_1}\wedge\dots \wedge e_{i_k}\to M\cdot  e_{i_1}\wedge\dots \wedge e_{i_k}, \quad m\mapsto u_{i_1}\cdots u_{i_k} m.\]
This gives an isomorphism of complexes as the $\gamma_i$ and $A_i$ all commute with each other.
\end{proof}
In summary, we obtain for the $v$-vector bundle $V$  natural isomorphisms in $\mathcal D(R)$
\[ R\Gamma_v(X,V)=R\Gamma_{\cts}(\Delta,V(X_\infty))=R\Gamma_{\cts}(\Delta,M)=R\Gamma_{\Higgs}(X,(E,\theta))\]
where the first morphism comes from the Cartan--Leray sequence \cite[Proposition~2.8]{heuer-v_lb_rigid}.

This finishes the proof of \cref{t:local-paCS} in the case that $K$ contains $\Q_p^\cyc$.

\medskip

It remains to deduce the general case by Galois descent. For this we use the following:

\begin{Corollary}\label{c:Psi-G-trivial-kernel}
	For any perfectoid field $K$, any smoothoid $X$ over $K$ and any rigid group $G$ over $K$, the morphism
	$\Psi_G$
	has trivial fibre over the trivial element.
\end{Corollary}
\begin{proof}
	Let $X_{K'}\to X$ be the base-change to a completed algebraic closure of $K$. The diagram
	\[\begin{tikzcd}
		\Higgs_G(X_{K'}) \arrow[r,"\Psi_G"] & R^1\nu_{\ast}G(X_{K'}) \\
		\Higgs_G(X) \arrow[r,"\Psi_{G}"] \arrow[u] & R^1\nu_{\ast}G(X) \arrow[u]
	\end{tikzcd}\]
	commutes by \cref{p:Psi-summary}. The top morphism is injective by \cref{p:Psi-G-injective}. The left map has trivial kernel by \cref{l:Zariski-dense-restriction-on-Higgs}. Hence the bottom map has trivial kernel.
\end{proof}
\begin{Corollary}\label{c:descent-along-base-change}
	Let $X$ be a smoothoid space over $K$. Let $K'|K$ be any extension of perfectoid fields. Then we have $v$-descent of \'etale $G$-torsors along the base-change map $X_{K'}\to X$.
\end{Corollary}
\begin{proof}
	Any descent datum for an \'etale $G$-torsor along the morphism $X_{K'}\to X$ defines a $v$-$G$-bundle $V$ on $X$ that becomes trivial on $X_{K'}$. This is \'etale if its class in $R^1\nu_{\ast}G(X)$ vanishes. This follows by chasing the above diagram using that $\Psi_G$ is surjective by \cref{t:Psi-G-isomorphism-good-reduction}.
\end{proof}

Let $C$ be the completion of the cyclotomic extension of $K$ obtained by adjoining all $p$-power roots. By almost purity, this is again a perfectoid field. Let $G=\Gal(C|K)$ be the Galois group. Let $f_C:X_C\to \mathbb T^d_C\times Y$ be the base-change of $f$ from $K$ to $C$.

Let now $V$ be a $G$-torsor on $X_v$. Let $V_C$ be the pullback to $X_C$ and let $(E,\theta)$ be the Higgs bundle associated to $V_C$ via the local correspondence $LS$ with respect to $f_C$. Then by functoriality in the toric chart, the $v$-descent datum for $V_C$ along $X_C\to X$ induces a descent datum on $(E,\theta)$ for $X_C\to X$. By  \cref{c:descent-along-base-change}, this is effective and thus defines a Higgs bundle on $X$. 
The functoriality of this construction and all desired compatibilities now follow from those for the correspondence on $X_C$ using that the functor from \'etale $G$-torsors to $v$-$G$-torsors is fully faithful by \cref{l:fully-faithful-functor-torsors}, and thus morphisms can be defined $v$-locally.

Similarly, it suffices to check that the morphism between cohomologies in part 3 is an isomorphism after an extension of base-field.
This completes the proof in general.
\end{proof}	

The construction of \cref{t:local-paCS} is clearly compatible with that of \cref{t:main-thm-for-both-O-and-O^+}, in the sense that after passing to isomorphism classes and sheafifying on $X_{\et}$, the functor $LS$ computes the map $\Psi$. This completes the last missing piece of the proof of \cref{t:main-thm-for-both-O-and-O^+}: The above descent shows that $\Psi$ is an isomorphism over general perfectoid base fields. \qed

\subsection{Relation to Faltings' local $p$-adic Simpson correspondence}\label{s:rel-loc-corresp}\label{s:comparison-to-AGT}

We now elaborate on the comparison of our correspondence in the case of smooth rigid spaces and $G=\GL_n$ to Faltings' local correspondence as studied by Abbes--Gros--Tsuji.

The idea to build a global $p$-adic Simpson correspondence out of  local correspondences between small objects goes back to \cite{Faltings_SimpsonI}, and instances of such a correspondence for $G=\GL_n$ are now known in good generality: For example, it is known for a certain class of log schemes defined over discretely valued $k$ due to Faltings \cite{Faltings_SimpsonI}, Abbes--Gros \cite{AGT-p-adic-Simpson} and completed by Tsuji \cite{Tsuji-localSimpson}, as well as in cases of good reduction, due to Wang \cite{Wang-Simpson}. Let us also mention in this context the ``q-deformed'' version due to Morrow--Tsuji \cite{MorrowTsuji}.

\medskip

The first difference of these to our version is the technical setup in which the correspondence is formulated:
The approach of Faltings/Abbes--Gros/Tsuji is rooted in the setting of Faltings' \mbox{$p$-adic} Hodge theory, and deals with certain log schemes with toroidal singularities, defined over discretely valued base fields $k$. On these one defines the category of generalised representations, and the local correspondence then essentially happens after the base change to the completed algebraic closure $K$ of $k$. Let us refer to this as the ``arithmetic'' setup.

An alternative technical foundation (of course inspired by Faltings') is that of Scholze's $p$-adic Hodge theory \cite{Scholze_p-adicHodgeForRigid}, which instead works with smooth rigid spaces $X$ over algebraically closed $K$. For the Hodge--Tate comparison (on which the $p$-adic Simpson correspondence is based), this is possible without assuming that $X$ admits a model over a discretely valued field. One then uses locally perfectoid constructions, like the pro-\'etale site, or the category of diamonds, which allow one to reinterpret generalised representations as locally free sheaves on $X$. Let us call this the ``geometric'' setup.

In many situations, it is possible to pass from the ``arithmetic setup'' in the above sense to the ``geometric setup'' by analytifying and passing from $k$ to $K$. This is not to say that one setup is more general than the other, for example Faltings' setting allows for more flexibility in the regularity conditions imposed on $X$. Moreover, there are of course many variations in between, for example many authors use Scholze's technical language but work with varieties that allow models over discretely valued fields, for various good reasons: For example on can then expect a ``de Rham'' side to the non-abelian picture, as in \cite{LiuZhu_RiemannHilbert}.

\medskip

The second difference is that we use a different notion of smallness, related to the difference between \cite[Definition~13.1 vs 13.2]{AGT-p-adic-Simpson}: In terms of generalised representations, the definition used by Faltings/Abbes--Gros/Tsuji is that the underlying module is projective and \textit{generated} by elements on which the action becomes trivial $\bmod p^{\alpha}$, whereas the definition that we use here is that the module is already finite free with trivial action $\bmod p^{\alpha}$.

The proof of the local Simpson correspondence for $G=\GL_n$ in Faltings' setting is quite subtle, as explored in detail by Abbes--Gros--Tsuji (see in particular \cite[II.14]{AGT-p-adic-Simpson}, \cite[Remark~2.3]{Tsuji-localSimpson}). The major difficulty  is that when using their notion of smallness, the construction  requires an additional technical result, which in the language of Abbes--Gros is that ``small representations are Dolbeault''. This was the missing piece provided by \cite{Tsuji-localSimpson}. We choose to sidestep these problems by way of our more restrictive notion of smallness, which also seems more natural in our general setting of $G$-torsors.

\medskip

If we ignore all of these technical differences, the main novelty in \cref{t:local-paCS} is of course the generalisation from $\GL_n$ to rigid groups $G$ and from smooth rigid spaces $X$ to the relative setting of perfectoid families thereof. A minor additional point is the generalisation from algebraically closed $K$ to any perfectoid base field over $\Q_p$. 
 
 \subsection{Relation between sheafified and local correspondence}

	One could in principle deduce a weaker version of \cref{t:main-thm-for-both-O-and-O^+} from \cref{t:local-paCS} by passing to isomorphism classes, using that $G$-bundles and $G$-Higgs bundles are \'etale-locally small. But it is important for our purposes that \cref{t:main-thm-for-both-O-and-O^+} is stronger than this version, as it is ``more canonical and more functorial in $X$'' in the following ways:
	
	\medskip
	
	\textbf{Dependence on the chart:} As mentioned in the introduction, in contrast to the complex case, the $p$-adic Simpson correspondence is in general non-canonical and will depend on certain choices. The situation for the local $p$-adic Simpson correspondence is slightly better, but still quite subtle: It still depends on the choice of chart up to non-canonical isomorphism  (see also \cite[Remark~13.3]{Tsuji-localSimpson}). This means that it is in general not clear how to globalize from the local case, since one has no canonical glueing data on overlaps. Other versions of the local correspondence have similar restrictions.
		Indeed, there are conceptual reasons why one cannot expect a completely canonical correspondence in general: If this existed, then by functoriality one would be able to glue the local correspondences to get a global correspondence without any smallness assumptions, which is known not to exist.
		
		In contrast, the advantage of the ``sheafified correspondence'' of \cref{t:main-thm-for-both-O-and-O^+} is that it is completely canonical and functorial with respect to localisation. In particular, the local correspondence, \cref{t:main-thm-for-both-O-and-O^+} holds for any smooth rigid space. This functoriality is crucial for our purposes in this article, and especially for constructing the Hitchin morphism $\wt {\mathcal H}$.
		
		\medskip
		
	\textbf{Functoriality:} Given a morphism of smoothoids $h:X_1\to X_2$, it is not always  possible to find compatible toric charts, and hence to compare the  local correspondences, even \'etale locally. The functoriality property of \cref{t:main-thm-for-both-O-and-O^+}  is stronger and includes such morphisms.
		
		\medskip
	
	\textbf{Smallness:} A third way in which \cref{t:main-thm-for-both-O-and-O^+} is more canonical than \cref{t:local-paCS} is that for general rigid groups $G$, there is no canonical notion of ``smallness'' of $G$-torsors.

\section{Applications to $v$-vector bundles on rigid spaces}\label{s:applications}
We now give several immediate applications of \cref{t:main-thm-for-both-O-and-O^+} and \cref{t:local-paCS}. Let $G$ be any rigid group.
Our first application is that the explicit description of $R^1\nu_{\ast}G$ lets us deduce $v$-descent criteria for $G$-torsors. By this we mean methods to answer the following question:

\begin{Question}\label{q:v-descent}
	Let $f:X'\to X$ be a $v$-cover of smoothoid (e.g.\ smooth rigid or perfectoid) spaces. What are conditions on $f$ that ensure descent of \'etale $G$-torsors along $f$?
\end{Question}
As it is tautological that we have descent of $v$-$G$-torsors along $f$, it is equivalent to ask how we can tell whether $v$-$G$-torsors on $X$ are already \'etale-locally trivial.

Examples of situations where this question arises include moduli stacks of \'etale vector bundles, but also the study of automorphic sheaves like bundles of overconvergent $p$-adic modular forms defined via descent from perfectoid Shimura varieties of infinite level.
\subsection{Criteria for $v$-vector bundles to be \'etale-locally trivial}
Since \cref{q:v-descent} is \'etale-local, we can use either of \cref{t:main-thm-for-both-O-and-O^+} and \cref{t:local-paCS} to study it. We shall mainly use the former as it is slightly more convenient in this setting, e.g.\ it does not require additional choices. We begin with the following immediate consequence:
\begin{Corollary}\label{c:Leray-sequence}
	For any smoothoid space $X$ and a rigid group $G$ over $K$, the Leray sequence for $\nu:X_v\to X_{\et}$ induces an exact sequences of pointed sets, functorial in $X$ and $G$,
\[
1 \to H^1_{\et}(X,G)\to   H^1_{v}(X,G)\to  \Higgs_G(X).
\]
\end{Corollary}
We deduce  the following criterion for $v$-descent, a generalisation of \cite[Corollary~1.5]{heuer-v_lb_rigid}:

\begin{Corollary}\label{c:analyticity-criterion}
Let $V$ be a $G$-torsor on $X_v$. Let $U\to X$ be any \'etale map with Zariski-dense image. Then $V$ is \'etale, i.e.\ comes from a $G$-torsor on $X_{\et}$, if and only if $V|_U$ is.
\end{Corollary}
In particular, if $X$ is connected, this means that to prove that $V$ is \'etale on $X$, it suffices to prove this on \textit{any} non-empty open subspace $U\subseteq X$. This is a priori quite surprising.
\begin{proof}
By functoriality in Theorem~\ref{t:main-thm-for-both-O-and-O^+}, the restriction to $U$ fits into a commutative diagram
\[\begin{tikzcd}
1 \arrow[r] &  \arrow[d] \arrow[r] H^1_{\et}(X,G)&  H^1_{v}(X,G)\arrow[d] \arrow[r] &  \Higgs_G(X)\arrow[d] \\
1 \arrow[r] &  \arrow[r] H^1_{\et}(U,G)&  H^1_{v}(U,G)\arrow[r] &  \Higgs_{G}(U).
\end{tikzcd}\]
Thus it suffices to see that the restriction map on the right
has trivial kernel. This is precisely the statement of Lemma~\ref{l:Zariski-dense-restriction-on-Higgs} applied in the case of Example~\ref{ex:pullback-wtOm-inj}.1.
\end{proof}
We also get the following $v$-descent criterion, giving a satisfactory answer to \cref{q:v-descent}:

\begin{Corollary}\label{c:descent-criterion}
	Let $f:X'\to X$ be a $v$-cover of affinoid smoothoid spaces. Then descent of \'etale $G$-torsors along $f$ is effective if and only if $f^{\ast}:H^0(X,\wtOm_X)\to H^0(X',\wtOm_{X'})$ is injective.
\end{Corollary}
\begin{proof}
	Any descent datum for an \'etale $G$-torsor $E$ along $f$ defines a $v$-topological $G$-torsor on $X$ whose pullback to $X'$ is $E$. The result thus follows from a similar diagram as in the proof of \cref{c:analyticity-criterion}, again using \cref{l:Zariski-dense-restriction-on-Higgs}: Here by \cref{p:RnuO-for-small-dmd}, the sheaf $\wtOm_X$ is a vector bundle, so $\wtOm_X\to f_{\ast}\wtOm_{X'}$ is injective if and only if it is injective on global sections.
\end{proof}
This generalises \cref{c:descent-along-base-change}, where $f$ is a base-change $X_{K'}\to X_K$. Another case is:

\begin{Corollary}\label{c:test-v-local-analyticity-criterion}
	Let $X$ be a smooth rigid space, let $Y$ be a perfectoid space, and let $Z\to X\times Y$ be \'etale. Then a $G$-torsor on $Z_v$ is  \'etale if it is after pullback to a perfectoid $v$-cover $Y'\to Y$.
\end{Corollary}

Applying the same argument as in \cref{c:analyticity-criterion} to the case of Example~\ref{ex:pullback-wtOm-inj}.2 shows:
\begin{Corollary}\label{c:analyticity-criterion-G->GL_n}
	Let $G\hookrightarrow G'$ be an injective homomorphism of rigid groups.
Let $V$ be a $G$-torsor on $X_v$. Then $V$ is \'etale-locally trivial if and only if the $G'$-torsor $V\times^{G}G'$ is.
\end{Corollary}
 As a special case of the Corollary is the map $\GL_n(\O^+)\hookrightarrow \GL_n(\O)$, for which this shows:
 
\begin{Corollary}\label{c:generically-et-loc-free-implies-et-loc-free}
	Let $E^+$ be a finite $v$-locally free $\O^+$-module on $Z_v$ such that $E:=E^+\tf$ is \'etale-locally free. Then $E^+$ is \'etale-locally free. 
\end{Corollary}
However,  \cref{ex:cong-over-Q_p-but-not-Z_p} shows that the map $R^1\nu_{\ast}G\to R^1\nu_{\ast}G'$ is not in general injective.

\begin{Remark}
	Another application of 
\cref{c:analyticity-criterion-G->GL_n} shows that for a linear algebraic group $G$ with embedding $G\hookrightarrow \GL_n$, a $G$-torsor $V$ on $X_v$ is \'etale-locally trivial if and only if the associated $v$-vector bundle $\GL_n\times^GV$ is. The more general question whether two $G$-torsors $V_1$, $V_2$ on $X_v$ are \'etale-locally isomorphic if the associated $v$-vector bundles are isomorphic for every representation of $G$ is more subtle, and \cref{t:main-thm-for-both-O-and-O^+} translates this into a question about simultaneous conjugation that is related to Steinberg's conjugacy conjecture \cite{Steinberg_conjugacy_conjecture}.
\end{Remark}

We finish with two results about $v$-vector bundles, generalising \cite[Corollary~3.6]{heuer-v_lb_rigid}:
\begin{Corollary}
Let $V_1\hookrightarrow V_2$ be an injective morphism of $v$-vector bundles on a smoothoid space $Z$. If $V_2$ is \'etale-locally free, then so is $V_1$.
\end{Corollary}
\begin{proof}
We can work locally and assume that $V_1$, $V_2$ are small. Then \cref{t:local-paCS}  translates this into an inclusion of Higgs bundles $(E_1,\theta_1)\subseteq (E_2,\theta_2)$, and $\theta_1$ vanishes if $\theta_2$ does. 
\end{proof}

Second, as a consequence of \cref{t:local-paCS}.3, we immediately see on the Higgs side:

\begin{Corollary}
	Let $X$ be a smoothoid space over $K$ and let $V$ be a $v$-vector bundle on $X$. Let $\nu:X_v\to X_{\et}$ be the natural morphism of sites and let $n \in \N$.
	\begin{enumerate}
		\item The $\O_X$-module
		$R^n\nu_{\ast}V$ vanishes for $n\geq \dim X$.
		\item If $X$ is a smooth rigid space, then the $\O_X$-module
		$R^n\nu_{\ast}V$
		on $X_{\et}$ is coherent.
	\end{enumerate}
\end{Corollary}
This reproves a result of Kedlaya--Liu \cite[Theorem 8.6.2.(a)]{KedlayaLiu-II} for pseudo-coherent modules in the special case of $v$-vector bundles.
For paracompact rigid $X$, it follows from Grothendieck vanishing \cite[Corollary 2.5.10]{deJongvdPut} that $H^i_v(X,V)=0$ for $i>2\dim X$. A related result has been obtained in a similar way by Min--Wang \cite[Corollary~1.10]{MinWang22} for proper $X$  over a discretely valued field $K_0$. We note that the Corollary also illustrates the failure of $R\nu_{\ast}$ to be conservative: If the associated Higgs complex is exact, then $R\nu_{\ast}V=0$.

\subsection{$v$-stacks of \'etale $G$-torsors}
Let $f:X\to \Spa(K)$ be a smooth rigid space over $X$.
As an application of \cref{c:test-v-local-analyticity-criterion}, we can now prove that the prestack of \'etale vector bundles on $X$ relatively over the $v$-site over $K$ is a stack. This is one of our main motivations for introducing smoothoid spaces.

We shall freely use the language of $v$-stacks from \cite[\S9]{etale-cohomology-of-diamonds}. We now recall their definition, adapted to our setting of perfectoid spaces over a fixed perfectoid field $K$. In other words, we will always consider $v$-stacks with a fixed structure map to $\Spd(K)$:
\begin{Definition}\label{d:v-stack}
	\begin{enumerate}
\item A prestack over $K$ is a contravariant functor $F: \Perf_K\to \mathrm{Groupoids}$. 
\item A prestack $F$ is called a \textbf{$v$-stack} if it satisfies $v$-descent, i.e.\ if for each $v$-cover $Y'\to Y$ with projections $\pi_{1},\pi_2:Y'\times_YY'\to Y'$, the following functor is an equivalence:
\[ F(Y)\to \{(s,\alpha)|s\in F(Y'), \alpha:\pi_1^{\ast}s\isomarrow  \pi_2^{\ast}s \text{ such that cocycle condition holds}\}\]

\item A $v$-stack $F$ is called \textbf{small} if there is a surjection $h:S\to F$ from a perfectoid space for which $R=S\times_FS$ is a small $v$-sheaf, i.e.\ there is a surjection $S'\to R$ of $v$-sheaves from a perfectoid space $S'$. It is not difficult to see that one gets an equivalent definition if $S$ and $S'$ are just assumed to be diamonds.
\end{enumerate}
\end{Definition}
\begin{Definition}
Let $G$ be any rigid group over $K$, for example $G=\GL_n$.
\begin{enumerate}
\item We denote by $\cBun_{G,v}$ the prestack on $\Perf_K$ that sends a perfectoid space $Y$ to the groupoid of $v$-$G$-bundles on $(X\times Y)_{v}$, where as usual the fibre product is over $\Spa(K)$.
\item We denote by $\cBun_{G,\et}$ the prestack on $\Perf_K$ that sends a perfectoid space $Y$ to the groupoid of \'etale  $G$-bundles on $(X\times Y)_{\et}$.
\item We  denote by $\cHiggs_G$
the prestack on $\Perf_K$ that sends a perfectoid space $Y$ to the groupoid of $G$-Higgs bundles on $(X\times Y)_{\et}$.
\end{enumerate}
\end{Definition}

We now use the descent criteria from the last section to prove our next main result:
\begin{Theorem}\label{t:Higgs-is-vstack}
All of the prestacks $\cBun_{G,v}$, $\cBun_{G,\et}$ and  $\cHiggs_G$ are  small $v$-stacks. 
\end{Theorem}
\begin{proof}
Let $g:Y'\to Y$ be a $v$-cover in $\Perf_K$, then $X\times Y'\to X\times Y$ is also a $v$-cover. We clearly have $v$-descent for $v$-topological $G$-bundles, so $\cBun_{G,v}$ is a $v$-stack. 

For $\cBun_{G,\et}$, we need to prove that we have descent of \'etale $G$-bundles along the map $f:X\times Y'\to X\times Y$. By \cref{c:descent-criterion} to Theorem~\ref{t:main-thm-for-both-O-and-O^+}, it suffices to see that the map  $\wtOm(X\times Y)\to \wtOm(X\times Y')$
is injective. This follows from \cref{p:RnuO-for-small-dmd} which shows $\wtOm_{X\times Y}=\wtOm_X\otimes_{\O_X} \O_Y$, and the fact that $g_{\ast}\O_Y\to \O_{Y'}$ is injective as $g$ is a $v$-cover.

The case of $\cHiggs_G$ now follows easily: Given a descent datum on a $G$-Higgs bundle $(E',\theta')$, we first descend the \'etale bundle $E'$ to a $G$-bundle $E$. Then $\mathcal F=\ad(E)\otimes \wtOm_X$, being an \'etale vector bundle by \cref{p:RnuO-for-small-dmd}, satisfies $\nu_{\ast}\nu^{\ast}\mathcal F=\mathcal F$. In other words, $\mathcal F$ already satisfies the sheaf property for the $v$-topology on $\Perf_K$, so $\theta$ can be defined $v$-locally. By the same argument, the vanishing of the section $\theta\wedge\theta\in \ad(E)\otimes \wtOm^2$ can be checked $v$-locally.

That the $v$-stack $\cBun_{G,\et}$ is small can be seen similarly as in \cite[Proposition III.1.3]{FS-Geometrization}: As explained there, it suffices to prove that for any affinoid perfectoid space $S=\Spa(R,R^+)$ over $K$, we have \[\cBun_{G,\et}(S)=\textstyle\varinjlim_{i\in I}\cBun_{G,\et}(S_i)\]
where $S_i=\Spa(R_i,R_i^+)$ and $(R_i)_{i\in I}$ ranges through the topologically countably generated perfectoid subalgebras of $R$. To see that this holds, let $E\subseteq \cBun_{G,\et}(R)$ be any \'etale $G$-torsor, and let $U\to X\times S$ be a trivialising standard-\'etale cover so that $G$ corresponds to a class in the kernel of $G(U)\to G(U\times_{X\times S}U)$. Note that $U$ descends to a standard-\'etale cover $U_i\to X\times S_i$ for any $i\gg 0$.  Then $\O(U)=\varinjlim \O(U_i)$ since any function of $U$ can be presented using countably many functions of $R$ which are thus already contained in some $R_i$.  After refining $U$, we may assume that $U$ is a finite disjoint union of affinoid opens that factor through an affinoid open $\Spa(A,A^+)\subseteq G$ of topologically finite presentation, so that $E$ is already in the image of
$\varprojlim G(U_i)\to \varprojlim G(U)$.
As $\varinjlim G(U_i\times_{X\times S}U_i)=G(U\times_{X\times S}U)$ is injective, it follows that $E$ descends to a $G$-torsor $E_i$ on $X\times S_i$ for $i\gg 0$, as desired.

The same approximation arguments show that $\ad(E)\otimes \wtOm(X\times S)=\varinjlim \ad(E)\otimes \wtOm(X\times S_i)$, so that any Higgs field $\theta$ on $E$ already descends to a Higgs field on $E_i$ for $i\gg 0$. Thus 
\[\cHiggs_{G,\et}(S)=\textstyle\varinjlim_{i\in I}\cHiggs_{G,\et}(S_i),\]
which shows that $\cHiggs_G$ is small. We deduce the case of $\cBun_{G,v}$ from this using \cref{t:local-paCS}: Let $V$ be a $v$-$G$-bundle on $X\times S$. By \cref{l:v-bundle-locally-small}, this becomes small after passing to an \'etale cover of $X$ by toric affinoid subspaces $U\to X$. By functoriality of the local correspondence in the toric chart, we deduce that the restriction of $V$ to $U\times S$ descends to some $U\times S_i$. Since morphisms of Higgs bundles also descend from $U\times S$ to some $U\times S_i$, we see that we can descend the gluing data, and deduce that $V$ descends to some $X\times S_i$.
\end{proof}

While we believe that the $v$-stack $\cBun_{G,\et}$ is of independent interest, its role in the following is that we use it to understand $\cHiggs_{G}$. For example, for commutative $G$ we have:

\begin{Lemma}\label{l:stacks-for-commutative-G}
	Let $G$ be a commutative rigid group. Let $X$ be a proper smooth rigid space. Then we have a canonical isomorphism of $v$-stacks on $\Perf_K$
	\[ \cHiggs_{G}=\cBun_{G,\et}\times H^0(X,\mg\otimes \wtOm_X)\otimes_K \G_a,\]
	where the second factor on the right hand side is a rigid vector group.
	In particular, all of $\cBun_{G,\et}$,  $\cBun_{G,v}$ and $\cHiggs_{G}$ then have the structure of a group $v$-stack.
\end{Lemma}
\begin{proof}
	In the proper case, we have $H^0(X\times Y,\wtOm)=H^0(X,\wtOm)\otimes \O(Y)$ by \cite[Proposition~4.1]{heuer-diamantine-Picard}. The description therefore follows from \cref{r:higgs-for-commutative-G}. The group structure on $\cBun_{G,\et}$ and $\cBun_{G,v}$ comes from the tensor product on the category of $G$-bundles given by sending $G$-torsors $E_1,E_2$ to the pushout of $E_1\times E_2$ along the multiplication $G\times G\to G$.
\end{proof}

A similar description is still possible if we drop the assumption that $X$ is proper. In this case, the second factor might no longer be a rigid group, but it is still a diamond:
\begin{Lemma}\label{l:pushforward-of-vb-is-diamond}
	Let $f:X\to \Spa(K)$ be a smooth rigid space. Let $E$ be an analytic vector bundle on $X$. Then the $v$-sheaf $f_{\ast}E$ on $\Perf_K$ defined by $Y\mapsto E(X\times Y)$ is a diamond.
\end{Lemma}
\begin{proof}
	Let us write $\mathcal F_{X,E}:=f_{\ast}E$.
	Given any cover $X=\cup U_i$, the map $\mathcal F_{X,E}\hookrightarrow \prod \mathcal F_{U_i,E}$ is injective, where $\mathcal F_{U_i,E}$ is the restriction to $U_i$ . By \cite[Proposition 11.10]{etale-cohomology-of-diamonds}, we can thus work locally on $X$ and assume that $E$ is free, reducing to $E=\O$. By a result of Achinger, we can moreover assume that $X$ admits a finite \'etale map $X\to \mathbb D$ to some rigid polydisc \cite[Proposition 6.6.1]{Achinger_Kpi1}\cite[Corollary B.5]{Zavyalov_acoh}. In this case, 
	$\O(X\times Y)=\O(\mathbb D\times Y)\otimes_{\O(\mathbb D)}\O(X)$, so we are reduced to the case that $X=\Spa(K\langle T_1,\dots,T_n\rangle)$. In this case, the result follows because $\mathcal F_{\O,\mathbb D}\subseteq \prod_{\N}\A^n$, and the latter is a diamond by \cite[Lemma 11.22]{etale-cohomology-of-diamonds}.
\end{proof}

\section{The Hitchin morphism for $v$-vector bundles}
We now give the second main application of the sheafified correspondence \cref{t:main-thm-for-both-O-and-O^+}: The construction of the Hitchin morphism on the Betti side.
To motivate the construction, we reiterate that our goal in this series is to investigate to what extent non-abelian $p$-adic Hodge theory can be understood in terms of a comparison of  the moduli $v$-stacks introduced in the last section.
As in the complex theory, it is clear from the case of $G=\G_m$ which we explain in detail in \S\ref{s:Hitchin-for-Gm} that we cannot expect  $\cHiggs_{G}$ and $\cBun_{G,v}$ to be isomorphic as $v$-stacks. Instead, as mentioned in \cref{t:part-II-twist-for-curves} in the introduction, what we will show in part II is that for $G=\GL_n$ and $X$ a smooth proper curve, the two $v$-stacks are twists of each other via two natural morphisms to the same rigid analytic base: the Hitchin base. 

The goal of this section is to introduce these morphisms and prove some first properties. We also explain why we believe that the Hitchin morphism on the Betti side is interesting for studying representations of the \'etale fundamental group.

As before, let $K$ be any perfectoid extension of $\Q_p$. Throughout this section, let $X$ be a smooth rigid space over $K$. We discuss the the Hitchin morphisms on the Higgs side.

\subsection{The Hitchin morphism for Higgs bundles in the case of $G=\GL_n$}\label{s:Hitchin-for-GLn}
We describe the Hitchin morphism for $G=\GL_n$ because this is the main case studied in the $p$-adic Simpson correspondence so far, and the description simplifies significantly in this case.
To simplify notation, let us in this subsection write
\[ \cBun_{n,v}:=\cBun_{\GL_n,v},\quad \cHiggs_{n}:=\cHiggs_{\GL_n}.\]

In  algebraic geometry, the Hitchin morphism for a smooth proper variety $Y$ is a morphism from the stack of Higgs bundles of $Y$ to the Hitchin base, which is a certain affine space depending on $Y$  \cite[cf \S6]{SimpsonModuliII}. The construction is straightforward to adapt to the $p$-adic analytic setting, as we now discuss, roughly following \cite[p.20]{SimpsonModuliII}.

\medskip

Let $Z$ be any smoothoid space and 
let $(E,\theta)$ be a Higgs bundle of rank $n$ on $Z_{\et}$. We regard $\theta$ as a section of $\End(E)\otimes \wtOm$. On any $U\in Z_{\et}$ where $E$ becomes trivial, choose an isomorphism $\psi:E|_U\cong \O^n_U$, then $\theta$ defines a homogeneous element of degree 1 in $M_n(H^0(U,\Sym\wtOm))$ and we can consider its characteristic polynomial in $H^0(U,\Sym\wtOm[T])$. As this is independent of the choice of $\psi$, this glues to a polynomial 
\[ \chi_{E,\theta}= T^n+a_{1}T^{n-1}+\dots + a_{n} \in H^0(Z,\Sym\wtOm[T])\]
defined over all of $Z$, 
where $a_k\in H^0(Z,\Sym^k\wtOm)$ for $k=1,\dots,n$.

\begin{Definition}
Since $\chi_{E,\theta}$ only depends on the isomorphism class of $(E,\theta)$, this defines a natural map from the set of isomorphism classes of Higgs bundles on $Z_{\et}$ of rank $n$
\[H:\{\text{Higgs bundles on $Z_{\et}$ of rank $n$}\}/\!\sim\to \textstyle\bigoplus\limits_{k=1}^n H^0(Z,\Sym^k\wtOm), \quad (E,\theta)\mapsto \chi_{E,\theta}.\]
As $H$ is functorial in $Z$ by construction, it induces a morphism of sheaves on $Z_{\et}$
\begin{equation}\label{eq:Higgs-as-morphism-of-shvs}
h_Z:\Higgs_n\to \textstyle\bigoplus\limits_{k=1}^n \Sym^k\wtOm.
\end{equation}
\end{Definition}

\begin{Remark}
If $\dim Z=1$, then $\Sym^i\wtOm=\wtOm^{\otimes i}$, and we get a simpler description of $\chi_{E,\theta}$. For  this we  consider for each $k=1,\dots,n$ the $k$-th wedge product of the Higgs field
\[ \wedge^k \theta:\wedge^kE\to  \wedge^k(E\otimes\wtOm)=\wedge^kE\otimes\wtOm^{\otimes k},\]
then $a_k:=\tr( \wedge^k \theta)$ is the $k$-th coefficient of $\chi_{E,\theta}$. This matches Hitchin's definition \cite{HitchinFibration}.
\end{Remark} 

We now return to the smooth rigid space $X$, for which we can assemble the various maps $h_Z$ on $Z=X\times Y$ for $Y\in \Perf_K$ to a Hitchin  morphism in terms of moduli stacks, using:

\begin{Definition}
The \textbf{Hitchin base} of $X$ is the $v$-sheaf $\mathcal A_{n}=\mathcal A_{X,n}$ on $\Perf_K$ defined by
\[  \mathcal A_{n}\colon Y\mapsto  \textstyle\bigoplus\limits_{k=1}^nH^0(X\times Y,\Sym^k\wtOm_{X\times Y}).\]
In general, $\mathcal A_{n}$ is a diamond by \cref{l:pushforward-of-vb-is-diamond}, but if $X$ is proper, then by \cref{p:RnuO-for-small-dmd} and  \cref{l:section-vb-on-product-of-proper-with-perfectoid} below, $\mathcal A_{n}$ is represented by the rigid space $ \bigoplus_{k=1}^nH^0(X,\Sym^k\wtOm)\otimes_K\G_a$.
\end{Definition}

\begin{Lemma}\label{l:section-vb-on-product-of-proper-with-perfectoid}
Let $X$ be a smooth proper rigid space and let $Y$ be a perfectoid space. Let $\pi:X\times Y\to X$ be the projection. Then for any vector bundle $E$ on $X_{\et}$ and $i\geq 0$, we have
\[H^i_{\et}(X\times Y,\pi^{\ast}E)=H^i_{\et}(X,E)\otimes_K \O(Y).\]
\end{Lemma}
\begin{proof}
This follows from a smoothoid version of ``cohomology and base change'', or by a direct computation in \cH-cohomology as in \cite[Proposition~3.31]{heuer-diamantine-Picard}.
\end{proof}

\begin{Definition}
Using the small $v$-stack of Higgs bundles $\cHiggs_n$ from \cref{t:Higgs-is-vstack}, we define a morphism of $v$-stacks over $K$, the \textbf{Hitchin morphism} for $\GL_n$
\begin{equation}\label{eq:Higgs-as-morphism-of-stacks}
\mathcal H: \cHiggs_n\to\mathcal A_{n}
\end{equation}
as follows: We compose the map $\cHiggs_{n}(X)\to \Higgs_n(X\times Y)$, given by passing to isomorphism classes and sheafifying on $(X\times Y)_{\et}$, with the map $h_Z$ from \cref{eq:Higgs-as-morphism-of-shvs}.

\end{Definition}

\begin{Remark}\label{r:product-structure-on-Hitchin-base}
We think of elements $(a_1,\dots,a_n)\in \mathcal A_{n}$ of the Hitchin base as monic polynomials 
$f(T)=T^n+a_{1}T^{n-1}+\dots+a_n$.
For example, this defines a natural map of $v$-sheaves
\[\cdot:\mathcal A_{n}\times \mathcal A_{m}\to  \mathcal A_{n+m}, \quad f(T),g(T)\mapsto (f\cdot g)(T).\]
The image of a direct sum of Higgs bundles $E_1\oplus E_2$ under $\mathcal H$ is then $\H(E_1)\cdot \H(E_2)$.
\end{Remark}

\subsection{The Hitchin morphism for general $G$}
We now describe the construction of the Hitchin morphism for general rigid groups $G$. Again, this is a fairly straightforward adaptation to $v$-stacks of the algebraic construction,  for which we roughly follow \cite{Ngo_Hitchin}:
Let $(R,R^+)$ be any Huber pair over $(K,K^+)$ and let $V$ be a finite free $R$-module of rank $d$. Write $V\otimes\A^1_R$ for the associated affine space over $\Spec(R)$, where $\A^1_{R}$ is the affine line over $\Spec(R)$ considered as a scheme. Any element of $\mg\otimes V$ defines a morphism 
$V^\vee\otimes \A^1_{R}\to \mg\otimes  \A^1_{R}$ of vector groups over $R$, where $V^\vee$ is the $R$-linear dual.
Composed with the quotient by $G$ in the sense of geometric invariant theory
\[ \mg\otimes  \A^1_{R} \to \mg\otimes  \A^1_{R}\sslash G :=\Spec(R[\mg]^G),\]
we obtain a $\G_m$-equivariant morphism
$V^\vee\otimes \A^1_{R}\to \Spec(R[\mg]^G)$.
\begin{Definition}\label{d:prep-of-Hitchin-base}
Let $A_{V,G,R}$ be the $v$-sheaf over $\Spa(R,R^+)$ sending a perfectoid algebra $(S,S^+)$ to the set of $\G_m$-equivariant morphisms of $R$-schemes
$V^\vee\otimes \A^1_{S}\to \Spec(R[\mg]^G)$.
One verifies that this is a $v$-sheaf, using that $\O$ on $\Perf_K$ is a $v$-sheaf.
\end{Definition}

Passing from schemes to $v$-sheaves, we obtain a natural morphism of $v$-sheaves on $\Perf_K$
\[
\mg\otimes_K V\otimes_R \G_{a,R}\to A_{V,G,R}
\]
that is functorial in $R$ and equivariant for the natural $G\times \GL_V$-action where  $G$ acts via the adjoint action on $\mg$ on the left and trivially on the right, and $\GL_V$ acts on $V$ on  both sides.

Let now $Z$ be any smoothoid space over $K$. We apply this construction to $R=\O(U)$ for any toric open subspace $U\to Z$ in $Z_{\et}$ and $V:=H^0(U,\wtOm_Z)$. We thus obtain natural maps of $v$-sheaves
\[ \mg \otimes H^0(U,\wtOm_U)\otimes_R \G_{a,R}\to A_{V,G,\O(U)}\]
that are still $G\times \GL_V$-equivariant. For varying $U$, the $A_{V,G,\O(U)}$ clearly glue to a sheaf $A_{G}$ over $Z_{\et}$. Due to the $G$-invariance, the above maps then glue to a natural morphism of sheaves on  $Z_{\et}$
\begin{equation}\label{eq:Hitchin-morphims-general-G-sheafified}
h_Z:\Higgs_G\to A_G.
\end{equation}

Let now $X$ be a smooth rigid space and consider $Z=X\times Y$ for test objects $Y\in \Perf_K$.

\begin{Definition} The $v$-sheaf $\mathcal A_G$ (or $\mathcal A_{G,X}$) on $\Perf_K$ defined by 
	$\mathcal A_G(Y):=A_G(X\times Y)$
	 is the \textbf{Hitchin base for $G$} and $X$. 
	 It is clear from the construction that the formation of $\mathcal A_{G,X}$ is functorial in both $G$ and $X$.
	 We then easily verify that $h_{X\times Y}$ is functorial in $Y$. 
\end{Definition}
In general, the $v$-sheaf $\mathcal A_G$ is a diamond by \cref{l:pushforward-of-vb-is-diamond}. However, we have the following:

\begin{Lemma}\label{l:Hitchin-base-is-rigid-space-if-G-red-or-com}
	Assume that $X$ is proper, and that  $G$ is a rigid group such that the Lie algebra $\mathfrak g$ satisfies
	$K[\mg]^G=K[u_1,\dots,u_n]$
	for some homogeneous generators $u_i$ of degree $e_i$, for example by Chevalley's Restriction Theorem we can take split reductive $G$, but we can also take commutative $G$.
 	Then $\mathcal A_G$ is represented by an affine rigid space over $K$, namely
\[  \mathcal A_G\cong \textstyle\bigoplus\limits_{k=1}^nH^0(X,\Sym^{e_k}\wtOm_X)\otimes \G_a.\]
\end{Lemma}
\begin{proof}
	The condition on $G$ ensures that for any toric affinoid $U\subseteq X$ and $R=H^0(U\times Y,\wtOm)$, we have
	$A_{V,G,R}\cong \prod_{k=1}^n(\Sym_R^{e_k}V)\otimes\A^1$.
	Gluing for a cover of $X$ by such $U$, we see that $A_G= \prod_{k=1}^nH^0(X\times Y,\Sym^{e_k}\wtOm_{X\times Y})\otimes\A^1$. We now use that $X$ is proper and therefore $H^0(X,\Sym^{e_k}\wtOm_X)$ is a finite dimensional $K$-vector space. It follows from \cref{l:section-vb-on-product-of-proper-with-perfectoid} that 
	\[H^0(X\times Y,\Sym^{e_k}\wtOm_{X\times Y})=H^0(X,\Sym^{e_k}\wtOm_X)\otimes \O(Y),\]
	and the right hand side is indeed represented by the rigid space $H^0(X,\Sym^{e_k}\wtOm_X)\otimes \G_a$.
\end{proof}

\begin{Definition}
Composing the morphism of $v$-sheaves $h_{X\times Y}$ from \cref{eq:Hitchin-morphims-general-G-sheafified} with the sheafification $\cHiggs_G(Y)\to \Higgs_G(X\times Y)$ for varying $Y\in \Perf_K$, we get the \textbf{Hitchin morphism}
\begin{equation}\label{eq:Hitchin-morphims-general-G}
\mathcal H:\cHiggs_G\to \mathcal A_G,
\end{equation}
a morphism of small $v$-stacks on $\Perf_K$ that is functorial in $G$ and $X$.
\end{Definition}
For $G=\GL_n$, this recovers the earlier description since the coefficients of the characteristic polynomial generate $k[\mg]^G$. Indeed, these are mapped to the elementary symmetric polynomials under the restriction map $k[\mg]^G\to k[\mathfrak t]^W$ in Chevalley's Restriction Theorem. 
\begin{Example}\label{ex:commutative-G}
	If $G$ is commutative, then we have $K[\mg]^G=K[\mg]$. If $X$ is proper, we then have $\mathcal A_G=\mathfrak g\otimes \wtOm_X$. In any case, $\mathcal H$ is for such $G$ simply the projection to the second factor
	$\cHiggs_G\isomarrow\cBun_{G,\et}\times \mathcal A_G\to \mathcal A_G$
	where the first map is from \cref{l:stacks-for-commutative-G}.
\end{Example}

\subsection{The Hitchin morphism on the Betti side}
On the other side of the $p$-adic Simpson correspondence, we now construct the promised Hitchin morphism for the stack of $v$-topological $G$-bundles for any rigid group $G$ over $K$.

Let $Z$ be any smoothoid space. Combining the isomorphism $\HTlog$ from Theorem~\ref{t:main-thm-for-both-O-and-O^+} with the Hitchin map \cref{eq:Higgs-as-morphism-of-shvs} or \cref{eq:Hitchin-morphims-general-G-sheafified}, we obtain a morphism of sheaves on $Z_{\et}$
\[R^1\nu_{\ast}G\xrightarrow{\HTlog} \Higgs_G\xrightarrow{h_Z} \mathcal A_G\]

Let $X$ be any smooth rigid space over $K$.  Using crucially that $\HTlog$ is functorial in $Z$, we see that these morphisms for  $Z=X\times Y$ glue for varying $Y$ to a morphism of $v$-stacks:

\begin{Definition}\label{d:Definition-of-wtH}
The \textbf{Hitchin morphism on the Betti side} is the morphism of $v$-stacks
\[\wtcH:\cBun_{G,v}\to \mathcal A_G\]
defined on $Y\in \Perf_K$ as the composition
\[ \cBun_{G,v}(Y)\to H^1_v(X\times Y,G)\to  R^1\nu_{\ast}G(X\times Y)\xrightarrow{\HTlog} \Higgs_G(X\times Y) \xrightarrow{h_{X\times Y}}  \mathcal A_G(Y)\]
where the first map is the passage from groupoids to sets of isomorphism classes, where  $\nu:(X\times Y)_v\to (X\times Y)_{\et}$ is the natural map, where $\HTlog$ is the isomorphism from \cref{t:main-thm-for-both-O-and-O^+}, and $h_{X\times Y}$ is the sheafified Hitchin morphism on the Higgs side \cref{eq:Higgs-as-morphism-of-shvs} or \cref{eq:Hitchin-morphims-general-G-sheafified}.
\end{Definition}

With this definition, the basic idea for our moduli theoretic approach to $p$-adic non-abelian Hodge theory is to compare $\cBun_{G,v}$ and $	\cHiggs_G$ geometrically via the morphisms
\begin{equation}\label{eq:the-two-Hitchin-maps}
\begin{tikzcd}[row sep =-0.05cm,column sep = 1cm]
{\cBun_{G,v}} \arrow[rd,"\wtcH"] &   \\                        &  \mathcal A_{G}\\
\cHiggs_G\arrow[ru,"\H"']  &          
\end{tikzcd}
\end{equation}

By \cref{l:Hitchin-base-is-rigid-space-if-G-red-or-com}, this is particularly interesting if $X$ is proper and $G$ is split reductive or commutative, as the Hitchin base $\mathcal A_{G}$ is then represented by an affine rigid space.

\begin{Proposition}\label{p:props-of-HHTlog}
The morphism $\wtcH:\cBun_{G,v}\to \mathcal A_{G}$ has the following properties:
\begin{enumerate}
\item $\wtcH$ is functorial in $X\to \Spa(K)$. In particular, if $X$ has a model over a subfield $K_0\subseteq K$, then $\wt {\mathcal H}$ is equivariant with respect to the $\Aut(K|K_0)$-actions on both sides.
\item $\wtcH$ is functorial in $G$. In particular, if $0\to V_1\to V\to  V_2\to 0$ is a short exact sequence of $v$-vector bundles, then in the notation of \cref{r:product-structure-on-Hitchin-base} we have $\wtcH(V)=\wtcH(V_1)\cdot \wtcH(V_2)$.
\item If $G$ is commutative, then $\wtcH$ is a homomorphism of group stacks.
\end{enumerate}
\end{Proposition}
\begin{proof}
The functoriality follows from the analogous properties for Higgs bundles and from functoriality of $\HTlog$. To see part 2, we apply $\HTlog$ for parabolic subgroups $G\subseteq \GL_n$ and use functoriality in $G$. On the Higgs side, the statement then follows from \cref{r:product-structure-on-Hitchin-base}.

If $G$ is commutative, then $\HTlog$ is a homomorphism by \cref{t:main-thm-for-both-O-and-O^+}, and $\mathcal H$ is a homomorphism of group stacks by \cref{ex:commutative-G}.
\end{proof}

The possibility of constructing a Hitchin morphism for $v$-$G$-bundles is a  new idea already for the case $G=\GL_n$ of $v$-vector bundles, which was suggested to us by Scholze in reaction to an earlier version of \cref{t:main-thm-for-both-O-and-O^+}. There are however two abelian special cases that can be described explicitly using known results, as we now explain in detail.
From now on, we assume that $K$ is algebraically closed and that $X$ is a smooth proper rigid space over $K$.
\subsection{The Hitchin morphism on the Betti side for $G=\G_a$}\label{s:Hitchin-for-Ga}

We first treat $G=\G_a$: In this case, \cref{l:Hitchin-base-is-rigid-space-if-G-red-or-com} shows that $\mathcal A_{\G_a}=H^0(X,\wtOm)\otimes\G_a$ and
by \cref{ex:commutative-G}, the map $\H:\cHiggs_{\G_a}\to \mathcal A_{\G_a} $ is simply the projection
\[ \H:\cBun_{\G_a,\et}\times H^0(X,\wtOm)\otimes \G_a\to H^0(X,\wtOm)\otimes \G_a.\]
Here $\cBun_{\G_a,\et}$ parametrises \'etale $\G_a$-torsors on $X\times Y$. By \cref{l:section-vb-on-product-of-proper-with-perfectoid} for $i=1$, it can therefore be identified with the $v$-sheaf represented by the rigid vector group $H^1_{\et}(X,\O)\otimes \G_a$. 

On the other hand, the Hitchin morphism on the Betti side is of the form
\[ \wtcH:\cBun_{\G_a,v}=H^1_v(X,\O)\otimes \G_a\to H^0(X,\wtOm)\otimes \G_a\]
and we easily see by comparing the construction of $\wtcH$ with Scholze's construction of the Hodge--Tate sequence \cite[\S3]{Scholze2012Survey} that this map is the morphism of rigid varieties associated to the Hodge--Tate map $\HT$ from \cref{eq:HT-SES}. This makes precise the idea that $\wtcH$ generalises $\HT$.

We deduce from these explicit descriptions that \cref{eq:the-two-Hitchin-maps} takes the following form:
\begin{Proposition}
Any splitting of  \cref{eq:HT-SES} induces an isomorphism of rigid spaces (and thus of $v$-stacks) $\cBun_{\G_a,v}\isomarrow \cHiggs_{\G_a}$ that commutes with the Hitchin fibrations.
\end{Proposition}

In summary, the Hitchin morphism for $G=\G_a$ encodes the Hodge--Tate sequence.
This is the precise way in which in this moduli-theoretic setting, non-abelian Hodge theory is about generalising a $p$-adic Hodge theoretic result from $\G_a$ to more general rigid groups $G$.

\subsection{The Hitchin morphism on the Betti side for $G=\G_m$}\label{s:Hitchin-for-Gm}
We now turn to $G=\G_m$ to explain the precise relation to \cite{heuer-diamantine-Picard}. In order to get a description in terms of rigid spaces like for $\G_a$, we pass to coarse moduli spaces: Let $\mathbf{Bun}_{n,v}$ be the ``coarse moduli sheaf'' of $\cBun_{n,v}=\cBun_{\GL_n,v}$, by which we mean the functor obtained by passing from groupoids to sets of isomorphism classes and $v$-sheafifying on $\Perf_{K}$. We analogously define the coarse moduli $v$-sheaf $\mathbf{Higgs}_{n}$ of isomorphism classes of Higgs bundles of rank $n$ on  $\Perf_{K}$.
For $n=1$, we thus obtain on the one hand the $v$-Picard functor
\[ \mathbf{Bun}_{1,v}=\mathbf{Pic}_{X,v}\]
introduced in \cite[\S3]{heuer-diamantine-Picard}, and on the other hand the coarse moduli space of Higgs line bundles
\[ \mathbf{Higgs}_{1}=\mathbf{Pic}_{X,\et}\times H^0(X,\wtOm)\otimes \G_a.\]
Both are representable by rigid groups over $K$ if the usual rigid analytic Picard functor is representable, by \cite[Theorem~1.1]{heuer-diamantine-Picard}. For example, for algebraic $X$, $\mathbf{Pic}_{X,\et}$ is the diamondification of the algebraic Picard functor. There is then a natural short exact sequence
\[0\to \mathbf{Pic}_{X,\et}\to  \mathbf{Pic}_{X,v}\to H^0(X,\wtOm)\otimes \G_a \to 0\]
 of commutative rigid groups. Crucially, this sequence is not split outside of trivial cases.  By inspecting the proof in \cite[\S2]{heuer-diamantine-Picard}, we see that in the language of this article, the last map is precisely the Hitchin morphism $\wtcH$ for $n=1$, up to composition with $\cBun_{1,v}\to \mathbf{Pic}_{X,v}$.

The upshot of this discussion is that in the case of $G=\G_m$, the two Hitchin morphisms
over $\mathcal A_1= H^0(X,\wtOm_X)\otimes \G_a$
are both torsors under $\mathbf{Pic}_{X,\et}$, but one is split while the other is usually not. In particular, in contrast to the case of $G=\G_a$, this shows that in general
$\mathbf{Bun}_{n,v}\not\cong 	\mathbf{Higgs}_{n}$.
Returning to moduli $v$-stacks, this is what we believe might generalise: It seems plausible to us that for any reductive group $G$, the morphism $\wtcH$ exhibits $\cBun_{G,v}$ as a twist of $\cHiggs_G$ over $\mathcal A_G$.
To recover the $p$-adic Simpson correspondence from this geometric statement, one would have to see that the usual choices induce a trivialisation of this twist on $K$-points. This explains the motivation behind \cref{conj:p-adic-Simpson-global-continuity}.

\subsection{The rigid analytic representation variety}
Our final goal in this article is to obtain from the morphism of $v$-stacks $\wtcH$ a morphism of rigid analytic spaces by passing from $\cBun_{G,v}$ to the rigid analytic representation variety.

\begin{Definition}\label{d:cts-hom-over-Y}
Let $\Gamma$ be a profinite group. Let $G$ be a rigid group over $K$. We denote by \[\HOM(\Gamma,G)\]
 the inner Hom sheaf on $\Perf_{K,v}$ where $\Gamma$ is considered as a profinite $v$-sheaf and $G$ as a diamond. This sends a perfectoid space $Y$ to the $Y$-linear homomorphisms $\pi\times Y\to G$. Let us say that a homomorphism $\Gamma\to G(Y)$ is continuous if it comes from a morphism of adic spaces of this form. To justify this name, we observe that if $G\subseteq \GL_n$ is a linear algebraic group, then $\HOM(\Gamma,G)(Y)$ is the sheaf on $\Perf_K$ sending a perfectoid space $Y$ to
$\Hom_{\cts}(\Gamma,G(Y))$,
where $G(Y)$ is endowed with the subspace topology of $\GL_n(Y)\subseteq M_n(\O(Y))$. 
\end{Definition}

\begin{Proposition}\label{p:rep-of-HOM(Gamma-GLn)}
Let $\Gamma$ be any profinite group. Let $G$ be a rigid group over $K$. Then
\begin{enumerate}
\item The $v$-sheaf $\HOM(\Gamma,G)$ is a diamond. 
\item If $\Gamma$ is topologically finitely generated, then $\HOM(\Gamma,G)$ is represented by a semi-normal rigid space over $K$. We call this the \textbf{continuous representation variety} of $\Gamma$. 
\end{enumerate}
\end{Proposition}
\begin{proof}
	Let $\gamma:=(\gamma_i)_{i\in I}$ be a set of algebraic generators of $\Gamma$, not necessarily finite. Then we have an injection $\mathrm{ev}(\gamma):\HOM(\Gamma,G)\hookrightarrow \prod_IG$
of $v$-sheaves, where  $\mathrm{ev}(\gamma)$ denotes the evaluation map $\rho\mapsto (\rho(\gamma_i))_{i\in I}$.
By \cite[Lemma~11.22]{etale-cohomology-of-diamonds}, the right hand side is a diamond, thus so is the left hand side by \cite[Proposition~11.10]{etale-cohomology-of-diamonds}. This proves part 1.

For the proof of part 2, let $\Gamma_0$ be a dense subgroup of $\Gamma$ with generators $g=(g_1,\dots,g_r)$.

\paragraph{Step 1: Reduction to the case that $\Gamma=\wh{\Gamma}_0$ is the profinite completion of $\Gamma_0$.} By the Theorem of Nikolov--Segal \cite{NikolovSegal}, any finite index subgroup of $\Gamma$ is open, so profinite completion induces a surjective, open, continuous homomorphism
$\phi:\wh{\Gamma}_0\twoheadrightarrow \Gamma$.
We deduce that for any set of generators  $a=(a_i)_{i\in I}$ of $\ker \phi$, a continuous map from $\Gamma$ is the same as a continuous map from $\wh{\Gamma}_0$ that vanishes on the $a_i$. We thus have a left-exact sequence
\[ 0\to \HOM(\Gamma,G)\to \HOM(\wh{\Gamma}_0,G)\xrightarrow{\mathrm{ev}(a)} G^{I}.\]
If the Proposition holds for $\wh{\Gamma}_0$, this shows that $\HOM(\Gamma,G)$ is the kernel of a homomorphism from a rigid group to a product of adic groups, hence is itself represented by a rigid group. 

\paragraph{Step 2: Deformations of the trivial representation.} By \cref{c:open-subgroup-of-good-reduction} and \cref{p:inductive-lifting-ses-for-G}, there is an open subgroup $G^+\subseteq G$ of good reduction and a basis of neighbourhoods given by open subgroup $(G_n)_{n\in\N}$ of $G^+$  such that $G_n/G_{n+1}=\overline{\mg}^+_{t}:=\mg^+_0/\mg^+_t$ for some $t>0$.

\begin{Lemma}
Let $(R,R^+)$ be any perfectoid $K$-algebra.
Then any group homomorphism $\rho:\Gamma_0\to G_0(R)$ extends uniquely to a continuous morphism $\wh{\rho}:\Gamma\to G_0(R)$.
\end{Lemma}
\begin{proof}
By \cref{l:completeness-of-G}, it suffices to see that each
$\rho_n:\Gamma_0\to G_0(R)\to G_0/G_n(R)$ is continuous, i.e.\ has finite image. For this we argue by induction on $n$.

For $n=1$, the group $G_0/G_1(R)=\overline\mg^+_{t}(R)$ is an abelian torsion group. As $\Gamma_0$ is finitely generated, it follows that any map
$\varphi:\Gamma_0\to \overline\mg^+_{t}(R)$
has finite image. Thus $\ker\varphi$ is a finite index subgroup, which is open by the assumption $\Gamma=\wh{\Gamma}_0$. This gives the case of $\rho_1$.

For the induction step, we consider the short exact sequence obtained from \cref{p:inductive-lifting-ses-for-G}
\[ 0\to \overline\mg^+_{t}(R)\to G_0/G_{n+1}(R)\to  G_0/G_n(R)\to 1.\]
By induction hypothesis, the image $\im(\rho_{n})$ of $\Gamma_0$ on the right is finite, so $H:=\ker(\rho_n)$ is a finite index subgroup of $\Gamma_0$. In particular, $H$ is again finitely generated. The restriction of $\rho_{n+1}$ to $H$ factors through $\overline\mg^+_{t}(R)$, and thus has finite index kernel by the induction start. This shows that $\ker \rho_{m+1}$ contains a finite index subgroup, hence is open.
\end{proof}
Hence
$\HOM(\Gamma,G_0)=\HOM(\Gamma_0,G_0)$
is represented by a Zariski-closed subspace of the space $(G_0)^r$ parametrising the images of the $g_i$, cut out by the relations between the $g_i$.
\paragraph{Step 3: The general case.}
Consider now any finite index subgroup $H\subseteq \Gamma_0$ and let $F_H\subseteq \HOM(\Gamma,G)$ be the open subfunctor defined as the preimage of 
$\HOM(H,G_0)\subseteq \HOM(H,G)$
under the restriction map $\mathrm{res}_H:\HOM(\Gamma,G)\to \HOM(H,G)$.
Then by continuity, we have an equality of sheaves
$\HOM(\Gamma,G)=\varinjlim_{H<\Gamma_0}F_H$
where $H$ ranges through the finite index subgroups of $\Gamma_0$. It thus suffices to prove that each $F_H$ is representable by a semi-normal rigid space; the transition maps in the colimit are then clearly open immersions. For this let $s=(s_1,\dots,s_l)$ be a finite set of generators of $H$, and for each $k=1,\dots,l$ let $w_{s_k}(g)$ be the group-theoretic word expressing $s_k$ in terms of the generators $g=(g_1,\dots,g_r)$ of $\Gamma_0$. Let $E\subseteq 	G^{r}$ be the Zariski-closed subspace cut out by the relations between the $g_i$. Then
\[\begin{tikzcd}
E\arrow[r,hook] &G^{r} \arrow[r, "w_{s}"] &G^l \\
F_H \arrow[rr,"\mathrm{res}_H"] \arrow[u,"\mathrm{ev}(g)"] & &{\HOM(H,G_0)} \arrow[u,hook,"\mathrm{ev}(s)"']
\end{tikzcd}\]
is a Cartesian diagram.
It follows that $F_H$ is represented by the fibre product in rigid spaces. Let us denote the latter by $X_H$. It remains to pass from $X_H$ to its semi-normalisation:
\begin{Lemma}[{{\cite[Remark~3.7.3]{KedlayaLiu-II}\cite[p.78]{ScholzeBerkeleyLectureNotes}}}]\label{l:semi-normalisation}
The inclusion 
\[\{\text{semi-normal rigid spaces over }K\}\hookrightarrow \{\text{rigid spaces over }K\}\]  admits a right-adjoint, the semi-normalisation functor $-^{\mathrm{sn}}$. For any rigid space $X$, the co-unit map induces an isomorphism $X^{\mathrm{sn}\diamondsuit}\to X^\diamondsuit$.
\end{Lemma}
\begin{proof}
After replacing $X$ with its reduced subspaces $X^{\mathrm{red}}$, it suffices to construct this on reduced rigid spaces. 
Any reduced commutative ring $A$ has a semi-normalisation $A^{\mathrm{sn}}$ with the desired universal property by \cite[Theorem~4.1]{SwanSeminormality}, and the map $A\to A^{\mathrm{sn}}$ is subintegral, hence finite since $A$ is excellent. Thus $A^{\mathrm{sn}}$ is again of topologically finite type over $K$. This glues since seminormality is stable under rational localisation \cite[Proposition 3.7.2]{KedlayaLiu-II}.
\end{proof}
It follows that the semi-normal rigid space $X_{H}^{\mathrm{sn}}$ represents $\HOM(\pi,G)$.
\end{proof}

Returning to the $p$-adic Simpson correspondence, we wish to apply \cref{p:rep-of-HOM(Gamma-GLn)} to the \'etale fundamental group $\Gamma=\pi^{\et}_1(X)$ of a smooth proper rigid space $X$. 
In this situation, the algebraic conditions imposed on $\Gamma$ apply in decent generality:

\begin{Corollary}\label{c:rep-var-representable-alg-case}
Let $X$ be the rigid analytification of a smooth proper $K$-variety. Then the $v$-sheaf $\HOM(\pi^{\et}_1(X),G)$ is represented by a semi-normal rigid space.
\end{Corollary}
\begin{proof}
By the Lefschetz principle, we can assume that $X$ has a model $X_0$ over a subfield $K_0\subseteq K$  such that there is an isomorphism of fields $K_0\cong\C$. By invariance of $\pi^{\et}_1$ under extension of scalars,  $\pi^{\et}_1(X)$ is then the profinite completion of the topological fundamental group of the compact K\"ahler variety $X_0(\C)$, which is finitely generated.
\end{proof}

\begin{Question}
	Let $X$ be any (smooth) proper rigid space over $K$. Is the \'etale fundamental group $\pi^{\et}_1(X)$ topologically finitely generated? If yes, then \cref{c:rep-var-representable-alg-case} holds for such $X$.
\end{Question}

\subsection{The Hitchin morphism for representations}
Finally, we explain how  for the analytification $X$ of a connected smooth proper variety over a complete algebraically closed field $K$ over $\Q_p$ with fixed base point $x\in X(K)$, the Hitchin morphism $\wtcH$ attaches to any continuous representation
$\pi:=\pi^{\et}_1(X,x)\to \GL_n(K)$
a set of ``Hitchin--Hodge--Tate weights'' in a way that is compatible in families. The key to this is the following relation between $G$-representations of $\pi^{\et}_1(X,x)$ and $v$-$G$-bundles:

\begin{Definition}\label{def:hom-to-bun}
	Let $G$ be any rigid group.	There is a natural morphism of $v$-stacks 
	\[ \alpha:\HOM_{\cts}(\pi,G)\to \cBun_{G,v}\]
	defined as follows (cf.\  {\cite[Theorem~5.2]{heuer-v_lb_rigid}} for $G=\GL_n$): Let $\wt X\to X$ be the pro-finite-\'etale universal cover from \cref{ex:choice-of-cover}.2, which is a $\pi$-torsor in a natural way. For any $Y\in \Perf_K$, consider the projection $q:\wt X\times Y\to X\times Y$ as a $\pi\times Y$-torsor relatively over the base $Y$. Then $\alpha(Y)$ sends any continuous homomorphism $\rho:\pi\times Y\to G$ over $Y$ in the sense of \cref{d:cts-hom-over-Y} to the $G$-torsor on $(X\times Y)_v$  defined by pushout of $q$ along $\rho$.
\end{Definition}

\begin{Theorem}\label{c:props-of-H-on-reps}
 Let $G$ be a rigid group over $K$ satisfying the assumptions of \cref{l:Hitchin-base-is-rigid-space-if-G-red-or-com}, for example $G$ could be reductive, or commutative. Then $\wtcH\circ \alpha$  defines a canonical morphism of rigid analytic spaces
\[ \wt H:\HOM(\pi_1^{\et}(X,x),G)\to\cB_{G}\]
from the continuous $G$-representation variety of $\pi_1^\et(X,x)$  to the Hitchin base. Moreover:
\begin{enumerate}
\item $\wt H$ is functorial in $G$ and $X\to \Spa(K)$.
\item For any continuous representations $\rho_{1},\rho_2:\pi_1^\et(X,x)\to G(Y)$ over a perfectoid space $Y$ for which there is a finite index subgroup $\Gamma\subseteq  \pi_1^\et(X,x)$ such that the restrictions $\rho_{1|\Gamma}\sim \rho_{2|\Gamma}$ become conjugated over a $v$-cover $Y'\to Y$, we have $\wt H(\rho_{1})=\wt H(\rho_{2})$.
\item If $G$ is commutative, then $\wt H$ is a group homomorphism.
\item If $G=\GL_n$, then for any short exact sequence of representations $\rho_1\to \rho\to \rho_2$, we have $\wt H(\rho)=\wt H(\rho_1)\cdot \wt H(\rho_2)$ where $\cdot$ is the product on $\cB_n$ defined in \cref{r:product-structure-on-Hitchin-base}.
\end{enumerate}
\end{Theorem}
\begin{proof}
The composition $\wtcH\circ \alpha:\HOM(\pi_1^{\et}(X,x),G)\to\cB_{G}$ is a morphism of $v$-sheaves over $K$.
Both sides are represented by semi-normal rigid spaces by \cref{c:rep-var-representable-alg-case} and \cref{l:Hitchin-base-is-rigid-space-if-G-red-or-com}. As diamondification from semi-normal rigid spaces to $v$-sheaves over $K$ is fully faithful \cite[Theorem 8.2.3]{KedlayaLiu-II}\cite[Proposition 10.2.3]{etale-cohomology-of-diamonds}, this defines a unique morphism of rigid spaces.

Parts 1 and 4 are then immediate from \cref{p:props-of-HHTlog}. Part 2 follows from functoriality in $X$: After pullback to the finite \'etale cover $X'\to X$ corresponding to $\Gamma$, and any choice of lift $x'$ of $x$ to $X'$, the representations define the same element in $\HOM(\pi_1^{\et}(X',x'),G)$. Since the natural map $\mathcal A_{G,X}(R)\to \mathcal A_{G,X'}(R')$ is injective, it follows that $\wt H(\rho_{1})=\wt H(\rho_{2})$.

For part 3, it suffices by part 2 to prove this on an open subgroup $G_0\subseteq G$ on which $\exp$ is defined. As $K[\mathfrak g]^G=K[\mathfrak g]$, we have $\mathcal A_{G_0}=\mathcal A_G=\G_a\otimes\mathfrak g\otimes \wtOm$, and the result follows by functoriality from the morphism $\exp:G_0\to \mathfrak g\otimes \G_a$ and the case of $\G_a$.
\end{proof}

We believe that $\wt H$ encodes deep arithmetic information in a geometric way. Indeed, from known instances of the $p$-adic Simpson correspondence, one can extract the following:

\begin{Example}\label{ex:examples-of-Hitchin-for-reps}
\begin{enumerate}
\item For $G=\G_m$, it follows from \S\ref{s:Hitchin-for-Gm} that $\wt H$ is the homomorphism from \cite[Theorem~1.7]{heuer-diamantine-Picard}, hence fits into an exact sequence of rigid group varieties
\[ 0\to \mathbf{Pic}^{\tt}_{X,\et}\to \HOM_{\cts}(\pi,\G_m)\xrightarrow{\wt H } H^0(X,\wtOm)\otimes_K \G_a\to 0.\]
In particular, $\wt H$ is then a torsor under $\mathbf{Pic}^{\tt}_{X,\et}$, the topologically torsion Picard variety. For curves, the first map is a geometric incarnation of the Weil pairing of the Jacobian.
\item If a representation $\rho:\pi\to \GL_n(K)$ has abelian image, then $\wt H(\rho)$ can be described in a way that is closely related to the abeloid $p$-adic Simpson correspondence of \cite{HMW-abeloid-Simpson}: In this case, $K^n$ can be decomposed into simultaneous eigenspaces. Let $\lambda_i:\pi\to K^\times$ be the corresponding characters and set $a_i:=\wt H(\lambda_i)\in H^0(X,\wtOm)$, then it follows from \cref{p:props-of-HHTlog}.3 that $\wt H(\rho)$ is given by the coefficients of $(X-a_1)\cdots (X-a_n)$.
\item If $X$ has a model $X_0$ over a discretely valued subfield of $K$, then any representation $\rho:\pi\to \GL_n(\Q_p)$ associated to a $\Q_p$-local system $\mathbb L$ over $X_0$ lands in the fibre over $0$ by Galois equivariance in \cref{c:props-of-H-on-reps}.1, giving a different perspective on why the Higgs field associated to $\mathbb L$ by the $p$-adic Simpson functor of Liu--Zhu \cite{LiuZhu_RiemannHilbert} is nilpotent. 
\item Similarly, all the representations in the image of Deninger--Werner's functor \cite{DeningerWerner_Simpson} are also in the fibre over $0$, because the associated Higgs fields vanish.
This includes all representations with finite image, for which this also follows from \cref{c:props-of-H-on-reps}.2.
\end{enumerate}

\end{Example}

If $X$ is a smooth proper curve of genus $g\geq 2$, we think it is plausible that  \cref{ex:examples-of-Hitchin-for-reps}.1 generalises to $\GL_n$ for $n\geq 1$, namely that $\wt H$ could generically on $\mathcal A_G$ be a torsor under a certain rigid group related to the spectral curve. We will pursue this further in part II.
\subsection{Relation to the complex Corlette--Simpson correspondence}\label{s:rel-cpx}
In complex geometry, the Corlette--Simpson correspondence for a smooth projective variety $X$ over $\C$ is an equivalence of categories between finite dimensional complex representations of the topological fundamental group $\pi_1(X)$ on the one hand (``Betti side''), and semi-stable Higgs bundles on $X$ with vanishing Chern classes on the other (``Dolbeault side'') \cite{SimpsonModuliII}. It can be regarded as an analogue of the Hodge decomposition for non-abelian coefficients.
 
More generally, the Corlette--Simpson correspondence can be generalised to an equivalence of $G$-representations and $G$-Higgs bundles for any reductive group $G$ over $\C$ \cite{SimpsonCorrespondence}. However, we are not aware of an extension to general complex Lie groups. From this perspective, it is arguably somewhat surprising that our result holds in the given generality.  This is one reason why in the introduction we were cautious and imposed a condition that $G$ is reductive in the globalisation of \cref{intro-first-version-of-main-Thm} to proper $X$.

Regarding moduli spaces, the Corlette--Simpson correspondence for $\GL_n$ induces for any smooth projective variety $X$ over $\C$ a canonical and functorial homeomorphism
\begin{equation}\label{eq:complex-correspendence}
	\mathcal M_{\mathrm{B}}(\C)\isomarrow  \mathcal M_{\mathrm{Dol}}(\C)
\end{equation}
between the Betti moduli space of complex representations of the topological fundamental group $\pi_1(X)$ of dimension $n$ on the one hand, and the Dolbeault moduli space of semi-stable Higgs bundles on $X$ of rank $n$ with vanishing Chern classes on the other \cite{SimpsonModuliII}. However, even though both sides have natural complex analytic structures, the correspondence does not respect these: In fact, the analogue of $\wt H$ in this setting would be the composition 

	\[ 	\mathcal M_{\mathrm{B}}(\C)\isomarrow  \mathcal M_{\mathrm{Dol}}(\C)\xrightarrow{\mathcal H}\cB_n(\C)\]
	of the Corlette--Simpson correspondence with the Hitchin morphism. But this fails to be  complex analytic already for line bundles on curves \cite[Example on p21]{SimpsonCorrespondence}. In particular, the above homeomorphism cannot be upgraded to an isomorphism between moduli stacks.

From this perspective, we find it very surprising that $\wt{H}$ exists as a morphism of rigid analytic spaces: This seems to be a stronger statement than what is possible in the complex theory.  It thus appears that the correspondence over $K$ is less canonical as it requires the choice of lift of $X$, but on the other hand preserves more structure.

\subsection{Relation to Sen theory}\label{s:sen-theory}
In order to provide more context for the Hitchin--Hodge--Tate map, we now briefly sketch how the  above construction is related to Hodge--Tate--Sen weights:

Let $L$ be a $p$-adic local field and let $L_\infty|L$ be the extension obtained by adjoining all $p$-power roots of unity. Let $C$ be the completed  algebraic closure of $L$. Classical Sen theory associates to any semi-linear representation of the Galois group $G_L$ of $L$ on a finite dimensional $C$-vector space $W$ an $L_\infty$-vector space $E=D_{\mathrm{Sen}}(W)$ such that $E{\otimes}_{L_\infty}C=W$ together with an $L_\infty$-linear operator $\theta:E\to E$ (see \cite{Sen_cts_cohom}). We call $(E,\theta)$ a Sen module, and the Hodge--Tate--Sen weights of $W$ are defined as the generalised eigenvalues of $\theta$.

The construction of $(E,\theta)$ is very similar to the small $p$-adic Simpson correspondence, and one can turn this analogy into a more precise comparison: Consider the adic space $X=\Spa(L)$, then we can make a very similar construction for $X$ (the ``arithmetic setting'') as we did for smooth rigid spaces over $C$ in the last section (the ``geometric setting''):
 The pro-finite-\'etale cover $\wt X\to X$ of the last section is for $X=\Spa(L)$  simply the cover
$\Spa(C)\to \Spa(L)$. Like in \cref{def:hom-to-bun}, one obtains by descent along this cover an equivalence
\[\Big\{\begin{array}{@{}c@{}l}\text{semi-linear representations of $G_L$}\\\text{on fin.\ dim.\ $C$-vector spaces} \end{array}\Big\}\isomarrow \Big\{\text{$v$-vector bundles on $\Spa(L)$}\Big\}.\]
On the other hand, a result of Tate in Galois cohomology \cite[Theorem 1]{Tate_endomorphisms} can be interpreted as saying that for $\nu:X_{v}\to X_{\et}$, we have
$R^1\nu_{\ast}\O=\O$ and $R^2\nu_{\ast}\O=0$.
From the point of view taken in this article, it would therefore be sensible to define 
$\wtOm_{X}:=R^1\nu_{\ast}\O=\O$,
 so that according to \cref{d:differentials-on-smoothoids}, a Higgs bundle on $X_{\et}$ for this notion of ``differentials' would be a vector bundle $E$ on $X_{\et}$ together with an endomorphism of $E$. In particular, from  such a Higgs bundle, we get a Sen module by base-changing the global sections  to $L_\infty$.
 
 Going through the construction of the Sen module $D_{\mathrm{Sen}}(W)$, one now sees that one can interpret this as being obtained from a ``local $p$-adic Simpson correspondence on the space $X=\Spa(L)$'': The coefficients $L_\infty$ appear since one needs to go up the cyclotomic tower $L_\infty|L$ to make the $v$-vector bundle small, as with the toric tower in the geometric setting.

Carrying out the analogous constructions of the last section in the arithmetic setting now yields the map that sends a semilinear $G_L$-representation $W$ to the characteristic polynomial of its Sen operator $\theta$, which is an equivalent datum to the Hodge--Tate--Sen weights. This is the precise sense in which the Hitchin--Hodge--Tate morphism $\wt{H}$ is conceptually analogous in the geometric setting to sending a $G_L$-representation to its Hodge--Tate--Sen weights.
\appendix

\section{Recollections on non-abelian continuous cohomology}
In this appendix, we recall some generalities on non-abelian group cohomology sets.
\begin{Definition}\label{d:non-ab-1-cocycles}
Let $\Gamma$ and $A$ be not necessarily abelian topological groups, written multiplicatively. Assume that we have a continuous left-action of $\Gamma$ on $A$, i.e.\ a group homomorphism $\Gamma\to \Aut(A)$ such that the induced morphism $\Gamma\times A\to A$ is continuous. For $g\in \Gamma$, $a\in A$, we write $ga$ for its image under this map.  Then the set of continuous $1$-cocycles is
\[ \mathcal Z^1_{\cts}(\Gamma,A):=\{\text{continuous maps } c:\Gamma\to A\mid  c(gh)=c(g)\cdot gc(h)\}\]
The continuous group cohomology set $H^1_{\cts}(\Gamma,A):=\mathcal Z^1_{\cts}(\Gamma,A)/\!\sim$ is then defined as the quotient
by the following equivalence relation:
\[ f\sim f' \Leftrightarrow \exists a\in A:\quad  a^{-1}\cdot f(g)\cdot ga=f'(g) \text{ for all $g\in \Gamma$}.\]
For any continuous $1$-cocycle $c$, we denote its cohomology class by $[c]$.
\end{Definition}
\begin{Definition}\label{d:twisted-action}
Let $B$ be a topological group with continuous $\Gamma$-action and let $A\subseteq B$ be a normal subgroup that is preserved by the action of $\Gamma$. Let $b\in \mathcal Z^1_{\cts}(\Gamma,B)$ be any continuous $1$-cocycle. Then we denote by $_bA$ the  $\Gamma$-module with underlying group $A$ equipped with the $b$-twisted action, given for $g\in \Gamma$, $a\in A$ by
$ g\ast_b a:=b(g)\cdot ga\cdot b(g)^{-1}$,
where the right hand side is calculated inside $B$.
\end{Definition}
\begin{Proposition}[\cite{Serre-cohom-galoisienne}, \S5.5]\label{p:non-ab-les}
Let $\Gamma$ be a topological group, and  let 
$0\to A\to B\to C\to 0$
 be a short exact sequence of (not necessarily abelian) topological groups with continuous left-$\Gamma$-actions such that all maps in the sequence are $\Gamma$-equivariant. Then:
\begin{enumerate}
\item There is a short exact sequence of pointed sets
\[ 0\to A^\Gamma\to B^\Gamma\to C^\Gamma\to H^1_{\cts}(\Gamma,A)\to H^1_{\cts}(\Gamma,B)\to H^1_{\cts}(\Gamma,C).\]
\item Let $b\in \mathcal Z^1_{\cts}(\Gamma,B)$, then the subset of elements of $H^1_{\cts}(\Gamma,B)$ with the same image as $b$ in $H^1_{\cts}(\Gamma,C)$ is in natural bijection with the set of elements of $H^1_{\cts}(\Gamma,{}_bA)$.
	\item Assume that $A$ is abelian and let $c\in \mathcal Z_{\cts}^1(\Gamma,C)$, then there is a class 
	$\Delta(c)\in H^2_{\cts}(\Gamma,{_cA})$
	that vanishes if and only if $[c]$ is in the image of $H^1_{\cts}(\Gamma,B)\to H^1_{\cts}(\Gamma,C)$. The class $\Delta(c)$ is canonical and functorial both in $\Gamma$ and in the sequence.
\end{enumerate}
\end{Proposition}

One natural way that continuous group cohomology arises in practice is that it can be used to describe \cH-cohomology for Galois covers with group $\Gamma$ for non-abelian sheaves, by a non-abelian version of the Cartan--Leray spectral sequence:

\begin{Proposition}[{\cite[Proposition~2.8]{heuer-v_lb_rigid}}]\label{p:Cartan--Leray}
Let $q:X\to Y$ be a morphism of diamonds over $K$ that is Galois for the action of a locally profinite group $\Gamma$ on $X$. Let $\mathcal F$ be a sheaf of (not necessarily abelian) topological groups on $Y_{v}$ with the property that 	for $i=1,2$ we have
\[
\mathcal F(X\times \Gamma^i)=\Map_{\cts}(\Gamma^i,\mathcal F(X)),
\]
for example this holds if $\mathcal F$ is represented by a rigid group over $K$ that admits a locally closed embedding as a rigid space into some affine space $\A^n$ (e.g.\ any linear algebraic group, or rigid open subgroup thereof).
Then there is a left-exact sequence of pointed sets
\[0\to H^1_{\cts}(\Gamma,\mathcal F(X))\to H^1_v(Y,\mathcal F)\to H^1_v(X,\mathcal F)^\Gamma.\]
\end{Proposition}
\begin{proof}
	The only part that is not proved in the reference is that $\mathcal F$ satisfies the displayed condition if it embeds into $\A^n$.  It suffices to prove this for $i=1$ and $X$ in the basis of affinoid perfectoid spaces. By the known case of $\mathcal F=\O$, we then have a commutative diagram
\[\begin{tikzcd}
{\Map_{\cts}(\Gamma,\mathcal F(X))} \arrow[r,hook] & {\Map_{\cts}(\Gamma,\O^n(X))} \\
{\Mor_K(\Gamma\times X,\mathcal F)} \arrow[u, dotted] \arrow[r,hook] & {\Mor_K(\Gamma\times X,\mathbb A^n).} \arrow[u,"\cong"]
\end{tikzcd}\]
By evaluating at all points $g\in \Gamma$, we get the dotted arrow in the diagram: Here the continuity holds because the assumptions ensure that $\mathcal F(X)\subseteq \O^n(X)$ inherits the subspace topology.
Also due to the locally closed assumption, we can check on points whether a map $\Gamma\times X\to \mathbb A^n$ factors through $\mathcal F$. This shows that the morphism on the left  is surjective.
\end{proof}

For our applications, the Galois group is often $\cong\Z_p^d$. In this case, one  has the following description of continuous cohomology, see the proof of \cite[Lemma~5.5]{Scholze_p-adicHodgeForRigid}:
\begin{Lemma}\label{l:cts-grp-cohom-of-Zp}
	Let $\Delta=\Z_p^d$ be topologically generated by elements $\gamma_1,\dots,\gamma_d$. Let $R$ be a $p$-adically complete $\Z_p$-algebra and let $M$ be a $p$-adically complete $R$-module with a continuous $R$-linear $\Delta$-action. Then $R\Gamma_{\cts}(\Delta,M)$ is computed by the complex
	\[\mathcal C_{\mathrm{grp}}^*:=\big[M\to M\otimes_RR^d\to  M\otimes_R\wedge^2R^d\to \dots \to M\otimes_R\wedge^dR^d\big]\]
	whose $k$-th transition map is given by
	$m\otimes e_{i_1}\wedge \dots  \wedge e_{i_k}\mapsto \sum_{j=1}^d(\gamma_j-1)m \otimes  e_{i_1}\wedge \dots  \wedge e_{i_k}\wedge e_{j}$.
\end{Lemma}

\end{document}